\documentclass{article}
\usepackage[utf8]{inputenc}
\usepackage{pdfpages}
\usepackage{url}
\usepackage{euler} 
\usepackage[toc,page]{appendix}
\usepackage{amssymb, amsmath, amsthm}
\usepackage[Conny]{fncychap}
\usepackage{graphicx}
\usepackage{adjustbox}
\usepackage{comment}
\usepackage{mathrsfs}

\usepackage[T1]{fontenc}
\usepackage{scalerel}
\usepackage{calligra}
\usepackage{dirtytalk}

\usepackage{mathtools}
\usepackage{bbm}
\usepackage[english]{babel}

\usepackage[scaled=0.85]{beramono}%% mono
\usepackage{tikz-cd}
\usepackage{float}
\usepackage{fancyhdr}

\usepackage[left=0cm,right=0cm,top=2cm,bottom=2cm,margin=3cm]{geometry}
\usepackage{amsthm}
\usepackage{bigints}
\usepackage{tikz}
\usetikzlibrary{matrix,arrows,decorations.pathmorphing,decorations.markings, cd, scopes, backgrounds}
\usepackage[linktoc=all]{hyperref}
\usepackage[capitalise]{cleveref}
\hypersetup{colorlinks=false,pdfborder={0 0 0}}
\usetikzlibrary{matrix,arrows,decorations.pathmorphing}
\usetikzlibrary{patterns,calc,angles,quotes}
\usepackage{tikz-cd}
\usepackage[utf8]{inputenc}
\tikzset{commutative diagrams/.cd}
\newtheorem{theorem}{Theorem}
\numberwithin{theorem}{section} 
\newtheorem{corollary}[theorem]{Corollary}

\newtheorem{lemma}[theorem]{Lemma}

\newtheorem{proposition}[theorem]{Proposition}

\theoremstyle{definition}
\newtheorem{definition}[theorem]{Definition}
 
\newtheorem{claim}[theorem]{Claim}

\newtheorem{example}[theorem]{Example}

\newtheorem{remark}[theorem]{Remark}

\newtheorem{construction}[theorem]{Construction}

\newtheorem{notation}[theorem]{Notation}

\numberwithin{doubts}{section} 
\newcommand{\E}{\mathcal{E}}
\newcommand{\D}{\mathcal{D}}

\newcommand{\K}{\mathcal{K}}

\renewcommand{\P}{\mathbf{P}}

\newcommand{\op}[1]{\operatorname{#1}}

\newcommand{\C}{\mathcal{C}}

\newcommand{\Schf}{\op{Sch}_{fd}}

\newcommand{\F}{\mathcal{F}}

\renewcommand{\H}{\mathcal{H}}

\renewcommand{\O}{\mathcal{O}}

\newcommand{\A}{\mathbb{A}}
\newcommand{\Ca}{\mathcal{C}}

\newcommand{\X}{\mathcal{X}}

\newcommand{\bb}{\bullet}
\newcommand{\Hext}{\mathcal{H}_{\op{ext}}}

\numberwithin{subsection}{section}
\newcommand{\Sp}{\op{Spec}}

\newcommand{\sset}{\op{Set}_{\Delta}}

\newcommand{\CrrCopEal}{\op{Corr}(\Ca)^{\otimes}_{\E,\op{all}}}
\newcommand{\CrrCpEal}{\op{Corr}(\Ca)_{\E,\op{all}}}
\newcommand{\CrrCovopEal}{\op{Corr}^{\op{\E-cart}}(\op{Cov}(\Ca))^{\otimes}_{\tilde{\E},\op{all}}}
\newcommand{\CrrCovEopEal}{\op{Corr}^{\op{all-cart}}(\op{Cov}_{\E}(\Ca))^{\otimes}_{\tilde{\E},\op{all}}}

\newcommand{\SHe}{\mathcal{SH}_{\op{ext}}}

\newcommand{\Shext}{{\mathcal{SH}}^{\otimes}_{\op{ext}}}

\newcommand{\Shcl}{\mathcal{SH}^{\otimes}_{\op{cl}}}
\newcommand{\Shc}{\mathcal{SH}_{\op{cl}}}
\newcommand{\Hcl}{\mathcal{H}_{\op{cl}}}
\newcommand{\Sho}{\mathcal{SH}^{\otimes}}
\newcommand{\SH}{\mathcal{SH}}
\newcommand{\Y}{\mathcal{Y}}
\newcommand{\Ccc}{(\Ca^{\op{op}})^{\coprod,\op{op}}}

\makeatletter
\tikzset{
	open/.code     = {\tikzset{right hook->, circled};},
	closed/.code   = {\tikzset{right hook->, slashed};},
	open'/.code    = {\tikzset{left hook->, circled};},
	closed'/.code  = {\tikzset{left hook->, slashed};},
	circled/.code  = {\tikzset{markwith = {\draw (0,0) circle (.375ex);}};},
	slashed/.code  = {\tikzset{markwith = {\draw[-] (-.4ex,-.4ex) -- (.4ex,.4ex);}};},
	markwith/.code ={
		\pgfutil@ifundefined%
		{tikz@library@decorations.markings@loaded}%
		{\pgfutil@packageerror{tikz}{You need to say %
				\string\usetikzlibrary{decorations.markings} to use arrows with markings}{}}{}%
		\pgfkeysalso{/tikz/postaction = {
				/tikz/decorate,
				/tikz/decoration={markings, mark = at position 0.5 with {#1}}}
		}
	},
}
\makeatother

\DeclareMathOperator{\spec}{Spec}

\newcommand{\sk}{\mathbbm{k}}

\newcommand{\sseq}{\subseteq}

\newcommand{\mr}{\mathrm}
\newcommand{\mc}{\mathcal}
\newcommand{\mb}{\mathbb}
\newcommand{\mbbm}{\mathbbm}
\newcommand{\mf}{\mathfrak}

\newcommand{\into}{\hookrightarrow}

\DeclareMathAlphabet{\mathpzc}{OT1}{pzc}{m}{en}

\DeclarePairedDelimiter{\abs}{\lvert}{\rvert}

\newcommand{\oocatname}[1]{\scaleobj{1.25}{\mathpzc{#1}}}

\newcommand{\set}[1]{ \left \{ #1 \right \} }

\newcommand{\bigslant}[2]{
	\mathchoice
	{% \displaystyle
		{\raisebox{.2em}{$#1$}\left/\raisebox{-.2em}{$#2$}\right.}%
	}
	{% \textstyle
		#1\!\;\!/\!\,#2
	}
	{% \scriptstyle
		#1\!\;\!/\!\,#2
	}
	{% \scriptscriptstyle  
		#1\!\;\!/\!\,#2
	}
}

\newcommand{\Th}[2]{\mr{Th}_{#1}\left( #2 \right)}
\newcommand{\colim}{\op{colim}}

\newcommand{\epf}{{}_!}
\newcommand{\epfs}{{}_{\#}}
\renewcommand{\sseq}{\subseteq}
\newcommand{\AstNl}{\oocatname{ASt}^{\leq 1, NL} }

\newcommand{\iMap}{\underline{\mr{Map}}}

\newcommand{\restrict}[2]{{% we make the whole thing an ordinary symbol
		\left.\kern-\nulldelimiterspace % automatically resize the bar with \right
		#1 % the function
		\vphantom{\big|} % pretend it's a little taller at normal size
		\right|_{#2} % this is the delimiter
}}

\newtheorem{innercustomgeneric}{\customgenericname}
\providecommand{\customgenericname}{}
\newcommand{\newcustomtheorem}[2]{%
	\newenvironment{#1}[1]
	{%
		\renewcommand\customgenericname{#2}%
		\renewcommand\theinnercustomgeneric{##1}%
		\innercustomgeneric
	}
	{\endinnercustomgeneric}
}

\newcustomtheorem{customthm}{Theorem}
\newcustomtheorem{customlemma}{Lemma}

\fancypagestyle{main}{\fancyhf{}
	\fancyhead[LE,RO]{\scriptsize Non-rep six-fun formalisms.}
	\fancyhead[RE,LO]{\scriptsize \rightmark }
	\fancyfoot[CE,CO]{\thepage}}

\title{Non-representable six-functor formalisms.}
\author{Chirantan Chowdhury \and Alessandro D'Angelo }
\date{}

\begin{document}
	
	\maketitle{}
	\begin{abstract}
	In this article, we study the properties of motivic homotopy category $\mathcal{SH}_{\op{ext}}(\X)$ developed by Chowdhury and Khan-Ravi for $\X$ a Nis-loc Stack. In particular, we compare the above construction with Voevodsky's original construction using NisLoc topology. Using the techniques developed by Liu-Zheng and Mann's notion of $\infty$-category of correspondences and abstract six-functor formalisms, we also extend the exceptional functors and extend properties like projection formula, base change and purity to the non-representable situation. 	
	\end{abstract}
	
	\tableofcontents
 \section{Introduction}

Since its birth more than twenty years ago, motivic homotopy theory has progressed tremendously and motivic techniques have found numerous applications in different and broad areas of geometry and arithmetics. Grothendieck vision for a framework of the "Six Operations" found fertile land in $\A^1$-homotopy and these six operations turned out to be one of the key strength features of motivic homotopy theory. The motivic six functors provide a robust framework to tackle homological and cohomological questions and are flexible enough to allow us to actually make some concrete computations. Thanks to the work of Ayoub, Cisinski, Déglise, Hoyois and many others, the theory of motivic six functors on schemes is now well established. A natural follow up question is: can we extend this formalism to a larger class of geometric objects like algebraic stacks? The program of extending and studying the motivic homotopy category for algebraic stacks has started a few years ago for example in \cite{Hoyois_Equiv_Six_Op}, \cite{Khan-Ravi_Generalised_Coh_Stacks} and \cite{Chowdhury}, but many interesting questions are still open. The mere existence of $\SH(\mc X)$, for $\mc X$ an algebraic stack,  introduces a wide range of new cohomology theories on stacks, such as motivic cohomology, Hermitian K-theory, and many others.  To explore these new invariants and address questions arising from the motivic homotopy theory of stacks, we aim to gain a deeper understanding of the \textit{Borel}\footnote{As opposed to the \textit{genuine} $\SH$ studied in \cite{Hoyois_Equiv_Six_Op} and in \cite{Khan-Ravi_Generalised_Coh_Stacks} for scalloped stacks.} category $\SH_{\op{ext}}(-)$ of \cite{Chowdhury} and \cite{Khan-Ravi_Generalised_Coh_Stacks}. This is the primary goal of this paper.\\

There are two (equivalent) approaches in the literature for defining and constructing the \textit{Borel} category $\SH_{\op{ext}}(-)$ for stacks: the lisse-extended theory and the NisLoc-extended theory (cf. \cite{Khan-Ravi_Generalised_Coh_Stacks, Chowdhury}). But in the presence of a topology like the NisLoc topology, one can emulate Voevodsky’s original construction by considering $\A^1$-invariant, NisLoc sheaves in spectra, forming a new category called $\SH_{cl}(-)$. Remarkably, this leads to the same construction:

	\begin{customthm}{1}[Comparison \Cref{Sec.5:_cl_and_ext_SH_pointwise_equivalence}]
	We have a natural map:
	\begin{equation*}
		\Shcl(\X) \xrightarrow{} \Shext(\X).
	\end{equation*}
	\noindent and this map is an equivalence.
\end{customthm}
\noindent This new, more \textit{classical} version of $\SH(\mc X)$ facilitates the adaptation of schematic arguments, making it easier to extend results known for schemes to algebraic stacks. Let us note that in a forthcoming paper \cite{Neeraj-Felix-SHcl}, a similar construction was independently carried out. As mentioned in \cref{NFvsACconstruction}, both constructions yield the same category.\\

As a final result, we will also prove ambidexterity and purity for (possibly non-representable) smooth maps of algebraic stacks in \Cref{Thm:_stacky_Ambidexterity} and \Cref{Sec.4:_Stacky_Purity}:

\begin{customthm}{2}[Non-Representable Purity]
	Let $f: \mc X \longrightarrow \mc Y$ be a smooth map between NL-stacks. Then:
	\[ \varphi_f: f\epfs \overset{}{\longrightarrow} f_!\Sigma^{\mb L_f} \]
	\noindent is an equivalence.
\end{customthm}

The plan of the paper is as follows: in \Cref{Preliminaries} we compare the NisLoc topology with the Smooth Nis topology of Pirisi, and we review the main features of $\SH_{ext}$ constructed in \cite{Chowdhury} by the first named author; in \Cref{Comp_Thm}  we prove the main comparison theorem between the \textit{classical} definition of $\SH$ and the extended one; in \Cref{Exc_Funct}, we revise and extend the construction of the exceptional functor from \cite{Chowdhury}; in \Cref{Amb_and_Purity} we deal with the ambidexterity and purity statements for non-representable maps of algebraic stacks; finally in \Cref{Applications} we provide the reader with some easy extensions and applications of our constructions.  \\

 To make the paper as self-contained as possible, we provide the reader with appendices containing the technical material needed to deal with the $\infty$-category of correspondences, that were heavily used in \Cref{Exc_Funct}. Although some results about correspondences are already present in the literature, we provide proofs of relevant lemmas and propositions for the sake of completeness. We preferred to be overzealous rather than omit important details in our constructions.\\
 
 Finally, let us point out that we restricted ourselves to classical algebraic 1-stacks just for simplicity of exposition. More details on how to extend the theory to higher derived stacks are provided in \Cref{Applications}.
 %But experts should find it easy to generalize each argument presented here and, through induction, extend the theory from algebraic $(n)$-stacks to algebraic $(n+1)$-stacks. Moreover, the theory implicitly works for derived (algebraic) stacks too: one can just define the motivic homotopy category on their underived truncation. This is because the extended theory will be insensitive to the derived structure, for the same exact reason why for a derived scheme $X$, with classical truncation $X_0$, we have that $\SH(X)\simeq\SH(X_0)$ as proved in \cite{Adeel_PhD}.

 \subsection*{Acknowledgements:}
A major part of the project was carried while the first author was a PostDoc at University of Duisburg Essen. He would like to thank M.Levine, D. Aranha  for helpful discussions and conversations. \footnote{\thanks{C. Chowdhury was supported by the ERC through the project QUADAG.  This paper is part of a project that has received funding from the European Research Council (ERC) under the European Union's Horizon 2020 research and innovation programme (grant agreement No. 832833) \includegraphics[scale=0.08]{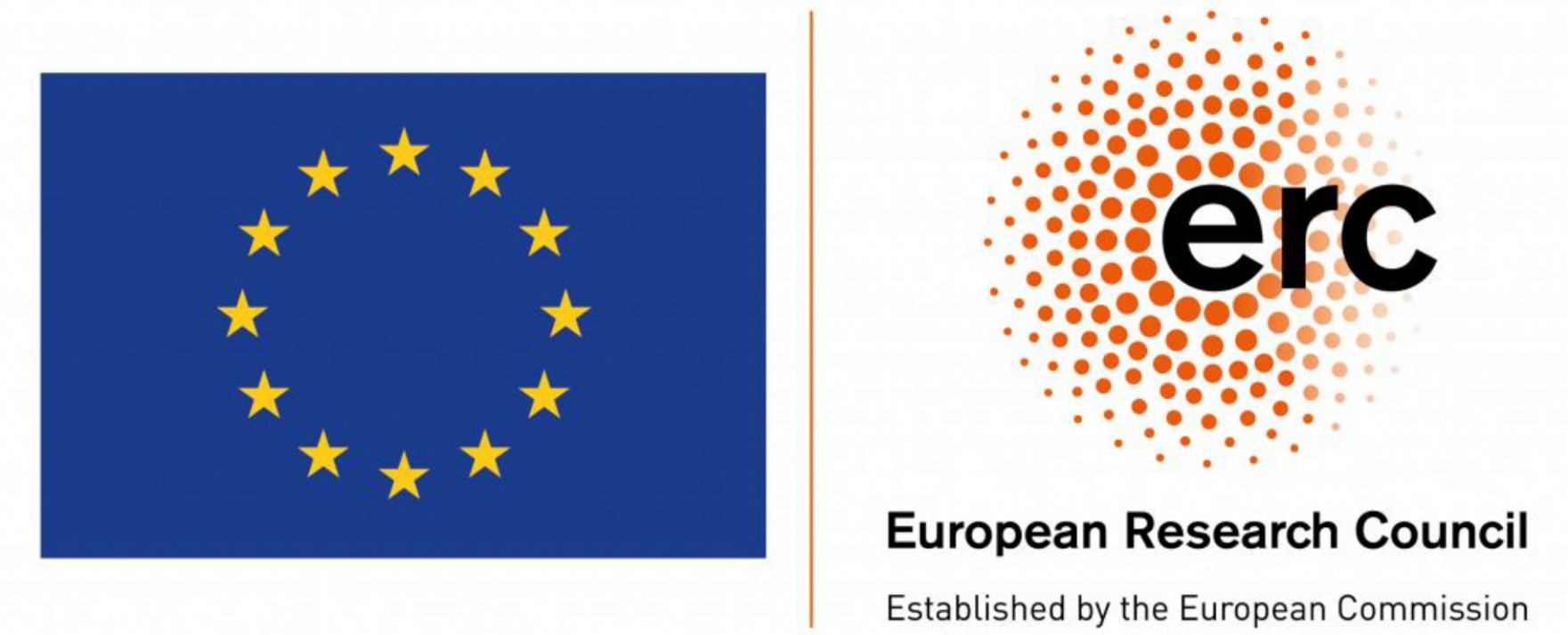}}}. C. Chowdhury also acknowledges support (through Timo Richarz) by the European Research Council (ERC) under Horizon Europe (grant agreement nº 101040935), by the Deutsche Forschungsgemeinschaft (DFG, German Research Foundation) TRR 326 \textit{Geometry and Arithmetic of Uniformized Structures}, project number 444845124 and the LOEWE professorship in Algebra, project number LOEWE/4b//519/05/01.002(0004)/87 \\
 The second named author was supported by the Göran Gustafsson Foundation for Research in Natural Sciences and Medicine. He would also like to thank J. Ayoub, M. Gallauer, J. Hekking, M. Levine, M. Pernice and D. Rydh for helpful conversations about the topics presented in this paper.\\
Both of the authors would also like to thank N. Deshmukh and F. Sefzig for sharing their results in \cite{Neeraj-Felix-SHcl}. 
	\subsection*{Conventions and Notations}
	We will use the acronyms \textit{qs, qcqs, lft, lci, gci} for \say{quasi-separated},\say{quasi-compact, quasi-separated}, \say{locally of finite type}, \say{locally of complete intersection}, \say{globally of complete intersection} respectively. Unless differently specified by \textit{(algebraic) stack} we mean a qcqs algebraic $ 1 $-stack, as opposed to higher or even derived stacks.  We will denote the category of (underived) 1-algebraic stacks as $\oocatname{ASt}^{\leq 1}$\\
     We will freely use the language of $\infty$-categories as presented in \cite{HTT,HA}, in particular this means that we will use quasi-categories as our model for $\infty$-categories.\\
	%We need to warn the reader that we have changed the notation a bit from \cite{Chowdhury}. While \textit{Nis-Loc} is phonetically pleasant to hear, and we still stick to it in the spoken language, we decided to adopt the shorter acronym $NL$ to avoid confusion when talking about Nisnevich maps of schemes.

\section{Preliminaries}\label{Preliminaries}	
	\subsection{Comparison of various notions of Nisnevich topologies on algebraic stacks.}
	
	In \cite{Pirisi}, Pirisi introduced the \textit{smooth Nisnevich} topology in order to study cohomological invariants of a large class of algebraic stacks. In some sense, that was already a first avatar of the existence of a motivic homotopy category for stacks. Indeed, a cohomological invariant of an algebraic stack is a functorial way to assign elements of a given Rost's cycle module to points on the stack itself (cf. \cite{Di_Lorenzo_Pirisi_Coh_Inv_Root}). On the other hand, D\'{e}glise proved in his thesis \cite{Deglise2003_Thesis} that Rost's cycle modules over a field $\sk$ give a presentation of the heart, with respect to the homotopy $t$-structure, of Voevodsky's mixed category of motives $\mr{DM}(\sk; \mb Z)$.
	
	In this subsection, we will compare  Pirisi's topology with the topology introduced in \cite{Chowdhury}. Let us remark that all the material present in this section is already present in the literature, but for completeness we will collect it here for the reader's convenience. Moreover, notice that we are going to slightly change the conventions in \cite{Chowdhury}: what was there denoted as \textit{Nisnevich} atlas, stack, etc.., here it will be denoted as \textit{NL}. Since we are setting all the new notation from the start, there should not be any cause of confusion.
	
	\begin{definition}\cite[Definition 2.1.4]{Chowdhury}
		\begin{enumerate}
			\item [$ (i) $]	Let $ \oocatname{ASt}^{\leq 1, NL} $ be the $\infty$-category of algebraic (1)-stacks (also known as Artin stacks) for which there exists an atlas which admits Nisnevich-local sections. In other words, the category consists of algebraic stacks $\X$ for which there exists an atlas $ x: X \to \X$ such that for any test scheme $T$ with a morphism $ t: T \to \X$, the base change morphism $x': X \times_{\X} T \to T$ admits sections Nisnevich-locally. 
			\item [$ (ii) $]			 A morphism of algebraic stacks $f : \X \to \Y$ in $\oocatname{ASt}^{\leq 1, NL}$ is said to admit \textit{Nisnevich-local sections} if there exists an atlas $ y : Y \to \Y$ admitting Nisnevich-local sections and a morphism $s : Y \to \X$ such $ f\circ s = y$. We will denote the set of these maps as $ \mr{NL} $.
			\item [$(iii)$] The topology generated by NL-maps will be denoted as $\tau_{NL}$ (or simply NL if it is clear from the context that we are talking about the topology)and it will be referred to as the \textit{Nisnevich Local} topology. Covers for $\tau_{NL}$ will be called NL-covers.
		\end{enumerate}

	\end{definition}
	
	\begin{remark}
		\begin{enumerate}
			\item At the level of schemes and quasi-separated algebraic spaces, the NL-topology agrees with the usual Nisnevich-topology.
			\item One important example of NL-stacks is given by local quotient stacks. In particular, for any affine algebraic group $G$ acting on a scheme $X$, the morphism \[ X':=X \times^G \op{GL}_n \to [X \times^G \op{GL}_n/ \op{GL}_n] \cong [X/G] \] is a $\op{GL}_n$-torsor which admits sections Zariski (hence Nisnevich)-locally. The algebraic space $X'$ is already a NL-atlas, but we can further reduce to a schematic atlas choosing a Nisnevich cover $X'' \to X'$.
			\item If $ f: \X \to \Y$, then $f$ admits Nisnevich-local sections if and only if for every $ T \to \Y$ where $T$ is an algebraic space, the morphism of algebraic spaces $f' : \X \times_{\Y} T \to T$ admits sections Nisnevich-locally. 
		\end{enumerate}
	\end{remark}
	
	We now recall the definition of \textit{smooth-Nisnevich} topology from \cite{Pirisi}.
	
	\begin{definition}\cite[Definition 3.2]{Pirisi}
		Let $f: \X \to \Y$ be a smooth representable morphism of algebraic stacks.
		\begin{enumerate}
			\item [$ (i) $]		 Let $ p : \Sp K \to \Y$ be a morphism where $K$ is a field. Then $f$ is \textit{smooth-Nisnevich neighbourhood}  of $p$ if there exists a representative $p' : \Sp K' \to \X$ such that $ f \circ p'= p$.
			
			\item [$ (ii) $]	 If for every $ p \in \abs{\mc Y} $, the morphism $f$ is a smooth-Nisnevich neighbourhood at $p$, then we say that $ f $ is a \textit{smooth-Nisnevich map}. The topology generated by smooth-Nisnevich maps will be called the smooth-Nisnevich topology.
			\item [$ (iii) $]	 We say that a family of smooth representable morphisms $ \set{u_{\alpha}: U_{\alpha} \rightarrow \mc Y} $ is a \textit{smooth-Nisnevich atlas} of $Y$ if for every field $K$ and for every $p : \Sp K \to \mc Y$, there exists $u_{\alpha}$ that is a smooth-Nisnevich neighbourhood of $p$.
		\end{enumerate}
	%	\noindent We will often denote smooth-Nisnevich maps/covers/etc. as simply \textit{sNis} maps/covers/etc.

	\end{definition}
	
	\noindent Before comparing the different notions of \textit{Nisnevich} maps, let us briefly recall an important property of \text{smooth-Nisnevich covers}:
	
	\begin{proposition}\label{smoothnisquassep}
		Let $\X$ be an algebraic stack. Then there exists a scheme $x: X \to \X$ where $X$ is a scheme such that $x$ is a \text{smooth-Nisnevich} atlas.
	\end{proposition}
	\begin{remark}
		\Cref{smoothnisquassep} was also independently proved by J. Hall (private communication). Let us point out that the result proved by J. Hall is much more refined and shows that with mild assumptions one can actually find an affine smooth-Nisnevich cover.
	\end{remark}

\Cref{smoothnisquassep} appeared for the first time in \cite[Proposition 3.6]{Pirisi} and \cite[Theorem 1.2]{deshmukh2023motivichomotopytypealgebraic}, where it was stated for algebraic stacks with a separated diagonal, because it relied on \cite[Th\'{e}or\`{e}me  6.3, Th\'{e}or\`{e}me  6.5]{Laumon2000} where the hypothesis was assumed. However, the hypothesis on the diagonal can actually be removed. We are grateful to M. Pernice for pointing out that our original argument, which relied on the diagonal being quasi-separated, is in fact applicable in greater generality. With minor modifications of the proof given in \cite[Theorem 4.2.15]{Alper_Book}, one can show:
	
	\begin{theorem}[Alper, {\cite[Thm 6.3]{Laumon2000}}]\label{S2:_Alper_LMB}
			Let $\mc X \in \oocatname{ASt}_{\bigslant{}{S}} $ be an algebraic stack over some base scheme $S$. Let $x \in \mc X(\sk)$ be a $\sk$-valued point, then there exists a smooth morphism $W \rightarrow \mc X$ from a scheme $W$ and a point $w\in W(\sk)$ over $x$.
		\end{theorem}
	\begin{proof}
			Let $p: U \rightarrow \mc X$ be a smooth atlas with $U$ a (disjoint) union of affine schemes. Consider:
				\begin{center}
					\begin{tikzpicture}[baseline={(0,0)}, scale=1.5]
							\node (a) at (0,1) {$ U_x $};
							\node (b) at (1, 1) {$  U $};
							\node (c)  at (0,0) {$ \spec(\sk) $};
							\node (d) at (1,0) {$\mc X$};
							
							\node (e) at (0.25,0.75) {$\ulcorner $};
							
							\path[font=\scriptsize,>= angle 90]
							
							(a) edge [->] node [above ] {$  $} (b)
							(a) edge [->] node [left] {$  $} (c)
							(b) edge [->] node [right] {$  $} (d)
							(c) edge [->] node [below] {$ x $} (d);
							
						\end{tikzpicture}
				\end{center}
		We will divide the proof in two: first we will prove the claim for $\X$ an algebraic space (giving a variant of \cite[6.3]{Knutson_AlgSp}); then we will generalise the claim to any algebraic stack.\\

		 Suppose $\X$ is an algebraic space and $p$ is étale. Then by \cite[\href{https://stacks.math.columbia.edu/tag/056U}{Tag 056U}]{stacks-project}, there exists a closed point $ u: \spec(\kappa(u))\into U_x $ such that $\bigslant{\kappa(u)}{\sk}$ is a finite separable extension of some degree d.\\
	
	       Now let us recall from \cite[\S 5]{Rydh_Etale_Devissage} the construction of the stack $ \mr{\Acute{E}T}^d\!\!\left(U/\mc X\right) $ (in the notation of \textit{loc. cit.} this would actually be $\mr{\Acute{E}T}^d_{\emptyset}\!\left(U/\mc X\right)$). Given a representable morphism $ U\rightarrow \mc X $, we denote by $ \left(U/\mc X\right)^d $ the \textit{dth} fiber product of $U$ over $\mc X$. The fiber product $\left(U/\mc X\right)^d$ comes equipped with an action of the symmetric group $\mf{S}_d$ permuting its factors, therefore we can consider:
	       \[ \mr{\Acute{E}T}^d\!\!\left(U/\mc X\right):=\left[\bigslant{\left(U/\mc X\right)^d}{\mf S_d}\right] \]
	        The stack $\mr{\Acute{E}T}^d\!\!\left(U/\mc X\right)$ classifies finite étale morphisms $ W\rightarrow \mc X $ of rank $d$, together with a map $ W\rightarrow U $ (cf. \cite[Remark 5.2]{Rydh_Etale_Devissage}).\\
	 
			\noindent Now the closed point $u\in U_x(\kappa(u))$, we found before, fits into a diagram:
			\begin{center}
					\begin{tikzpicture}[baseline={(0,0)}, scale=1.5]
							\node (a) at (0,1) {$ \kappa(u) $};
							\node (b) at (1, 1) {$  U $};
							\node (c)  at (0,0) {$ \spec(\sk) $};
							\node (d) at (1,0) {$\mc X$};
							
							\node (e) at (0.25,0.75) {$ $};
							
							\path[font=\scriptsize,>= angle 90]
							
							(a) edge [->] node [above ] {$  $} (b)
							(a) edge [->] node [left] {$  $} (c)
							(b) edge [->] node [right] {$  $} (d)
							(c) edge [->] node [below] {$ x $} (d);
							
						\end{tikzpicture}
				\end{center}
			
			\noindent that gives us a $\sk$-point $v \in \mr{\Acute{E}T}^d\!\!\left(U/\mc X\right)$. Choosing a faithful representation $\mf S_d \sseq GL_n$, we can rewrite $\mr{\Acute{E}T}^d\!\!\left(U/\mc X\right)\simeq \left[\bigslant{V_d(\bigslant{U}{\mc X})}{GL_n}\right]$, where $V_d(\bigslant{U}{\mc X}):= (\bigslant{U}{\mc X})^d\times^{\mf S_d} GL_n $. Since $GL_n$ is special, the point $v \in \mr{\Acute{E}T}^d\!\!\left(U/\mc X\right)$ admits a lifting to $V_d(\bigslant{U}{\mc X})$:
			\begin{center}
					\begin{tikzpicture}[baseline={(0,0)}, scale=1.5]
							\node (a) at (1,0.7) {$ V_d(\bigslant{U }{\mc X})$};
							\node (b) at (1, 0) {$  \mc X  $};
							\node (c)  at (0,0) {$\spec(\sk) $};

							\path[font=\scriptsize,>= angle 90]
							
							(a) edge [->] node [right ] {$  $} (b)
							(c) edge [->] node [below] {$ x $} (b)
							(c) edge [dashed, ->] node [left] {$  $} (a);
							
						\end{tikzpicture}
				\end{center}
					Since $\X$ is an algebraic space, its diagonal is a monomorphism and in particular it will be quasi-separated. The fact that $\Delta_{\X}$ is quasi-separated  and $U$ is a union of affines (hence it is quasi-separated over $S$), implies that $U \rightarrow \mc X$ is quasi-separated too by \cite[\href{https://stacks.math.columbia.edu/tag/050M}{Tag 050M}]{stacks-project}. Since $U\rightarrow \X$ and $U$ are quasi-separated, we get that $ (U/\X)^d $ is itself quasi-separated: this follows from the fact that $(U/\X)^d\rightarrow U^d$ is a pullback of $\X \rightarrow (\X/S)^d$. Then using \cite[Theorem 6.4]{Knutson_AlgSp}, without loss of generality, we can suppose that $W:=V_d(\bigslant{U}{\mc X})$ is represented by a scheme and we get our claim.\\
				
				For a general algebraic stack $\X$ the proof is quite similar. In this case, $p:U\rightarrow \X$ is just a smooth representable map, hence the fiber product $U_x$ will be an algebraic space, smooth over $\sk$. Using what we proved in the first part, we can suppose that $U_x$ is a smooth scheme over $\sk$. Again by \cite[\href{https://stacks.math.columbia.edu/tag/056U}{Tag 056U}]{stacks-project}, we can find a closed point $ u: \spec(\kappa(u))\into U_x $ such that $\bigslant{\kappa(u)}{\sk}$ is a finite separable extension of some degree d. By the same arguments used in the first part, from $\kappa(u)$ we get a map $\spec(\sk)\rightarrow V_d(U/\X)$ lifting the point $x$. Now $V_d(U/\X)$ is just an algebraic space, but using the claim from the first part of the proof, we can suppose that $V_d(U/\X)$ is actually a scheme and we are done.
				
		\end{proof}
	Now we can go back to the proof of \cref{smoothnisquassep}:
	\begin{proof}[Proof of Proposition \ref{smoothnisquassep}]
			Following the notations in the proof of \cref{S2:_Alper_LMB}, given a smooth atlas $ X\rightarrow \mc X  $, we can construct a smooth-Nisnevich cover considering $\bigcup_{n \in \mb N} V_n(X/\mc X)\rightarrow \mc X$.
		\end{proof}
	\begin{remark}
		With minor modifications of \cite[Theorem 4.3.1]{Alper_Book}, one can adapt our proof of \Cref{S2:_Alper_LMB} to actually get for Deligne-Mumford stacks an \'etale map that is also a smooth-Nisnevich cover, i.e. an \'etale cover such that we can lift all field-valued points. This was also independently proved by J. Hall (private communication) via similar, but different techniques.
	\end{remark}
	
		We would like to compare the notions of Nisnevich-local section and smooth-Nisnevich covers. Firstly, we need a result on smooth schemes over local Henselian rings.	

	\begin{proposition}\cite[Section 2.3,Proposition 5]{NeronModels}\label{henselizationsmoothlifts}
		Let $R$ be a local Henselian ring with residue field $k$. Let $X$ be a smooth $R$-scheme.  Then canonical map \[X(R) \to X(k) \] from the set of $R$-valued points to set of $k$-valued points is surjective.
	\end{proposition}

	\begin{remark}
		The above proposition follows from the geometric characterization of Henselian rings, namely an scheme which is \'etale  over an henselian ring is a local isomorphism at all points over the ring with trivial residue field extension (\cite[Section 2.3, Definition 1]{NeronModels}). As smooth morphisms admit sections \'etale locally, we get the following proposition.
	\end{remark}

	The following was independently proved also in \cite{deshmukh2023motivichomotopytypealgebraic}:
	\begin{proposition}\label{sec2:_Sm_Nis_iff_LS_Nis}
		Let $f : \X \to \Y$ be a smooth representable morphism of algebraic stacks. Then $f$ admits sections Nisnevich-locally if and only if $f$ is smooth Nisnevich cover.
		% In particular, every quasi-separated algebraic stack admits an atlas which admits Nisnevich-local sections.
	\end{proposition}
	
	\begin{proof}
		Suppose $f$ admits Nisnevich-local sections. Let $p : \Sp K \to \Y$ be a $K$-point, then the fiber product $f': \X \times_{\Y} \Sp K \to \Sp K$ admits sections Nisnevich locally. As Nisnevich covers over a field are given by a disjoint union of separable field extensions of $k$ where at least one of them is the field itself, the morphism $f'$ admits a section. This gives us the morphism $p': \Sp K \to \X$ such that $ f \circ p'= p$. \\
		
		Suppose $f$ is a smooth-Nisnevich cover. We need to show $f$ admits Nisnevich-local sections. As $f$ is representable; this reduces down to showing the case that if $f$ is a smooth-Nisnevich morphism of schemes, then $f$ admits Nisnevich-local sections.  Let $ y \in Y$ and let $R =\O_{Y,y}^h$ be the Henselization of the local ring with the residue field $k$. Let $x$ be the closed point of the field $R$. Let $p : \Sp R \to Y$ be the morphism induced by the Henselization. \\
		As $f$ can lift $k$-points, \ref{henselizationsmoothlifts} implies that $f': X':= X \times_{Y} \Sp R \to \Sp R$ admits a section. This gives a morphism $p': \Sp R \to X$ such that $f \circ p'=p$. As $R$ is the local ring at point $y$ in the Nisnevich topology (cf. \cite[Theorem 4.5]{MilneEtaleCoh}), there exists a morphism $q_y:V_y \to Y$ which is completely decomposed at point $y$ and a morphism $q'_y:V_y \to X$ such that $f \circ q'_y = q_y$. Varying the argument over all points $y \in Y$, we get that a Nisnevich cover $q: V \to Y$ and a morphism $q': V \to X$ such that $f \circ q' =q$.
		\end{proof}
	
	\begin{corollary}\label{NL-ASt=ASt}
		Let $  \X \in \oocatname{ASt} $ be an algebraic stack. Then it has a smooth atlas admitting Nisnevich local sections, hence $  \X \in \oocatname{ASt}^{\leq 1, NL} $.
	\end{corollary}
	
	\begin{proof}
		Follows directly from the previous proposition and from Proposition \ref{smoothnisquassep}.
	\end{proof}

	We have a notion of smooth-Nisnevich maps only for representable morphisms, but we can always consider the following definition in the non-representable case:
	
	\begin{definition}
		A map $ f: \mc X \longrightarrow \mc Y $ admits smooth-Nisnevich local sections if there exists a smooth Nisnevich atlas $ y: Y \longrightarrow \mc Y $ and a section:
		\begin{center}
			\begin{tikzpicture}[baseline={(0,0)}, scale=1.5]
				\node (a) at (1,0.7) {$ \mc X $};
				\node (b) at (1, 0) {$  \mc Y $};
				\node (c)  at (0,0) {$ Y $};

				\path[font=\scriptsize,>= angle 90]
				
				(a) edge [->] node [right ] {$ f $} (b)
				(c) edge [->] node [below] {$ y $} (b)
				(c) edge [dashed, ->] node [left] {$ s $} (a);
				
			\end{tikzpicture}
		\end{center}
		\noindent such that $ f\circ s=y $. We will denote the set of maps admitting smooth-Nisnevich local section with $ \mr{LS}_{P} $, where the $ P $ stands for Pirisi.
	\end{definition}
	But we can just talk about $ \mr{LS}_{Nis} $ maps since we already implicitly proved the following:
	\begin{corollary}
		Let $ f: \mc X \longrightarrow \mc Y $ a map between algebraic stacks. Then $ f \in \mr {LS}_{P} $ if and only if $ f \in \mr{LS}_{Nis} $. In particular a map $ f: X \longrightarrow Y $ between algebraic spaces (or schemes) admits Nisnevich local sections if and only if it admits smooth Nisnevich sections, without any smoothness assumption on  $ f $ itself.
	\end{corollary}
	\begin{proof}
		It is straightforward from the definitions and the statement we proved in  \cref{sec2:_Sm_Nis_iff_LS_Nis} about smooth representable maps. 
	\end{proof}

%	Since we \ed to quasi-separated algebraic stacks $ \mc X $ only, then every point $ x \in \abs{\mc X} $ is algebraic, meaning that there exists the \textit{residual gerbe} $ \mc G_x \into  \mc X $ by \cite[Appendix B]{Rydh_Etale_Devissage} and it can be realised as an inverse limit similar to the representable case by \cite[Lemma 2.1]{Addendumpushoutrydh}. In \textit{loc. cit.} we also have the following:

	\subsection{Review of construction of $\Shext(-)$ and four functors.}
 
    In this subsection, we briefly recall the construction of $\Shext$ and mention the six functor formalism which has been constructed in \cite{Chowdhury}.
	
	In \cite{robalo2013noncommutative}, Robalo explains that the construction of motivic stable homotopy theory (\cite{Morel-Voevodsky}) can be viewed as a functor taking values in $\infty$-categories.\\
	
	Let $\Schf$ denote the category of Noetherian schemes with finite Krull dimension. Robalo uses the construction of $\SH$ to define a functor 
	\begin{equation}
		\Sho(-): \Schf^{op}\to \op{CAlg}(\op{Pr}^L_{stb}). ~~~ \cite[\text{Section 9.1}]{Robalothesis}
	\end{equation}
	Let us unravel the information contained in this functor. As the classical $\op{SH}$ admits a symmetric monoidal structure, the $\infty$-category $\Sho(S)$ is a symmetric monoidal $\infty$-category, i.e. it comes equipped with a coCartesian fibration (\cite[Definition 2.4.2.1]{HTT}) $p_S:\Sho(S) \to N(\op{Fin}_*)$  such that $\Sho(S)_{\langle n \rangle} \cong \Sho(S)^{\times n}_{\langle 1 \rangle}$. Denote $\SH(S):= \Sho(S)_{\langle 1 \rangle}$. The coCartesian fibration encodes the symmetric monoidal structure of the $\infty$- category $\SH(S)$ in a coherent way.\\
	
	The $\infty$-category $\SH(S)$ is also presentable as it arises from localization of presheaves of smooth schemes over $S$ (\cite[Theorem 5.5.1.1]{HTT}). It is also a stable $\infty$-category (\cite[Definition 1.1.1.9]{HA}) as it inherits the triangulated structure of $\op{SH}$ constructed in \cite{Morel-Voevodsky}. As pullback morphism for a morphism $f$ in $\Schf$ is colimit preserving, the $\infty$-category $\Sho(S)$ lands in the $\infty$-category of presentable stable symmetric monoidal $\infty$-categories which is denoted by $\op{CAlg}(\op{Pr}^L_{stb})$. The  $\infty$-category $\SH(S)$  is the called the stable motivic homotopy category of $S$. \\
	
	As explained in \cite[Remark 9.3.1]{Robalothesis}, the stable motivic homotopy theory can be extended to all schemes.
	The functor $\Sho(-)$ is a Nisnevich sheaf (\cite[Proposition 6.24]{Hoyois_Equiv_Six_Op}). 
	\begin{proposition}\label{Shstckdef}\cite[Theorem 2.2.1]{Chowdhury}
		The functor $\Sho(-)$ extends to an $\infty$-sheaf \[\Shext(-): {\oocatname{ASt}^{\leq 1, NL}}^{\op{op}} \to \op{CAlg}(\op{Pr}^L_{stb}). \] \\
		
		Moreover, for any algebraic stack $\X \in \oocatname{ASt}^{\leq 1, NL}$ that admits a schematic atlas \linebreak $x: X \to \X$, one has
		
		\begin{equation}
			\Shext(\X) \cong \op{lim}\Big(\begin{tikzcd}
				\Sho(X) \arrow[r, shift right =2] \arrow[r,shift left =2] & \Sho(X \times_{\X} X) \arrow[l,dotted] \arrow [r,shift right =4] \arrow[r] \arrow[r,shift left =4] & \cdots \arrow[l, shift left =2 ,dotted] \arrow[l, shift right =2, dotted]
			\end{tikzcd} \Big)
		\end{equation}
		where the limit is over the \v{C}ech nerve of $x$.

	\end{proposition} 
	
	\begin{remark}
		In \cite[Section 12]{Khan-Ravi_Generalised_Coh_Stacks}, Khan and Ravi introduced the notion of limit-extended functor $\mathbf{SH}_{\triangleleft}(-)$ for derived algebraic stacks.
		For any $\X \in \op{Nis-locSt}$, let $\op{Lis_{\X}}$ be the $\infty$-category of pairs $(t,T)$ where $T$ is a scheme and $t: T \to \X$. The functor $\mathbf{SH}_{\triangleleft}(\X)$ ((\cite[Construction 12.1]{Khan-Ravi_Generalised_Coh_Stacks}) is defined as 
		\begin{equation}
			\mathbf{SH}_{\triangleleft}(\X):= \varprojlim_{(t,T) \in \op{Lis}_{\X}}\Sho(T).
		\end{equation}
		The canonical functor $\mathbf{SH}_{\triangleleft}(\X) \to \Shext(\X)$  is an equivalence when $\X \in \op{Nis-locSt}$ (\cite[Corollary 2.2.3]{Chowdhury}).
	\end{remark}
	\begin{remark}
		As already noticed in \cite{Chowdhury}, the NL topology is equivalent to the Nisnevich topology, when we restrict to (qcqs) schemes. Applying \cite[Lemma C.3]{Hoyois_Quad_Lefschetz} for $\mc C=Sch$ and $\mc D=\oocatname{ASt}^{\leq 1}$, equipped with the quasi-topologies given by Nisnevich and NL maps respectively, we get that a presheaf on $\oocatname{ASt}^{\leq 1}$ is a NL-sheaf if and only if it is a Kan extension of a Nisnevich sheaf on schemes.
	\end{remark}
	\begin{notation}
		\begin{enumerate}
			\item 	Let $f: \X \to \Y$ be a morphism in $\oocatname{ASt}^{\leq 1, NL}$. We denote the \textit{pullback functor} $\Shext(f): \Shext(\Y) \to \Shext(\X)$ by $f^{*\otimes} $. We shall also write $f^*: \SHe(\Y) \to 
			\SHe(\X)$ as the functor $f^{*\otimes}$ on the level of underlying $\infty$-categories. As $f^*$ is a colimit preserving functor, the adjoint functor theorem ( \cite[Corollay 5.5.2.9]{HTT}) says there exists a right adjoint  \[f_*: \SHe(X) \to \SHe(Y) \] which we call the \textit{pushforward functor}. 
			\item  	As $\Shext(\X)$ is a symmetric monoidal $\infty$-category, we shall denote the functor induced by the symmetric monoidal structure by \[ -\otimes -: \SHe(\X) \times \SHe(\X) \to \SHe(\X). \]
			\item As $\Shext(\X)$ is closed (\cite[Remark 1.5.3]{liu2017enhanced}), we can define the following:	for any objects $\mb E,\mb E' \in \SHe(\X)$, let $\iMap_{\SHe(\X)}(\mb E,\mb E')$  be the internal mapping space. 
          %\item  The context of exceptional functors will be discussed in Section 5.
		\end{enumerate}
	\end{notation}

      \section{Comparison Theorem and Exceptional Functors}	\label{Comp_Thm_and_Exc_Funct}
	\subsection{Construction of classical stable homotopy category and comparison with $\Shext(-)$.}\label{Comp_Thm}

	By \cite[Proposition 4.8]{Hoyois_Equiv_Six_Op}, we have a Nisnevich sheaf:
	\[ \mc H^{\otimes}:  \Schf^{op}\to \op{CAlg}(\op{Pr}^L_{}) \]
	\noindent that informally sends a scheme $X$ to $\mc H^{\otimes}(X)$ and sends $f: X \rightarrow Y$ to $f^*: \mc H^{\otimes}(Y) \rightarrow \mc H^{\otimes}(X) $. 
	
	Similarly to \cite[Theorem 2.2.1]{Chowdhury}, the functor $\mc H^{\otimes}$ can be extended into an $\infty$-sheaf of the form:
	 \[\mc H^{\otimes}_{ext}: \oocatname{ASt}^{\leq 1} \to \op{CAlg}(\op{Pr}^L_{}). \] 
	 \begin{definition}
	 	For $\mc X \in \oocatname{ASt}^{\leq 1}$, we define the (extended) unstable motivic category of $\mc X$ to be:
	 	\[ \mc H^{\otimes}_{ext}(\mc X) \]
	 \end{definition}
	\begin{remark}
			For $\mc X \in \oocatname{ASt}^{\leq 1}$, given any schematic atlas  $x: X \to \X$, one has
		
		\begin{equation}
			\Hext^{\otimes}(\X) \cong \op{lim}\Big(\begin{tikzcd}
				\mc H^{\otimes}(X) \arrow[r, shift right =2] \arrow[r,shift left =2] & 	\mc H^{\otimes}(X \times_{\X} X) \arrow[l,dotted] \arrow [r,shift right =4] \arrow[r] \arrow[r,shift left =4] & \cdots \arrow[l, shift left =2 ,dotted] \arrow[l, shift right =2, dotted]
			\end{tikzcd} \Big)
		\end{equation}
		where the limit is over the \v{C}ech nerve of $x$.
	\end{remark}

	But since we have at our disposal the $NL$-topology for our stacks, we can also think of mimicking the classical construction of motivic spaces and stable motivic  homotopy categories for $NL$-algebraic stacks. 
%	\begin{definition}
%		\begin{enumerate}
%			\item   Let $\X \in \op{Nis-locSt}$, let $\op{Sm}^{ft}_{\X}$ be the $(2,1)-$category of  schemes \textcolor{red}{\textbf{AlgStakcs?! I'd say: ft $\mc X$-alg stacks, or at the very least ft, smooth, representable maps over the base}} $X$ which admits a smooth finite type morphism to $\X$.
%			\item A presheaf $\F \in \Pp(\op{Sm}^{ft}_{\X})$ is \textit{$\A^1$-invariant} if for every object $X \in \op{Sm}^{ft}_{\X}$, the morphism $\F(X) \to \F(X \times \A^1)$ is an equivalence.
%			\item For every $\X \in \op{Nis-locSt}$, the \textit{classical unstable homotopy category of $\X$} be the $\infty$-category of $\A^1$-invariant presheaves on $\op{Sm}^{ft}_{\X}$ which satisfy $\op{Nis-loc}$ descent. We shall denote it by $\Hcl(\X)$. In symbols, it  is written as follows: 
%			\[ \Hcl(\X):= \op{Sh}_{\op{NL,\mathbb{A}^1}}(\op{Sm}^{ft}_{\X}). \]
%		\end{enumerate}
%		We will denote by $ L_{\A^1} $ and $L_{NL}$ the localization functors with respect to $\A^1$-equivalences and $NL$-local equivalences. We will also denote by $L_{mot}$ the motivic localization induced by $L_{\A^1}$ and $L_{NL}$.
%	\end{definition}

		 \begin{definition}\label{Sec.5:_stacky_def}
			\begin{enumerate}
				\item   Let $\X \in \oocatname{ASt}^{\leq 1} $ and let $\op{Lis}^{ft}_{\X}$ be the category of  smooth $ \X $-algebraic stacks of finite type. 
				\item A presheaf $\F \in \op{PSh}(\op{Lis}^{ft}_{\X})$ is \textit{$\A^1$-invariant} if for every object $X \in \op{Lis}^{ft}_{\X}$, the morphism $\F(X) \to \F(X \times \A^1)$ is an equivalence.
				\item For every $\X \in  \oocatname{ASt}^{\leq 1} $, the \textit{classical unstable homotopy category of $\X$} be the $\infty$-category of $\A^1$-invariant presheaves on $\op{Lis}^{ft}_{\X}$ which satisfy NL-descent. We will denote this category as:
				\[ \Hcl(\X):= \op{Sh}_{\op{NL,\mathbb{A}^1}}(\op{Lis}^{ft}_{\X}). \]
			\end{enumerate}
			\end{definition}
   
			\begin{notation}
				We will indicate by $ L_{\A^1} $ and $L_{NL}$ the localization functors with respect to $\A^1$-equivalences and $NL$-local equivalences. We will also denote by $L_{mot}$ the motivic localization induced by $L_{\A^1}$ and $L_{NL}$ together. We will say that a map $f$ in $\op{PSh}\left( \op{Lis}_{\bigslant{}{\X}} \right)  $ is a NL-, $\A^1$-, motivic equivalence if $ L_{NL}f, L_{\A^1}f $ or $L_{mot}f$ are equivalences respectively.
			\end{notation}
   
        \begin{remark}
            Notice that $\Hcl(\X)$ is an accessible localisation of the $\infty$-category of presheaves $\op{PSh}\left(\op{Lis}^{ft}_{\X}\right)$, in particular $\Hcl(\X)$ is presentable.
        \end{remark}
%	\textcolor{purple}{\begin{remark}[\textbf{Care needed!}]
%			\begin{enumerate}
%				\item A: it's not clear why in AGV the motivic localization is just L-A1 composed once with L-Nis, while in Hoyois is a colimit over N of this composition. Any idea? Might be relevant for us too.
%				\item A: See rmk 2.1.9 in AGV about presentability and locally smallness!!
%			\end{enumerate}
%		\end{remark}}
	Let us justify why the previous definition deserves to be \textit{classical} in some sense:
	\begin{proposition}\label{Sec.5:_classical_H_schemes}
		Let $ X $ be a scheme, then $\Hcl(X)$ is exactly the unstable motivic homotopy category of Morel-Voevodsky.
	\end{proposition}
	\begin{proof}
		Let $ \op{Sm}_{X}^{ft} $ be the category of smooth $X$-schemes of finite type. Since both $\Hcl$ and Morel-Voevodsky categories are defined as the subcategory of $\A^1$-invariant presheaves inside $ \op{Sh}_{NL}\left( \op{Lis}_{X}^{ft} \right) $ and $ \op{Sh}_{Nis}(\op{Sm}_{X}^{ft})$ respectively, for our claim it is enough to show that:
		\begin{equation}\label{Sec.5:_eq_topoi}
			 \op{Sh}_{NL}\left( \op{Lis}_{X}^{ft} \right)\simeq \op{Sh}_{Nis}(\op{Sm}_{X}^{ft})
		\end{equation}
		Notice that the topology on $\op{Lis}_{X} $ generated by smooth, representable NL-maps (that we will denote SRNL for short) coincides with the NL-topology by definition of NL-maps. So in \eqref{Sec.5:_eq_topoi}, we can replace the left hand side with $\op{Sh}_{SRNL}\left( \op{Lis}_{X}^{ft} \right)$. Now \eqref{Sec.5:_eq_topoi} follows from \cite[Lemma C.3]{Hoyois_Quad_Lefschetz}, where we choose $\mathcal{C}:= \op{Sm}^{ft}_{X} $ equipped with the quasi-topology (cf. \textit{loc. cit.}) induced by Nisnevich maps and $ \mc D:=\op{Lis}_{X}^{ft} $ equipped with the quasi-topology coming from smooth, representable NL-maps. The only non-trivial statement, we need to verify in order to apply  \cite[Lemma C.3]{Hoyois_Quad_Lefschetz}, is point (b) in \textit{loc. cit.}: this simply follows from the fact that the NL-topology on schemes coincides with the Nisnevich topology. And this finishes our proof.
	\end{proof}
%	\begin{remark}
%		For a given NL-stack $\X$, one could have opted to construct $\Hcl(\X) $ starting from sheaves in the SRNL topology (i.e. the one generated by smooth, representable, NL maps). If instead of taking the whole $\op{Lis}_{\X}$ category, containing non-representable maps, we had taken only the subcategory $\op{LisRep}_{\X}$ spanned by  smooth and representable maps $f: \Y \rightarrow \X $, we would have recovered the same \textit{genuine} unstable motivic category constructed in \cite{khan2021generalized}, at least for $\X$ scalloped. Indeed, this follows from a similar argument as the one used in the proof of \cref{Sec.5:_classical_H_schemes}. We will not undertake this path here. We will always work with the \textit{Borel} (or \textit{extended}) motivic homotopy category as in \cite{Chowdhury}, although it will be very interesting to further understand the interplay between the genuine and the extended motivic category.\textcolor{red}{I think we can safely remove this remark: it's good to know for our own sake, but it only creates confusion for the reader I think. I tend to write this things to avoid forgetting about it, but for a proper paper I'd remove it.}
%	\end{remark}
	\begin{remark}\label{Sec.5:_classical_H_alg_sp}
		Now let $X$ be an algebraic space and denote by $\H^{MV}(X)$ the $\infty$-category given by $\A^1$-invariant, Nisnevich sheaves on the category $ \op{ASp}_{X}^{sm,ft} $ of smooth, finite type, $X$-algebraic spaces. By \cite[Chap.2, Theorem 6.3]{Knutson_AlgSp}, any (qs) algebraic space admits a Nisnevich cover given by a scheme, hence the Nisnevich topology coincide with the NL topology on $\op{ASp}_{X}^{sm,ft} $. The same argument of the proof of \cref{Sec.5:_classical_H_schemes}, replacing $\op{Sm}_X$ with $\op{ASp}_{X}^{sm,ft}$, tells us that $\Hcl(X) =\H^{MV}(X) $. 
	\end{remark}
		From now on, we will stick with the following notation:
	\begin{notation}
		For a scheme (resp. algebraic space) $X$, we will simply denote $\Hcl(X)$ (resp. $\H^{MV}(X)$) as $\H(X)$ since there is no cause of confusion by \cref{Sec.5:_classical_H_schemes} (resp. \cref{Sec.5:_classical_H_alg_sp}).
	\end{notation}
	
	Following the same arguments as in \cite[Theorem 2.30]{robalo2013noncommutative} (or similarly \cite[\S 7]{Drew-Gallauer}), it is possible to construct a functor $ \Hcl^{*}: \oocatname{ASt}^{\leq 1} \longrightarrow Pr^L$ informally given as:
	\[ \begin{array}{cccc}
		\Hcl^{*}: & \oocatname{ASt}^{\leq 1} & \longrightarrow & Pr^L\\
		& \X & \mapsto & \Hcl(X) \\
		& f: \X \rightarrow \Y & \mapsto & f^*_{cl}
	\end{array} \]
	\begin{remark}
		Notice that to check NL-descent for presheaves with values in any $\infty$-category $\oocatname{D}$ with finite products, it is enough to just check \v{C}ech descent by \cite[App. A, Corollary 2]{Pstragowski_Synth} (or \cite[Remark 2.3.3(3)]{AGV_Rigid6FF}).
	\end{remark}
	
	We will denote by $f_{*,cl}$ the right adjoint to $f_{cl}^*$. If $f: \mc X \rightarrow \mc Y$ is smooth, then it is not hard to see that $\Hcl^{\otimes, *}$ actually lands in $Pr^{LR}$, i.e. that $f_{cl}^*$ admits a left adjoint $f_{\#, cl}$ (cf. \cite[\S 4]{Hoyois_Equiv_Six_Op}). On representables, $f^*_{cl}$ sends $ \mc Z \rightarrow \mc Y \in \op{Lis}^{ft}_{\mc Y} $ to $\mc Z\times_{\mc Y} \mc X \rightarrow \mc X \in \op{Lis}^{ft}_{\mc X} $. Similarly, for $f$ smooth,  $f_{\#,cl}$ sends $\mc T \rightarrow \mc X $ to $ \mc T \rightarrow \mc Y $ under the forgetful map $ \op{Lis}_{\mc X}^{ft} \rightarrow \op{Lis}_{\mc Y}^{ft} $ induced by $f$.
	
%	\begin{remark}[Functorial $(-)_{\#}$]
%		\textcolor{purple}{A: should we use that f* for schemes lands in PrLR and hence its extension does to? this would give us a functorial statement for lower sharp. (this was your idea Ciruzzo!)}
%	\end{remark}
	
	\begin{proposition}[Smooth Base Change]\label{Sec.5:_smooth_cl_BC}
		Suppose we are given the following cartesian square in $\oocatname{ASt}^{\leq 1}$:
			\begin{center}
			\begin{tikzpicture}[baseline={(0,0)}, scale=1.5]
				\node (a) at (0,1) {$ \mc W $};
				\node (b) at (1, 1) {$  \mc Z $};
				\node (c)  at (0,0) {$ \mc X $};
				\node (d) at (1,0) {$\mc Y$};
				
				\node (e) at (0.25,0.75) {$\ulcorner $};
				\node (f) at (0.5,0.5) {$\Delta $};
				
				\path[font=\scriptsize,>= angle 90]
				
				(a) edge [->] node [above ] {$ g $} (b)
				(a) edge [->] node [left] {$ q $} (c)
				(b) edge [->] node [right] {$ p $} (d)
				(c) edge [->] node [below] {$ f $} (d);
				
			\end{tikzpicture}
		\end{center}
		
		\noindent where $p,q$ are smooth. Then the natural exchange transformations:
		\[ Ex^*_{\#}(\Delta)_{cl}: q_{\#,cl}g^*_{cl} \stackrel{\eta^*_{\#}(p)}{\longrightarrow} q_{\#,cl}g^*_{cl}p^*_{cl}p_{\#,cl}= q_{\#,cl}q^*_{cl}f^*_{cl}p_{\#,cl}\stackrel{\varepsilon^*_{\#}(q)}{\longrightarrow} f^*_{cl}p_{\#,cl} \]
		\[ Ex^*_*(\Delta)_{cl}: p^*_{cl}f_{*,cl} \stackrel{\eta^*_*(q)}{\longrightarrow} p^*_{cl}f_{*,cl}q_{*,cl}q^*_{cl}=p^*_{cl}p_{*,cl}g_{*,cl}q^*_{cl} \stackrel{\varepsilon^*_*(p)}{\longrightarrow} g_{*,cl}q^*_{cl} \]
		\noindent are equivalences.
	\end{proposition}
	\begin{proof}
		The two exchange transformations are obtained from one another passing to the respective adjoints under the equivalence $(Pr^L)^{op}\simeq Pr^R$ (cf. \cite[Corollary 5.5.3.4]{HTT}), therefore it is enough to show that $Ex^*_{\#}(\Delta)_{cl}$ is an equivalence. The natural transformation  $Ex^*_{\#}(\Delta)_{cl}$ is obtained by the analogous transformation at the level of presheaves under the motivic localisation $L_{mot}$. If $Ex^*_{\#}(\Delta)_{cl}$ is an equivalence on presheaves then we are done, but to check this we can reduce to check it on representable presheaves, where it becomes clear. Indeed, consider $\mc T \rightarrow \mc Z$ smooth over $\mc Z$, then:
		
		\[ q_{\#,cl}, g^*\left( \mc T\rightarrow \mc Z \right)= \mc T\times_{\mc X} \mc W \rightarrow \mc X \]
		\[ f^*_{cl} p_{\#,cl}\left( \mc T\rightarrow \mc Z \right)= \mc T\times_{\mc Y} \mc X \rightarrow \mc X \]
		\noindent The exchange transformation in this case is induced by the transitivity isomorphism of fiber products $  T \times_{\mc X} \mc W  = \mc T \times_{\mc X} \left(\mc X \times_{\mc Y} \mc Z\right) = \mc T\times_{\mc Y} \mc Z $, and this proves our claim by \cite[Corollary 2.2.2]{Land_introductionQC}.
		%\textcolor{purple}{A: this looks like a object-wise argument, but it actually enough as stated above. It is the classical argument, that was implicity used by Hoyois Equiv 6 Op and more explicity used by \cite[1.4.32]{Ayoub2015-ml}.}
	\end{proof}
	The following proposition is formal and already known (cf. \cite[Proposition 4.5]{Hoyois_Equiv_Six_Op}), but we include a proof here for completeness:
\begin{proposition}[NL Separation]\label{Sec.5:_NL_separation}
	Let $\mc X\in \oocatname{ASt}^{\leq 1}$ and let $ \set{f_i: U_i \rightarrow \mc X}_{i \in I} $ be a NL-cover. Then:
	\begin{enumerate}
		\item The family of functors $\set{f_{i,cl}^*:  PSh(\op{Lis}^{ft}_{\X})\rightarrow PSh(\op{Lis}^{ft}_{U_i})}$ detects NL and motivic equivalences.
		\item The family of functors $ \set{f_{i,cl}^*:\Hcl(\mc X) \rightarrow \Hcl(U_i) } $ is conservative.
	\end{enumerate}
\end{proposition}
\begin{proof}
	We only need to show (1), since (1) implies (2). Up to refinements, we may assume that our cover is just a NL-atlas $x: X \rightarrow \X $ with $X$ a scheme. Let $x_n: X^n_{\X}\rightarrow \X $ be the map on the $n$-fold fiber product induced by $x$. Denote by $\check{C}_{\bullet}(x)$ the \v{C}ech simplicial gadget given at each level $n$ by:
	\[ \check{C}_{n}(x):=(x_{n,cl})_{\#}(x_{n,cl})^* \]
	For each $n$, we get natural maps $\varepsilon_{\#}^*(x_n)_{cl}: (x_{n,cl})_{\#}(x_{n,cl})^*\rightarrow Id  $. Taking the colimit, we obtain a map:
	\[ \varepsilon_{\#}^*(x_{\infty})_{cl}: \colim_{n\in \Delta} (x_{n,cl})_{\#}(x_{n,cl})^*\rightarrow Id \]
	For each $ Z \in \op{Lis}_{\X}^{ft} $, $ \varepsilon_{\#}^*(x_{\infty})_{cl}(Z): \colim_{n\in \Delta^{op}} (x_{n,cl})_{\#}(x_{n,cl})^*Z\rightarrow Z $ is a NL-covering sieve; hence for each $\mc V \in \op{PSh}\left(\op{Lis}_{\X}^{ft}  \right)$, the map $  \varepsilon_{\#}^*(x_{\infty})_{cl}(\mc V) $ is a NL-equivalence. \\
	Let $h: \mc F \rightarrow \mc G$ be a morphism in $ \op{PSh}\left( \op{Lis}^{ft}_{\X} \right) $, such that $x_{cl}^*h$ is a NL (resp. motivic) equivalence. Notice that $x_{n,cl}^*h$ will be a Nisnevich (resp. motivic) equivalence too for any $n$. This implies that:
	\[ \check{C}_{\bullet}(x)(h): \check{C}_{\bullet}(x)(\mc F)\longrightarrow \check{C}_{\bullet}(x)(\mc G) \]
	is an equivalence by \cite[Corollary 2.2.2]{Land_introductionQC}, where we look at $\check{C}_{\bullet}(x)(\mc F), \check{C}_{\bullet}(x)(\mc G)$ as objects in $\op{Fun}\left( \Delta^{op}, \op{PSh}\left( \op{Lis}^{ft}_{\X} \right) \right)$. Therefore taking the colimit over $\Delta^{op} $ will also result in an equivalence:
	\[ \check{C}_{\infty}(x)(h):= \colim_{n \in \Delta^{op}} \check{C}_{n}(x)(h): \colim_{n \in \Delta^{op}} \check{C}_{n}(x)(\mc F) \stackrel{\sim}{\longrightarrow} \colim_{n \in \Delta^{op}}\check{C}_{n}(x)(\mc G)  \]
	The maps $\check{C}_{\infty}(x)(h), \varepsilon_{\#}^*(x_{\infty})_{cl}(\mc F)$ and $\varepsilon_{\#}^*(x_{\infty})_{cl}(\mc G)  $ fit into a commutative square:
	
				\begin{center}
		\begin{tikzpicture}[baseline={(0,0)}, scale=1.95]
			\node (a) at (0,1) {$ \colim_{n \in \Delta^{op}} \check{C}_{n}(x)(\mc F) $};
			\node (b) at (2.5, 1) {$ \colim_{n \in \Delta^{op}} \check{C}_{n}(x)(\mc G) $};
			\node (c)  at (0,0) {$ \mc F $};
			\node (d) at (2.5,0) {$\mc G$};
			
			\node (e) at (0.25,0.75) {$ $};
			\node (f) at (0.5,0.5) {$ $};
			
			\path[font=\scriptsize,>= angle 90]
			
			(a) edge [->] node [above ] {$\check{C}_{\infty}(x)(h) $} (b)
			(a) edge [->] node [left] {$ \varepsilon_{\#}^*(x_{\infty})_{cl}(\mc F) $} (c)
			(b) edge [->] node [right] {$ \varepsilon_{\#}^*(x_{\infty})_{cl}(\mc G) $} (d)
			(c) edge [->] node [below] {$ h $} (d);
			
		\end{tikzpicture}
	\end{center}
	Since the top arrow and the vertical arrows in the square are Nisnevich (resp. motivic) equivalences, it follows that $h$ is a Nisnevich )resp. motivic equivalence too as we wanted.
	
\end{proof}

As a special case of the \cref{Sec.5:_NL_separation} we get:
\begin{corollary}\label{Sec.5:_Classical_conservativity_of_pullbacks}
	Let  $\mc X \in \oocatname{ASt}^{\leq 1}$ and let $x: X \rightarrow \mc X $ be a NL-atlas. Then the pullback functor:
	\[ x_{cl}^*: \Hcl(\mc X) \longrightarrow \H(X) \]
	\noindent is  conservative.
\end{corollary}

\begin{proposition}\label{Sec.5:_cl_and_ext_mot_spaces_pointwise}
	Let $\X \in \oocatname{ASt}^{\leq 1}$.  Then there exists a natural equivalence: 
	%\[i_{X^n/\X}: \op{Sm}^{ft}_{X^n_{\X}} \to \op{Sm}^{ft}_{\X}  \]
%	 \begin{equation} 
%	 	i^*_{X,\X}:\op{PSh}(\op{Sm}^{ft}_{\X}) \xrightarrow{\sim}\lim_{[n] \in\Delta} \op{PSh}(\op{Sm}^{ft}_{X^n_{\X}}) 
%	\end{equation}
%	
%	\noindent Morever, $i^*_{X, \mc X}$ commutes with the motivic localization and gives us the following:
	\begin{equation}\label{S2:_classical_unstable_equivalence}
		 \Hcl(\X) \xrightarrow{\simeq} \Hext(\X). 
	\end{equation} 
\end{proposition}
\begin{proof}
    Consider the cosimplicial diagram :
    \begin{equation}
        \Ca_{\bb} : N(\Delta_+) \to \op{Cat}_{\infty}    \end{equation}
        given by 
        \begin{equation}
           \begin{tikzcd}
               \Hcl(\X) \arrow[r] & \H(\X) \arrow[r,shift left =2] \arrow[r,shift right =2] & \H(X^2_{\X}) \arrow[l,dotted] \arrow [r, shift left =4] \arrow[r,shift right =4] & \cdots \arrow[l,shift left =2,dotted] \arrow[l,shift right =2,dotted]
            \end{tikzcd}
        \end{equation}
        We verify the conditions of \cite[Corollary 4.7.5.3]{HA}. The first condition follows from the fact that the pullback is a colimit preserving functor and $\Hcl(\X)$ is a presentable $\infty$-category hence admitting geometric realizations.\\
        The second condition follows from smooth base change for $\H(-)$ on schemes and smooth base change for $\Hcl(-)$(\cref{Sec.5:_smooth_cl_BC}). \\
        Thus applying \cite[Corollary 4.7.5.3]{HA}, we see that $\Hcl(\X) \cong \Hext(\X)$.
\end{proof}
\begin{corollary}\label{Sec.5:_cl_and_ext_mot_spaces}
	Let $\Hcl, \Hext: \oocatname{ASt}^{\leq 1}\rightarrow \oocatname{CAlg}\left( \op{Pr}^{L} \right) $ be the functors arising from the classical and extended construction of motivic spaces on NL-stacks. There exists a unique (up to contractible choice) natural transformation:
	\[\Psi: \Hcl \longrightarrow \Hext  \]
	\noindent that is an equivalence.
\end{corollary}
\begin{proof}
	By \cite[Lemma C.3]{Hoyois_Quad_Lefschetz}, taking the fully faithful functor given by the inclusion $u: \op{Sch}\into \oocatname{ASt}^{\leq 1} $, where the two categories are equipped with the Nisnevich and the NL topology respectively, we get that $ \Hext $ is the right Kan extension of $\H: \op{Sch}^{op}\rightarrow \oocatname{CAlg}(\op{Pr}^{L})  $. Since by \cref{Sec.5:_classical_H_schemes}, $\restrict{\Hcl}{\op{Sch}^{op}}= \H $, by the universal property of the right Kan extension we get a unique (up to contractible choice) functor:
	\[ \Psi: \Hcl \longrightarrow \Hext  \]
	We proved in \cref{Sec.5:_cl_and_ext_mot_spaces_pointwise} that we have an object-wise equivalence $\Hcl(\X)\simeq\Hext(\X)$ for any NL-algebraic stack $\X$. By \cite[Corollary 2.2.2]{Land_introductionQC}, this implies that $\Psi$ is indeed an equivalence as claimed.
\end{proof}

%	\begin{remark}
%		The above construction is functorial in $\Schf$ and thus we have a functor
%		\begin{equation}\Hcl(-)^*: N^D_{\bb}(\op{Nis-locSt})^{op} \to \op{Pr}^L.\end{equation} 
%		\noindent that loosely speaking sends $\X \in \oocatname{ASt}^{\leq 1}$ to $\Hcl(X)$ and sends a morphism $f$ to the natural pullback map $f^*$. It is an $\infty$-sheaf with respect to the $\op{Nis-loc}$ topology (this is because $\Hext(-)$ is an $\infty$-sheaf).
%	\end{remark}	
	
	\begin{notation}
		From now on due to \cref{Sec.5:_cl_and_ext_mot_spaces}, we will drop the distinction between $\Hcl$ and $\Hext$ if not necessary, and we will denote the underlying $\infty$-category of motivic spaces just as $\H$.
	\end{notation}

 As a first quick application of our comparison theorem between \textit{classical} and \textit{extended} motivic spaces categories, we have the following:

 \begin{proposition}[Smooth Projection Formula]\label{smooth_projection_formula}
     Let $ f: \X \rightarrow \Y $ be a smooth morphism in $\oocatname{ASt}^{\leq 1}$. Let $B\rightarrow A$ be a morphism in $ \H(\X) $. Then for every $C\in \H(\Y)_{/f^*A}$ and every $D \in \H(\X)_{/B}$, the canonical maps:
     \[ f\epfs(f^*B\otimes_{f^*A} C)\rightarrow B\otimes_{A}f\epfs C \]
     \[ f^*\iMap_{/A}(B,D) \rightarrow \iMap_{/f^*A}(f^*B, f^*D) \]
     \noindent are equivalences in $\H(\X)$ and $\H(\Y)$, respectively.
 \end{proposition}
 \begin{proof}
     If we consider the \textit{classical} incarnation $\H_{cl}(-)$ of the category of motivic spaces, we can follow the same exact proof as in \cite[Proposition 4.3]{Hoyois_Equiv_Six_Op}. Namely, \cite[Proposition 3.15]{Hoyois_Equiv_Six_Op} works also in our case and shows that $ L_{mot} $ is locally cartesian, in the sense of \cite{Gepner-Kock2017}, and preserves finite products. This means that we can reduce to the level of (representable) presheaves where the claim clearly holds.
 \end{proof}
	
	\begin{definition}
		As $\Hcl(\X)$ admits products and a final object, let $\Hcl(\X))_{*}$ be the $\infty$-category of pointed motivic objects. Following the notations of \cite{Robalothesis}, we can define the symmetric monoidal presentable $\infty$-category
		\begin{equation}
			\Hcl(\X)^{\otimes}_{*}:= (\Hcl(\X)^{\times})^{\otimes}_*
		\end{equation}
		whose underlying $\infty$-category is $\Hcl(\X)_*$.
	\end{definition}
	
	Now we would like to stabilise our category of pointed motivic spaces over $\mc X\in \oocatname{ASt}^{\leq 1}$ with respect to $\mb P^1_{\mc X}$. 
	
	\begin{proposition}\label{Sect_5:_3-symmetry_of_P1}
		Given $\mc X \in \oocatname{ASt}^{\leq 1}$, the cyclic permutation on $\left(\mb P^1_{\mc X}\right)^{\otimes 3}$ is homotopic to the identity in $  \Hcl(\mc X)  $, i.e. $\mb P^1_{\mc X}$ is 3-symmetric in the sense of \cite[Remark 2.20]{robalo2013noncommutative}.
	\end{proposition}
	\begin{proof}
		The proof given in \cite[Proposition 5.21]{Hoyois_Equiv_Six_Op} works verbatim for algebraic stacks too.
		
%		It is enough to show that the cyclic permutation of order two on $ \mb P^1_{\mc X}\otimes \mb P^1_{\mc X}$ is homotopic to $-Id\otimes Id$. Clearly we have $\mb P^1_{\mc X}\simeq \Th{\mc X}{\A^1_{\mc X}}  $; by the functoriality of the Thom space construction, as in the case of schemes, we have:
%		\[  \mb P^1_{\mc X} \otimes \mb P^1_{\mc X}\simeq \Th{\mc X}{\A^1_{\mc X}}\otimes \Th{\mc X}{\A^1_{\mc X}}\simeq \Th{\mc X}{\A^2_{\mc X}} \]
%		  \textcolor{red}{A: the functoriality of Th(-) is actually with target in SH, but H to SH is fully faithfull, therefore the iso of Th(A1 x A1)=Th(A2) in SH translates in the same equivalence in H. I don't think we should specify this, but we can always add a couple of lines if we think it's better.} \\
%		But then the same argument of \cite[Proposition 5.21]{Hoyois_Equiv_Six_Op} works in our case too. Indeed, since $ \left( \begin{array}{cc}
%			0 & 1 \\
%			1 & 0
%		\end{array} \right)$ and $\left( \begin{array}{cc}
%		-1 & 0 \\
%		0 & 1
%		\end{array} \right)$ are obtained from one another through elementary transformations, we can find an homotopy between them. Using the action of $SL_2(\mb Z)$ on $\Th{\mc X}{\A^2}$, we get our desired homotopy. 
	\end{proof}

	\begin{definition}
		Let $\X\in \oocatname{ASt}^{\leq 1}$ be an algebraic stack admitting Nisnevich-local sections. The \textit{classical} stable motivic homotopy category over $\X$ is the underlying $\infty$-category of the stable presentable symmetric monoidal $\infty$-category $\Shcl(\X)$ defined as:
		\begin{equation}
			\Shcl(\X):= \Hcl^{\otimes}(\X)_*[{(\P^1,\infty)}^{-1}]
		\end{equation}
		It is denoted by $\Shc(\X)$. Here the inversion of $(\P^1,\infty)$ is in the sense of \cite[Chapter 4]{Robalothesis} (also cf. \cite[\textsection 1]{Annala_Iwasa}).
	\end{definition}
		\begin{remark}\label{Sec5:_P1_3-symmetric}
		It is well known (cf. \cite[Proposition 1.6.3]{Annala_Iwasa}) that given an $n$-symmetric object $c \in \oocatname{C}$ in an $\infty$-category $\oocatname{C}$, we can model $\oocatname{C}[c^{-1}]$ with the category of telescopic objects:
		\[ \oocatname{C}[c^{-1}]\simeq\mr{Tel}_c(\oocatname{C}):=\colim\left( \oocatname{C} \stackrel{c\otimes }{\rightarrow}\oocatname{C} \stackrel{c\otimes }{\rightarrow} \oocatname{C} \stackrel{c\otimes }{\rightarrow} \ldots  \right) \]
		In particular, by \cref{Sect_5:_3-symmetry_of_P1} we get that:
		$ \Hcl^{\otimes}(\X)_*[c^{-1}]\simeq\mr{Tel}_c(\Hcl^{\otimes}(\X)_*)$.
	\end{remark}
	
	Let us conclude this section with the final comparison theorem:
	
%	\begin{proposition}
%		The equivalence of \cref{Sec.5:_cl_and_ext_mot_spaces}:
%		\begin{equation*}
%			\Hcl(\X) \xrightarrow{} \Hext(\X)%~~~ (\cref{classicalunstableequivalence})
%		\end{equation*}
%		induces a natural equivalence:
%		\begin{equation*}
%			\Shcl(\X) \xrightarrow{} \Shext(\X).
%		\end{equation*}
%%		Morever the equivalence is functorial and induces a natural transformation between functors which is an equivalence 
%%		\begin{equation}
%%			\Shcl(-) \xrightarrow{\cong} \Shext(-).
%%		\end{equation}
%	\end{proposition}
%	\begin{proof}
%		To show that the equivalence on the unstable level extends to stable level, we need to check that taking pointed objects and stabilizations commute along \v{C}ech covers of Nis-loc atlas $x: X \to \X$. \\
%		
%		Taking pointed spaces means equipping  a morphism $* \to \F$ for every object $\F \in \Hcl(\X)$. Given a morphism $* \to \F_i$ induces a morphism $ * \to \F$ by limit properties. Thus taking pointed objects commute along \v{C}ech covers. \\
%		
%		The argument of stabilization follows from the fact that $(\mathbf{P}^1,\infty) \otimes -$ is same as the inverse limit of $(\mathbf{P}^1,\infty)\otimes -$ on each $\Hcl^{\otimes}(X^n_{\X})_*$ which are invertible. 
%	
%			\end{proof}

	\begin{theorem}[Comparison Theorem]\label{Sec.5:_cl_and_ext_SH_pointwise_equivalence}
				We have a natural map:
				\begin{equation*}
					\Shcl(\X) \xrightarrow{} \Shext(\X).
				\end{equation*}
				\noindent and this map is an equivalence.
		\end{theorem}
		\begin{proof}
		From \cref{Sec.5:_cl_and_ext_mot_spaces}, we get an equivalence on pointed categories:
		\[ \Hcl(\mc X)_*\stackrel{\sim}{\longrightarrow} \Hext(\mc X)_* \]
		%\noindent with inverse:
		%\[ \Phi_*: \Hext(\mc X)_*\stackrel{\sim}{\longrightarrow} \Hcl(\mc X)_*  \]
		By definition:
		\[ \Shcl(\X):= \Hcl^{\otimes}(\X)_*[{(\P^1,\infty)}^{-1}] \]
		\noindent From here on in this proof, we will drop the superscript $ (-)^{\otimes} $, but we will always work with monoidal categories and monoidal functors between them.\\
		  Since the NL-topology is equivalent to the Nis-topology on algebraic spaces, if $ Z \in \oocatname{ASp} $ we recover the classical definition of the stable motivic homotopy category $\SH_{cl}(Z)=\SH^{}(Z)$. On the other hand, by construction we already know that for $Z \in \oocatname{ASp}$ we have $\SH_{ext}(Z)=\SH^{}(Z)$, and thus $\SH_{cl}(Z)=\SH_{ext}(Z)$ too.
		\noindent Let $x: X \rightarrow \X$ be an NL-atlas. For each $n$, denote by $i_{X^n/\mc X}^k: X_{\mc X}^{n}\rightarrow \mc X$ the natural projection from $X_{\mc X}^{n}$ to its $k^{th}$-component, composed with the atlas map $x$. Then we have pullback maps:
		\[ (i_{X^n/\mc X}^k)^*: \SH_{cl}(\X) \longrightarrow \SH(X_{\mc X}^{n})=\H^{}(X_{\mc X}^n)_*[{(\P^1,\infty)}^{-1}] \]
		Since it does not't really matter which projection  we are choosing, we will simply denote this map as $(i_{X^n/\mc X})^*$.
		By the universal property of the limit, we get a map:
		\[ i_{X^{\bullet}/\mc X}^*:=\lim_{n \in \Delta} i_{X^n/\mc X}^*: \SH_{cl}(\mc X) \longrightarrow \SH_{ext}(\mc X) \]
		To show that this is an equivalence, it is enough to prove that $\Shcl(-)$ is a NL-sheaf. Indeed, if $\SH_{cl}$ is a NL-sheaf, given any NL-atlas $X \rightarrow \X$, we can write:
		\[ \SH_{cl}(\X)\simeq \lim_{n \in \Delta} \SH_{cl}(X^n_{\X}) \]
		\noindent and $ i_{X^{\bullet}/\mc X}^*: \SH_{cl}(\mc X) \longrightarrow \SH_{ext}(\mc X)$ will induce a map of $\Delta$-indexed diagrams, that is point-wise an equivalence. To show that $\SH_{cl}$ is a sheaf we will follow standard arguments, reducing ourselves to the unstable case as in \cite[Theorem 2.3.4]{AGV_Rigid6FF}. From the fact that $\Hext$ is a NL-sheaf (cf. \cite[Theorem 2.2.1]{Chowdhury}) and from \cref{Sec.5:_cl_and_ext_mot_spaces}, we get for free that $\Hcl$ is a NL-sheaf. Denote by:
		\[ \begin{array}{cccc}
			\Hcl(-)^*: & \left(\oocatname{ASt}^{\leq 1}\right)^{op}& \longrightarrow & Pr^L
		\end{array} \]
		\noindent the functor informally given by assigning to a NL-stack $\X$ the $\infty$-category $\Hcl(\X)$ and to any map $f$ the pullback functor $f^*$; we will use $\SH_{cl}(-)^*$ for the stable  version of the functor, while with $\Hcl(-)_*$ (and $ \SH_{cl}(-)_* $) we will denote the similar functors sending $f \mapsto f_*$. By \cref{Sec5:_P1_3-symmetric}, we have that the functor $\SH_{cl}(-)^*$ can be computed as a colimit:
		\begin{equation}\label{Sect5:_colim_suspension}
			\SH_{cl}(-)^*\simeq \colim_{\mb N} \left( \Hcl(-)^*\stackrel{\Sigma_{\mb P^1_{-}}}{\longrightarrow} \Hcl(-)^*\stackrel{\Sigma_{\mb P^1_{-}}}{\longrightarrow} \Hcl(-)^*\stackrel{ \Sigma_{\mb P^1_{-}}}{\longrightarrow}  \ldots   \right)
		\end{equation}
		\noindent where $\Sigma_{\mb P^1_{-}}:= \mb P^1_{-}\otimes\cdot$.
		By \cite[Corollary 5.5.3.4, Theorem 5.5.3.18]{HTT}, under the equivalence $(Pr^L)^{op}\simeq Pr^{R}$, the diagram \eqref{Sect5:_colim_suspension} tells us that we can compute $\SH_{cl}(-)_*$ as an $\mb N^{op}$-indexed limit:
		\begin{equation}\label{Sect5:_lim_desuspension}
			\SH_{cl}(-)_*\simeq \lim_{\mb N^{op}} \left(\ldots \stackrel{\Omega_{\mb P^1_{-}}}{\longrightarrow} \Hcl(-)_* \stackrel{\Omega_{\mb P^1_{-}}}{\longrightarrow} \Hcl(-)_* \stackrel{\Omega_{\mb P^1_{-}}}{\longrightarrow} \Hcl(-)_*  \right) 
		\end{equation}
		\noindent with $ \Omega_{\mb P^1_{-}}:= \iMap(\mb P^1_{-}, \cdot)  $ given by the internal mapping space. To check that $\Shcl(-)^*$ is a NL-sheaf, it is enough to check that for each $\X \in \oocatname{ASt}^{\leq 1}$, $ \SH_{cl}(-)^* $ satisfies the descent condition on the site $ \left( RSNL, \tau_{NL} \right) $ of representable, smooth, NL-algebraic stacks over $\X$ equipped with the NL-topology. But for $f$ smooth maps, (unstable) pullbacks $f^*$ admit a left adjoint and hence they commute with right adjoints. Therefore we have an equivalence $f^*\Omega_{\mb P^1_{-}}\simeq\Omega_{\mb P^1_{-}} f^*$ and \eqref{Sect5:_lim_desuspension} along with \cite[Corollary 4.7.4.18]{HA} is telling us that we can compute $\SH_{cl}(-)^*$ as:
		\[  	\SH_{cl}(-)^*\simeq \lim_{\mb N^{op}} \left(\ldots \stackrel{\Omega_{\mb P^1_{-}}}{\longrightarrow} \Hcl(-)^* \stackrel{\Omega_{\mb P^1_{-}}}{\longrightarrow} \Hcl(-)^* \stackrel{\Omega_{\mb P^1_{-}}}{\longrightarrow} \Hcl(-)^*  \right)  \]
		But each term in the limit is already a NL-sheaf and therefore $ \SH_{cl}(-)^*$ must be a sheaf too and we are done.
		\end{proof}

		\begin{corollary}
			Let $\Shcl, \Shext: \oocatname{ASt}^{\leq 1}\rightarrow \oocatname{CAlg}\left( \op{Pr}^{L}_{stb} \right) $ be the functors arising from the classical and extended construction of stable motivic categories on NL-stacks. There exists a unique (up to contractible choice) natural transformation:
			\[\overline{\Psi}: \Shcl \longrightarrow \Shext  \]
			\noindent that is an equivalence.
		\end{corollary}
		\begin{proof}
			The same proof used in \cref{Sec.5:_cl_and_ext_mot_spaces} applies here too replacing the functors in \textit{loc.\!\!\ \  cit.\!} with their stable counterparts.
		\end{proof}
	As a consequence of the previous corollary, from here on, we are now allowed to exchange the underscripts of $\Shcl$ and $\Shext$ (and of $\SH_{cl}, \SH_{ext}$) when not really necessary. Moreover we will drop the superscripts completely if they are not necessary.
\begin{remark}\label{NFvsACconstruction}
    As mentioned in the introduction, a similar construction of $\Shcl(\X)$ has been carried out in \cite{Neeraj-Felix-SHcl}. In \textit{loc. cit.}, the authors define the stable homotopy category as the $\mathbb{P}^1$-stabilization of  $\mathbb{A}^1$-invariant sheaves on the Lisse-Nisnevich site $\op{Sm}_{\X}$, where objects are smooth $\X$-schemes, for $\X$ a NL-stack, and where coverings are given by Nisnevich coverings.  \\
    We have a natural inclusion of categories $\op{Sm}_{\X} \hookrightarrow \op{Lis}^{}_{\X}$ and, by \cite[Lemma C.3]{Hoyois_Quad_Lefschetz}, one can see that we have a natural equivalence of topoi $ \op{Sh}_{Nis}\left(\op{Sm}_{\X}\right)\simeq \op{Sh}_{NL}\left(\op{Lis}^{}_{\X}\right) $. Hence their construction agrees with ours already at the unstable level and therefore it agrees also after stabilisation.
\end{remark}	 
	\subsection{Extending Six-functor formalisms}\label{Exc_Funct}
In this section, we extend the exceptional functors on $\Shext(-)$ to non-representable locally of finite type morphism of $\op{NL}$-stacks. In \cite{Chowdhury}, the author only extended to representable separated of finite type morphisms. This section aims to extend the functors to non-representable situation. 

In order to extend this, we at first recall that exceptional functors already to extend to locally finite type morphisms in the level of schemes. We shall use Mann's formalism of abstract six-functor formalism using the category of correspondences to extend this functors. We recall the relevant definitions and provide proofs of relevant statements and propositions in the appendix.
\subsubsection{Extending $(-)_!$ to locally of finite type morphisms in the level of schemes.}
 Let $\Ca = \op{Sch}$ be the category of schemes and let $\E= \op{sep-ft}$ be the collection of morphisms which are separated and of finite type and $\E'=\op{lft}$ be the collection of morphisms which are locally of finite type. 

 It follows from \cite{Ayoub_6FF_vol1},\cite{Ayoub_6FF_vol2} and \cite{Robalothesis} that we have an abstract six-functor formalism of the form :

 \begin{equation}\label{Sec.3.2.1:_Schematic_SH*_!}
    \SH^*_!:=\SH_{(\Ca,\E)} : \op{Corr}\left( Sch \right)^{\otimes}_{sep\text{-}ft,all} \to \op{Cat}_{\infty}^{\otimes}
 \end{equation}
 \noindent that can be informally described as:
 \[ \begin{array}{cccc}
 	\SH^*_!: &  \op{Corr}\left( Sch \right)^{\otimes}_{sep\text{-}ft,all}  &\longrightarrow &\op{Cat}_{\infty}^{\otimes}\\
 	& & &\\
 	& X & \mapsto & \Sho(X)\\
 	& \begin{tikzcd}
 		& Z \arrow[ld, "f"'] \arrow[rd, "g"] &   \\
 		X &                                    & Y
 	\end{tikzcd} &\mapsto & g_!f^*
 \end{array} \]
which sends $X \mapsto \Sho(X)$ and morphisms, given by roofs where the right leg is in $sep$-$ft$, to the expected functors. It also encodes base change, projection formula as explained in the appendix.
The following proposition extends this functor to locally of finite type situation.
\begin{proposition}\label{CorrSchlocalfintype}
    The functor $\SH^*_!$ of \eqref{Sec.3.2.1:_Schematic_SH*_!} extends  to a functor:
    \begin{equation*}
        \SH^*_!:  \op{Corr}\left( Sch \right)^{\otimes}_{lft,all} \to \op{Cat}_{\infty}^{\otimes}
    \end{equation*}
\end{proposition}
\begin{proof}
 We verify the conditions of \cref{extendingsixfunctorEside}. Let $S$ be the collection of Zariski covers.
 \begin{enumerate}
     \item  From \cite[Proposition 6.24]{Hoyois_Equiv_Six_Op}, we know that $\SH(-)_!$ is a Nisnevich cosheaf. In particular it is codescent along Zariski hypercovers.
     \item Let $f: X \to Y$ be a morphism locally of finite type. It follows from the definition of locally of finite type, that we have a commutative square :
     \begin{equation*}
         \begin{tikzcd}
             X_0 \arrow[r] \arrow[d,"f_0"] & X \arrow[d,"f"] \\
             Y_0 \arrow[r] & Y
         \end{tikzcd}
        \end{equation*}
        where $f_0$ is separated of finite type. This follows from $X_0$ and $Y_0$ to be affine coverings of $X$ and $Y$ respectively. Iterating this process, we get a morphism of Zariski-hypercovers where levelwise the morphism is separated and of finite type.
        
 \end{enumerate}
Thus it follows from \cref{extendingsixfunctorEside} that $\SH_{(\op{Sch},sep\text{-}ft)}$ extends to $\SH_{(\op{Sch},lft)}$.
\end{proof}

The following proposition not only extends the exceptional functors but also proves base change and projection formula.

\subsubsection{Extending exceptional functors to the non-representable case.}

Let $\Ca= \op{Sch}$ and $\Ca' =\oocatname{ASt}^{\leq 1, NL}$. Let $\E_1 \subset \E_2$ be subset of morphisms in $\Ca'$ where:
\begin{itemize}
    \item $\E_1$ is the collection of morphisms of schemes which are locally of finite type.
    \item $\E_2=lft$ is the collection of morphism of $\op{NL}$-stacks which are locally of finite type (not necessarily representable).
\end{itemize}
We have the following proposition:
\begin{proposition}\label{Sect.3.2:_SH*!_Corr_functor}
    The functor  $\SH^*_!$ constructed in \cref{CorrSchlocalfintype} extends to functor :
    \begin{equation*}
        \SH^*_!: \op{Corr}\left(\oocatname{ASt}^{\leq 1}\right)^{\otimes}_{lft,\op{all}} \to \op{Cat}_{\infty}^{\otimes}.
    \end{equation*}
\end{proposition}
\begin{proof}
Let $\E_1'$ be the collection of morphisms in $\oocatname{ASt}^{\leq 1, NL}$ which are representable and locally of finite type. The extension follows by extending the functor in two steps.
\begin{enumerate}
    \item \textbf{Extending the functor from $\E_1$ to $\E_1'$:} We verify the conditions of \cref{extendingsixfunctorCside}. 
    \begin{itemize}
        \item By \cite[Proposition 6.24]{Hoyois_Equiv_Six_Op}, we see that $\SH^{\otimes*}(-)$ is a Nisnevich sheaf. Thus it satisfies descent for $S$-\v{C}ech covers where $S$ is the collection of $\op{NL}$ covers
        \item The second condition follows from the definition of $\oocatname{ASt}^{\leq 1, NL}$.
        \item The third condition follows from the definition of representable locally of finite  type.
    \end{itemize}
    Thus \cref{extendingsixfunctorCside} shows that the functor $\SH_({\Ca,\E_1)}$ extends to a functor $\SH_{(\Ca,\E_1')}$.
    \item \textbf{Extending the functor from $\E_1'$ to $\E_2$:} We verify the conditions of \cref{extendingsixfunctorEside}.
    \begin{itemize}
        \item By  \cite[Proposition 6.24]{Hoyois_Equiv_Six_Op}, $\SH_!(-)$ is a Nisnevich cosheaf on schemes, in particular it is also a $\op{NL}$-cosheaf (since on schemes the two topologies agree). As representable $\op{NL}$-covers in $\AstNl$ fit in a pullback square where the map between schemes is $\op{NL}$, it follows that $\SH(-)_!$ satisfies codescent along representable $\op{NL}$-covers. Let $S$ be the collection of morphisms which are representable $\op{NL}$-covers.
        \item Given a non representable map locally of finite type of $\op{NL}$-stacks $f: X \to Y$, we have a commutative square 
         \begin{equation*}
         \begin{tikzcd}
             X_0 \arrow[r] \arrow[d,"f_0"] & X \arrow[d,"f"] \\
             Y_0 \arrow[r] & Y
         \end{tikzcd}
        \end{equation*}
        where the horizontal rows are $\op{NL}$-atlases and $f_0$ is a locally of finite type map of schemes. Iterating this, we get a map of $\op{NL}$-hypercovers $f_{\bb} : X_{\bb} \to Y_{\bb}$ such that $f_n$ is locally of finite type.
    \end{itemize}
    Thus the conditions of \cref{extendingsixfunctorEside} are verified. Hence $\SH_{(\Ca,\E_1')}$ extends to a functor :
    \begin{equation*}
               \SH_{(\Ca',\E_2)} : \op{Corr}(\Ca')^{\otimes}_{\E_2,\op{all}} \to \op{Cat}_{\infty}^{\otimes}.
    \end{equation*}
\end{enumerate}
\end{proof}	

\begin{remark}
    It is not hard to show that with similar techniques to the ones used in the proposition above, it is possible to construct a functor:
    \[ \SH^*_{\#}: \op{Corr}\left(\oocatname{ASt}^{\leq 1}\right)^{\otimes}_{smooth,\op{all}} \to \op{Cat}_{\infty}^{\otimes} \]
    \noindent using \textit{smooth base change} proved in \cref{Sec.5:_smooth_cl_BC} and NL-descent along pullback functors. Similarly to what explained in Appendix \ref{App.A.2:_6FF}, this will give us the smooth projection formula one can deduce from the unstable one proved in \cref{smooth_projection_formula}.\\
    There is also a third way to get the smooth projection formula. Indeed given a smooth map $f:\X\rightarrow \Y$, from \cref{Sect.3.2:_SH*!_Corr_functor},  one obtains a projection formula of the form:
    \[ f_!(-\otimes f^*(-))\simeq -\otimes f_!(-) \]
    Using the purity equivalence (that we will prove) in \cref{Sec.4:_Stacky_Purity}, we get:
    \[ f\epfs(-\otimes f^*(-))\simeq -\otimes f\epfs(-) \]
    \noindent that is exactly what we were looking for.
    
\end{remark}

%	\begin{remark}
%		Since $ f^*:\mathcal{SH}_{ext}(\mathcal{X}) \rightarrow \mathcal{SH}(S) $ corresponds to the usual pullback functor when $ \mathcal{X}\rightarrow S $ is just a scheme, we know that $ f_{\#}(\mbbm 1_{\mathcal{X}})=\Sigma^{\infty}\mathcal{X} $. The same will be true for when $ \mathcal{X} $ is an honest smooth, finite type, (quasi-separated) Nis-loc algebraic stack over some base scheme $ S $. Indeed, if we denote by $ \Sigma^{\infty}: St_{/S} \longrightarrow \mathcal{H}_{cl}(S) $ the map induced by the Yoneda embedding into presheaves, composed with motivic localization, we get that $ f_{\#}(\mbbm 1_{\mathcal{X}})=\Sigma^{\infty} \mathcal{X} $ reducing by \cref{Sec.5:_cl_and_ext_mot_spaces} to the case of schemes.
%	\end{remark}

	\section{Non-Representable Relative Purity}\label{Amb_and_Purity}
	\subsection{Borel J-homomorphism}
	
		Let $ \mc X $ be an NL-stack and let $ \mc E $ be a locally free sheaf on $ \mc X $. Denote by $ p: E \rightarrow \mc X $ be a vector bundle associated to $ \mc E $, with zero section given by $ s: \mc X \rightarrow E $. Notice that both $ p $ and $ s $ are representable maps. Following \cite[\S 5.2]{Hoyois_Equiv_Six_Op}, we have adjoint functors:
	\[ p\epfs s_*: \mr{SH}(\mc X)\leftrightarrows \mr{SH}(\mc X): s^!p^* \]
	
	\begin{lemma}
		The adjoint functors:
		\[  p\epfs s_*: \mr{SH}(\mc X)\leftrightarrows \mr{SH}(\mc X): s^!p^* \]
		\noindent are equivalence of $ \infty $-categories, inverse to one another. 
	\end{lemma}
	
	\begin{proof}
		Choose a NL-atlas $ x: X \rightarrow \mc X $ and consider:
		
		\begin{center}
			\begin{tikzpicture}[baseline={(0,-1)}, scale=2]
				\node (a) at (0,1) {$ E_X $};
				\node (b) at (1, 1) {$ X$};
				\node (c)  at (0,0) {$  E $};
				\node (d) at (1,0) {$ \mc X $};
				\node (e) at (0.2,0.75) {$ \ulcorner $};
				\node (f) at (0.5,0.5) {$ \Delta $};

				\path[font=\scriptsize,>= angle 90]
				
				(a) edge [->] node [below ] {$ p_X $} (b)
				(a) edge [->] node [left] {$ y $} (c)
				(b) edge[->] node [right] {$ x $} (d)
				(c) edge [->] node [below] {$ p $} (d)
				(b) edge [bend right=30,->] node [below] {$ s_X $} (a)
				(d) edge [bend left=30, ->] node [below] {$ s $} (c);
			\end{tikzpicture}
		\end{center}
		Pulling back via $ x^* $ the unit and co-unit of the adjunction $ (p\epfs s_* \dashv s^!p^*) $, we get:
		\[ x^*\epsilon: x^* \longrightarrow x^*(s^!p^*)(p\epfs s_* )\simeq (s_X^!p_X^*)(p_X\epfs s_X{}_* )x^* \]
		\[ x^*\eta: x^*(p\epfs s_*)(s^!p^*)\simeq (p_X\epfs s_X{}_*)(s_X^!p_X^*)x^* \longrightarrow x^* \]
		\noindent and both are equivalences since the natural transformations $ Id_{\mr{SH}(X)}\stackrel{\sim}{\rightarrow } (p\epfs s_* )\simeq (s_X^!p_X^*)(p_X\epfs s_X{}_* ) $ and $ (p_X\epfs s_X{}_*)(s_X^!p_X^*) \stackrel{\sim}{\rightarrow} Id_{\mr{SH}(X)} $ are equivalences themselves. By conservativity of $ x^* $, we get that:
		\[ \epsilon: Id_{\mr{SH}(\mc X)} \longrightarrow (s^!p^*)(p\epfs s_* ) \]
		\[ \eta: (p\epfs s_*)(s^!p^*) \longrightarrow Id_{\mr{SH}(\mc X)}\]
		\noindent are equivalences and hence $ p\epfs s_* $ and $ s^!p^* $ are inverse to one another.
	\end{proof}

	\begin{definition}
		In the same situation as above, the adjoint functors:
		\[ \Sigma^{\mc E}:=p\epfs s_*: \mr{SH}(\mc X)\leftrightarrows \mr{SH}(\mc X): s^!p^*=:\Sigma^{-\mc E} \]
		\noindent are equivalences of $ \infty $-categories and are called \textit{Thom transformations}. In particular we have $ \Sigma^{\mc E}\simeq \Sigma^{\mc E}\mbbm 1_{\mc X}\otimes (-) $. Denote by $ \mr{Pic}(\mr{SH}(\mc X)) $ the $ \infty $-category of invertible objects in $ \mr{SH}(\mc X) $, and denote by:
		\[ \Th{\mc X}{E}:=\Sigma^{\mc E}\mbbm 1_{\mc X} \in \mr{Pic}(\mr{SH}(\mc X)) \]
		\noindent the \textit{Thom space} of $ E $, with inverse $ \Sigma^{-\mc E}\mbbm 1_{\mc X} $.
	\end{definition}
	
	We can make a further upgrade of the Thom space construction. Thanks to \cite[§ 16.2]{Bachmann_Hoyois_Norms_MHT}, we have a $ J $-homomorphism functor:
	\[ J: \mr K \longrightarrow \mr{Pic}(\mr{SH}) \]
	\noindent where $ K: {Sch}^{qcqs}\rightarrow  \oocatname{S}\sseq \oocatname{Cat}_{\infty} $ it the $ cdh $-sheaf assigning to each qcqs scheme $ X $ the Thomason-Trobaugh K-theory space and sending each map $ f $ to $ f^* $ (in K-theory), while $ \mr{Pic}(\mr{SH}) $ is the $ cdh $-sheaf that assigns to each scheme the space of invertible objects $ \mr{Pic}(\mr{SH}(X)) $ and sends a map $ f $ to $ f^* $. Since both are $ cdh $-sheaves, and in particular Nisnevich sheaves, we can Kan extend $ J $ via NL-sheafification to a natural transformation:
	
	\[ J^{\triangleleft}: \mr K^{\triangleleft}(-) \longrightarrow \mr{Pic}(SH^{\triangleleft}(-)) \]
	
	\noindent where $ \mr K^{\triangleleft} $ and $ \mr{Pic}(SH^{\triangleleft}) $ are the functors defined on $ \left( \oocatname{ASt}^{\leq 1} \right)^{op} $ via Kan extension (as similarly done in \cite[Theorem 2.2.1]{Chowdhury}). \\
	
	On the other hand, the genuine K-theory for an algebraic stack $ \mc X \in \oocatname{ASt} $ is defined as:
	\[ \mr K\left( \mc X \right):=\Omega^{\infty}\mr K\left(  \oocatname{Perf}(\mc X)\right) \]
	\noindent where the right hand side is the Thomason-Trobaugh K-theory space of the $ \infty $-category of perfect complexes (cf. \cite{Khan_K-Theory}). For any map of algebraic stacks $ f: \mc X \rightarrow \mc Y $, we have a pullback map $ f^*: 	\mr K(\mc Y) \rightarrow \mr K(\mc X) $ induced by the pullback at the level of perfect complexes. If $ x: X \rightarrow \mc X $ is a NL-atlas, then for any map $ f_n: X^n_{\mc X}:=X \times_{\mc X}\ldots \times_{\mc X} X \rightarrow \mc X $ in the \v{C}ech nerve of the atlas, we will get a pullback map $ f_n^*: \mr K(\mc X)\rightarrow \mr K(X^n_{\mc X}) $. For any NL-stack $ \mc X $, the space $ \mr K^{\triangleleft}(\mc X) $ is a limit over the \v{C}ech nerve of one of his NL-atlases (by construction) and this is functorial in $ \mc X $. Hence by the universal property of the limit we get a canonical map:
	
	\[ j: \mr K(-) \longrightarrow \mr K^{\triangleleft}(-) \]
	\noindent between presheaves $ \mr K, \mr K^{\triangleleft}: \left(\oocatname{ASt}^{\leq 1}\right)^{op}\rightarrow \oocatname{S} $ taking values in Kan complexes $\oocatname{S}$.
	\begin{remark}
		The map $j: \mr K \longrightarrow \mr K^{\triangleleft}$ is simply given by NL-sheafification.
	\end{remark}
	\begin{definition}
		We define the \textit{Borel J-homomorphism} as:
		\[ J_{Bor}:= J^{\triangleleft}\circ j: \mr K \longrightarrow \mr{Pic}(\mr{SH}^{\triangleleft}) \]
		For a given NL-stack $ \mc X $ and a given $ v \in \mr K_0(\mc X) $, we will denote the associated automorphism of $ \mr{SH}(\mc X) $ as $ \Sigma^{v} $, with inverse $ \Sigma^{-v} $. 
	\end{definition}

	\begin{remark}
		\begin{enumerate}
			\item	If $ \mc V $ is a locally free sheaf on a NL-stack $ \mc X $, with associated vector bundle $ V $ and with $ v:=[\mc V]\in \mr K_0(\mc X) $, then the $ J_{Bor} $ will send $ v $ to $ \Th{\mc X}{V}=\Sigma^{v}\mbbm 1_{\mc X} $ defined before: this is true at the level of schemes and the same claim follows for NL-stacks by uniqueness of Kan extensions. 
			\item Notice that a fiber sequence in $ \oocatname{Perf}(\mc X) $ of the form:
			\[ \mc E \rightarrow\mc F \rightarrow \mc G \]
			\noindent gives us a canonical path $ [\mc F]\simeq [\mc E]+[\mc G] $ in the space $ K(\mc X) $. Then, by construction of the $ J_{Bor} $ natural transformation, we get that:
			\[ \Sigma^{\mc F}\simeq \Sigma^{\mc E}\Sigma^{\mc G} \]
		\end{enumerate}
		
	\end{remark}
	
	\subsection{Pairs and Formal Thom Transformations}
	%Let's fix a base scheme $S$. Denote by:
	%\[ \catname{Pairs}_S^{cl} \]
	%the category formed by closed immersion between $S$-smooth NL-stacks and where the morphisms are given by cartesian squares thereof (cf. \cite[\S 3]{Hoyois_Six_Operations}). 
	
	Consider the category of pairs $Pairs$ where objects are of the form:
	\begin{center}
		\begin{tikzpicture}[baseline={(0,0.5)}, scale=1.5]
			\node (a) at (0,1) {$ \mc X $};
			\node (b) at (1,1) {$ \mc Y $};
			\node (c) at (1,0) {$ S $};

			\path[font=\scriptsize,>= angle 90]
			
			(a) edge [->] node [above ] {$ f $} (b)
			(a) edge [->] node [left] {$q$} (c)
			(b) edge [->] node [right] {$p$} (c);
			
		\end{tikzpicture}
	\end{center}
	
	\noindent with $f$ a representable map and $ p,q $ smooth (possibly non representable) maps. Notice that since $p,q$ are smooth, the map $f$ is automatically locally complete intersection (\textit{lci} from  here on), since one can reduce to schemes where this holds by \cite[\href{https://stacks.math.columbia.edu/tag/02FV}{Tag 02FV}]{stacks-project} and \cite[\href{https://stacks.math.columbia.edu/tag/0E9K}{Tag 0E9K}]{stacks-project}. Morphisms in this category are given by commutative diagrams of the form:
	
	\begin{center}
		\begin{tikzpicture}[baseline={(0,0.5)}, scale=1.5]
			\node (a) at (0,1.25) {$ \mc X $};
			\node (b) at (1,0.75) {$ \mc Y $};
			\node (c) at (1,0) {$ S $};
			
			\node (a1) at (2,1.25) {$ \mc Z $};
			\node (b1) at (3,0.75) {$ \mc W $};
			\node (c1) at (3,0) {$ S' $};

			\path[font=\scriptsize,>= angle 90]
			
			(a) edge [->] node [above ] {$ f $} (b)
			(a) edge [->] node [left] {$q$} (c)
			(b) edge [->] node [right] {$p$} (c)
			
			(a1) edge [->] node [above ] {$ g $} (b1)
			(a1) edge [->] node [left] {$r$} (c1)
			(b1) edge [->] node [right] {$t$} (c1)
			
			(b) edge [-, line width=1mm, white] node [] {} (b1)
			
			(a) edge [->] node [above ] {$ \alpha_{\mc X} $} (a1)
			(b) edge [->] node [below] {$  \alpha_{\mc Y} $} (b1)
			(c) edge [->] node [below] {$ \alpha_S $} (c1);
			
		\end{tikzpicture}
	\end{center}
	\noindent where the top square is cartesian (i.e. $\mc X=\mc Y\times_{\mc W}\mc Z$). For short we will denote a morphism like the one above as:
	\[ (\mc X,\mc Y)_S\stackrel{\alpha}{\longrightarrow}(\mc Z,\mc W)_{S'} \]
	\begin{definition}
		Let:
		\begin{center}
			\begin{tikzpicture}[baseline={(0,0.5)}, scale=1.5]
				\node (a) at (0,1) {$ \mc X $};
				\node (b) at (1,1) {$ \mc Y $};
				\node (c) at (1,0) {$ S $};

				\path[font=\scriptsize,>= angle 90]
				
				(a) edge [->] node [above ] {$ f $} (b)
				(a) edge [->] node [left] {$q$} (c)
				(b) edge [->] node [right] {$p$} (c);
				
			\end{tikzpicture}
		\end{center}
		\noindent be an element of $ Pairs^{lft} $. We will define the associated \textit{(formal) Thom transformation} as:
		\[ \mr{Th}(f,p):= p_{\#}f_! \]
	\end{definition}
	
	Now consider a commutative square:
	
	\begin{center}
		\begin{tikzpicture}[baseline={(0,0.5)}, scale=1.5]
			\node (a) at (0,1) {$ \mc Y $};
			\node (b) at (1,1) {$ \mc W$};
			\node (c) at (0,0) {$ \mc S $};
			\node (d) at (1,0) {$ \mc S' $};
			\node (e) at (0.25,0.75) {$  $};

			\path[font=\scriptsize,>= angle 90]
			
			(a) edge [->] node [above ] {$ r $} (b)
			(a) edge [->] node [left] {$ v $} (c)
			(b) edge [->] node [right] {$ u $} (d)
			(c) edge [->] node [below] {$ s $} (d);
			
		\end{tikzpicture}
	\end{center}
	\noindent and assume that $ u,v $ are smooth maps. Like in the classical schematic case, we get a natural transformation:
	\[ Ex_{\#}^*: v_{\#}r^* \longrightarrow s^*u_{\#} \]
	\noindent given by the composite:
	\[ v_{\#}r^* \overset{\eta_{\#}^*(u)}{\longrightarrow} v_{\#}r^*u^*u_{\#} \simeq v_{\#}v^*s^*u_{\#} \overset{\varepsilon_{\#}^*(v)}{\longrightarrow} s^*u_{\#} \]
	
	Following standard constructions (cf. \cite[\S 3]{Hoyois_Equiv_Six_Op}), for a given a morphism of pairs $\alpha: (\mc X,\mc Y)_S \longrightarrow(\mc Z,\mc W)_{S'}  $:
	\begin{center}
		\begin{tikzpicture}[baseline={(0,0.5)}, scale=1.5]
			\node (a) at (0,1.25) {$ \mc X $};
			\node (b) at (1,0.75) {$ \mc Y $};
			\node (c) at (1,0) {$ S $};
			
			\node (a1) at (2,1.25) {$ \mc Z $};
			\node (b1) at (3,0.75) {$ \mc W $};
			\node (c1) at (3,0) {$ S' $};

			\path[font=\scriptsize,>= angle 90]
			
			(a) edge [->] node [above ] {$ f $} (b)
			(a) edge [->] node [left] {$q$} (c)
			(b) edge [->] node [right] {$p$} (c)
			
			(a1) edge [->] node [above ] {$ g $} (b1)
			(a1) edge [->] node [left] {$r$} (c1)
			(b1) edge [->] node [right] {$t$} (c1)
			
			(b) edge [-, line width=1mm, white] node [] {} (b1)
			
			(a) edge [->] node [above ] {$ \alpha_{\mc X} $} (a1)
			(b) edge [->] node [below] {$  \alpha_{\mc Y} $} (b1)
			(c) edge [->] node [below] {$ \alpha_S $} (c1);
			
		\end{tikzpicture}
	\end{center}
	\noindent we can associate to it a natural transformation:
	\begin{equation}\label{Sec.4:_formal_Thom_naturality_map}
		\psi(\alpha): \mr{ Th}(f,p)\alpha_{\mc X}^*:=p_{\#}f_!\alpha_{\mc X}^* \longrightarrow \alpha_S^*t_{\#}g_!=:\alpha_S^*\mr{Th}(g,t)
	\end{equation}
	
	\noindent defined as the composition:
	\[   p_{\#}f_!\alpha_{\mc X}^* \underset{}{\overset{Ex_!^*}{\simeq}} p_{\#}\alpha_{\mc Y}^*g_! \overset{Ex_{\#}^*}{\longrightarrow}  \alpha_S^*t_{\#}g_!  \]
	\noindent where $Ex^*_!$ is the equivalence of \eqref{App.:_Base_Change_Ex*_!} witnessed by the functor constructed in \cref{Sect.3.2:_SH*!_Corr_functor}.
	
	\subsection{Homotopy Purity Statement}
	By \cite{Aranha-Pstragowski}, we can functorially assign to any lft map $f: \mc X \rightarrow \mc Y$ a deformation space $\mr{Def}_{f}$, flat over $\A^1_{\mc Y}$. If $f$ is a closed immersion, then the deformation space is the usual Verdier's deformation space. In general, we have the following cartesian diagrams:\\
	\begin{minipage}{0.5\textwidth}
		\begin{center}
			\begin{tikzpicture}[baseline={(0,0.5)}, scale=1.5]
				\node (a) at (0,1) {$ \mf C_f $};
				\node (b) at (1,1) {$ \mr{Def}_f $};
				\node (c) at (0,0) {$ \set{0}\times \mc Y $};
				\node (d) at (1,0) {$ \A^1_{\mc Y} $};
				\node (e) at (0.25,0.75) {$ \ulcorner $};

				\path[font=\scriptsize,>= angle 90]
				
				(a) edge [->] node [above ] {$  $} (b)
				(a) edge [->] node [left] {$ $} (c)
				(b) edge [->] node [right] {$ $} (d)
				(c) edge [->] node [right] {$ $} (d);
				
			\end{tikzpicture}
		\end{center}
	\end{minipage}
	\hfill 
	\begin{minipage}{0.5\textwidth}
		\begin{center}
			\begin{tikzpicture}[baseline={(0,0.5)}, scale=1.5]
				\node (a) at (0,1) {$ \mb G_m\times \mc Y $};
				\node (b) at (1,1) {$ \mr{Def}_f $};
				\node (c) at (0,0) {$\mb G_m \times \mc Y $};
				\node (d) at (1,0) {$ \A^1_{\mc Y} $};
				\node (e) at (0.25,0.75) {$ \ulcorner $};

				\path[font=\scriptsize,>= angle 90]
				
				(a) edge [->] node [above ] {$  $} (b)
				(a) edge [->] node [left] {$ $} (c)
				(b) edge [->] node [right] {$ $} (d)
				(c) edge [->] node [right] {$ $} (d);
				
			\end{tikzpicture}
		\end{center}
	\end{minipage}\\
	\begin{remark}
		If $f$ is $n$-representable then $\mr{Def}_f$ is $(n+1)$-representable. The stackiness is contained in the \textit{intrinsic normal cone} $\mf C_f$. In particular, if $f$ is an lft map between schemes, Zarisky locally we can factor $f$ as a closed immersion $\iota:  X\supseteq  U \into M$ and a smooth surjection $\mu: M \rightarrow V \subseteq Y$ and we can describe $\mf C_f$ (Zariski locally of course) as the quotient stack $ \left[ \bigslant{\mf C_{\iota}}{\iota^*T_{\mu }} \right] $ where $\mf C_{\iota}$ is the classical normal cone of the closed immersion and $T_{\mu}$ is the tangent bundle of the smooth surjection.
	\end{remark}
	
	\begin{construction}
	Let:
	\[ r_{\_}: \A^1_{\_} \longrightarrow - \]
	\noindent be the projection map of $\A^1$ onto its base. Let $ f:  X\rightarrow  Y$  be a representable map between smooth $NL$-stacks over some base $S$. Since $f$ is representable, we have that $\mr{Def}_f$ and the normal cone $ \mf C_f $ of $f$ are both representable by algebraic $1$-stacks. Then the deformation space gives us the following (co-)roof of pairs:
		
		\begin{center}
			\begin{tikzpicture}[baseline={(0,0.5)}, scale=1.5]
				\node (a) at (0,2.25) {$  X $};
				\node (b) at (1,1.75) {$  Y $};
				\node (c) at (1,0) {$ S $};
				
				\node (a1) at (2,2.25) {$ \A^1_X $};
				\node (b1) at (3,1.75) {$ \mr{Def}_f $};
				\node (d1) at (3,1) {$ \A^1_S $};
				\node (c1) at (3,0) {$ S $};
				
				\node (a2) at (4,2.25) {$ X $};
				\node (b2) at (5,1.75) {$ \mf C_f $};
				\node (d2) at (5,1) {$ X $};
				\node (c2) at (5,0) {$ S $};
				
				\path[font=\scriptsize,>= angle 90]
				%arrows in back ground
				(a1) edge [->] node [left] {$ \hat{q} $} (c1)
				(a2) edge [->] node [left] {$ q $} (c2)
				
				% arrows in the front
				(a) edge [->] node [above=1mm,right=-1mm ] {$ f $} (b)
				(a) edge [->] node [left] {$q$} (c)
				(b) edge [->] node [right] {$p$} (c)
				
				(a) edge [->] node [above ] {$ \alpha_{ X} $} (a1)
				(b) edge [-, line width=1mm, white] node [] {} (b1)
				(b) edge [->] node [below] {$  \alpha_{Y} $} (b1)
				(c) edge [double equal sign distance] node [right] {$  $} (c1)
				
				(a1) edge [->] node [above=1mm,right=-1mm ] {$ \hat{f} $} (b1)
				
				(b1) edge [->] node [right] {$ \hat{p} $} (d1)
				(d1) edge [->] node [right] {$ r_S $} (c1)
				
				(a2) edge [->] node [above ] {$ \beta_{ X} $} (a1)
				(b2) edge [-, line width=1mm, white] node [] {} (b1)
				(b2) edge [->] node [below] {$  \beta_{ \mf C_f} $} (b1)
				(c2) edge [double equal sign distance] node [right] {$  $} (c1)
				
				(a2) edge [->] node [above=1mm,right=-1mm ]  {$ f_0 $} (b2)
				
				(b2) edge [->] node [left] {$ q_0 $} (d2)
				(d2) edge [->] node [left] {$ q $} (c2)
				(b2) edge [bend left=30,->] node [right] {$ p_0 $} (c2);
				
			\end{tikzpicture}
		\end{center}
		\noindent where $\alpha$ and $\beta$ are induced by the inclusion of fibers $\iota_1: \set{1}\times Y \into \A^1_Y$ and $\iota_0: \set{0}\times Y \into\A^1_Y$ respectively.	From $\alpha$ and $\beta$, applying the construction of $\psi(-)$ from \eqref{Sec.4:_formal_Thom_naturality_map}, we get a (co-)roof:
		
		\[ \mr{Th}(f,p) \alpha_{ X}^* \overset{\psi(\alpha)}{\rightarrow} (r_S)\epfs \mr{Th}(\hat{f},\hat{p}) \overset{\psi(\beta)}{\leftarrow} q_{\#}\mr{Th}(f_0,q_0)\beta_{ X}^* \]
		
		Since both $\alpha_X$ and $\beta_X$ are inclusions of fibers over $1$ and $0$ in $\A^1_X$, this means that $ r_X\circ \alpha_X=Id $ and $r_X\circ \beta_X=Id$. Hence, precomposing $\psi(\alpha)$ and $\psi(\beta)$ with $r_X^*$ we get:
		
		\begin{equation}\label{eq:_Homotopy_Purity_Roof}
			\mr{Th}(f,p) \overset{}{\longrightarrow} (r_S)\epfs \mr{Th}(\hat{f},\hat{p})r_X^* \overset{}{\longleftarrow} q_{\#}\mr{Th}(f_0,\pi_0)
		\end{equation}
		\begin{notation}\label{Notation:_HP}
			From now on we will use the following notation:
			\[ \Pi_1(f,p):=\psi(\alpha)r_X^* \]
			\[ \Pi_0(f,p):=\psi(\beta)r_X^* \]
			Moreover for a given pair $(X,Y)_S\in Pairs$, we will always denote by $\hat{f}, \hat{p}, r_{-}, f_0,p_0 $ the maps:
			\[ \hat{f}: \A^1_X \rightarrow \mr{Def}_f\ \ \ \ \ \ \ \ \ \ \ \ \ \ \ \   f_0: X \rightarrow \mf C_f \]
			\[ \hat{p}: \mr{Def}_f \rightarrow \A^1_S \ \ \ \ \ \ \ \ \ \ \ \ \ \ \ \   p_0:=q\circ q_0:  \mf C_f \overset{q_0}{\rightarrow}X\overset{q}{\rightarrow} S  \]
			\[ r_{-}: \A^1_{-}\rightarrow - \]
		\end{notation}
		
	\end{construction}
	
		\begin{definition}\label{sec.6:_definition_HP}
			Given $f:X\rightarrow Y$ a smooth, representable map between smooth $S$-NL-stacks, we say that $f$ satisfies \textit{homotopy purity} if $\Pi_0(f,p)$ and $\Pi_1(f,p)$ are both equivalences.
		\end{definition}
		
		\subsection{Ambidexterity Statement}
	
Let us recall how to construct the natural exchange transformation $Ex_{\#!}$. 	Consider a cartesian square:
\begin{center}
	\begin{tikzpicture}[baseline={(0,0.5)}, scale=1.5]
		\node (a) at (0,1) {$ \mc Y $};
		\node (b) at (1,1) {$ \mc W$};
		\node (c) at (0,0) {$ \mc S $};
		\node (d) at (1,0) {$ \mc S' $};
		\node (e) at (0.2,0.75) {$ \ulcorner $};

		\path[font=\scriptsize,>= angle 90]
		
		(a) edge [->] node [above ] {$ q $} (b)
		(a) edge [->] node [left] {$ g $} (c)
		(b) edge [->] node [right] {$ f $} (d)
		(c) edge [->] node [below] {$ p $} (d);
		
	\end{tikzpicture}
\end{center}
\noindent where $ f,g $ are smooth maps and $ q,p $ are lft. Like in the classical schematic case, the natural transformation:
\[ Ex_{\#!}: f\epfs q_! \longrightarrow p_!g\epfs \]
\noindent is given by the composition:
\[ f_{\#}q_! \overset{\eta_{\#}^*(g)}{\longrightarrow} f_{\#}q_!g^*g_{\#} \overset{Ex_!^*}{\simeq} f_{\#}f^*p_!g_{\#} \overset{\varepsilon_{\#}^*(f)}{\longrightarrow} p_!g_{\#} \]	
\begin{remark}\label{Sec.4.4:_Ex-sharp-!_equiv_schematic_case}
	Notice that if $\mc Y,\mc W, \mc S,\mc S'$ are reprensented by schemes, then $Ex_{\#!}$ is actually an equivalence. If we assume that $f,g$ are (smooth and) separated, this follows dualising the equivalence exchange transformation $ Ex^{*!} $ of \cite[p.272]{Hoyois_Equiv_Six_Op}. In general, one can use Zariski descent to reduce the lft case to the separated one (for a completely analogous statement in the rigid context cf. \cite[Theorem 4.4.29]{AGV_Rigid6FF}).
\end{remark}
	
	Now, given $f: \mc X \longrightarrow \mc Y$ a smooth map in $\oocatname{ASt}^{\leq 1}$, we can consider the following diagram:
	
	\begin{center}
		\begin{tikzpicture}[baseline={(0,-1)}, scale=2]

			\node (a) at (0,1) {$ \mc X\times_{\mc Y} \mc X $};
			\node (b) at (1, 1) {$ \mc X $};
			\node (c)  at (0,0) {$  \mc X $};
			\node (d) at (1,0) {$ \mc Y $};
			\node (e) at (0.2,0.75) {$ \ulcorner $};
			\node (f) at (-0.75,1.5) {$ \mc X $};
			\node (g) at (0.5,0.5) {$ \Delta $};

			\path[font=\scriptsize,>= angle 90]

			(a) edge [->] node [above ] {$ p_2 $} (b)
			(a) edge [->] node [left] {$ p_1 $} (c)
			(b) edge[->] node [right] {$ f $} (d)
			(c) edge [->] node [below] {$ f $} (d)
			(f) edge [bend right=-30,double equal sign distance] node [below] {$  $} (b)
			(f) edge [bend left=-30, double equal sign distance] node [below] {$  $} (c)
			(f) edge [dashed,  ->] node [above] {$ \Delta_f $} (a);
		\end{tikzpicture}
	\end{center}
	\noindent where $p_i$ are the standard projection onto the $i^{th}$ component of the double product.
	Using the exchange transformation $Ex_!^*$ in \eqref{App.:_Base_Change_Ex*_!}, we have:
	\[ p_1\epf p_2^*\stackrel{ Ex_!^*(\Delta)}{\simeq} f^* f_! \]
	This give rise to a natural transformation $\gamma_f $ given by the following composition:
	\[ \gamma_f: \ \ \ f\epfs p_1\epf \stackrel{\epsilon\epfs^*(p_2)}{\longrightarrow} f\epfs p_1\epf p_2^*p_2\epfs \stackrel{Ex_!^*(\Delta)}{\longrightarrow} f\epfs f^* f_! p_2\epfs \stackrel{\eta\epfs^*(f)}{\longrightarrow} f_!p_2\epfs \]
	\noindent and this is exactly the definition of:
	\[ Ex_{\#!}: f\epfs p_1\epf  \longrightarrow f_!p_2\epfs \]
	\noindent as we have seen before.

	\noindent Applying $\gamma_f$ to $\Delta_f\epf $, we get a new natural transformation $\varphi_f$:
	\[ \varphi_f: \ \ \  f\epfs \simeq f\epfs p_1\epf \Delta_f\epf \longrightarrow f_! p_2\epfs \Delta_f\epf \]
	\noindent where we used the fact that  $p_1\circ \Delta_f=Id$.

	\begin{definition}
		With the same notation as above, given a smooth map $f: \mc X \rightarrow \mc Y$, we will denote by $\Sigma_f$ the \textit{formal suspension twist} associated to $f$ given by:
		\[ \Sigma_f:=p_2\epfs \Delta_f\epf \]
		This functor has a left adjoint that we will denote by $\Omega_f:= \Delta_f^!p_2^* $.
	\end{definition}
	
	\begin{definition}
		Given $f: \mc X \longrightarrow \mc Y $ smooth between NL-stacks, we say that $f$ satisfies \textit{ambidexterity} if:
		\[ \varphi_f: f\epfs \longrightarrow f\epf\Sigma_f \]
		\noindent is an equivalence. If moreover $\Delta_f$ satisfies homotopy purity, we say that\textit{(relative) purity} holds for $f$.
	\end{definition}
	
	Let us define some important notation for later use.
	
	\begin{definition}\label{sec6:_definition_pullback_formal_twist}
		Let $f: \mc X \rightarrow \mc Y$ and $g: \mc Z \rightarrow \mc X$ be a smooth map of NL stacks. Let $ \Delta_{g^*f} $ and $ \bar{p}_2 $ be the pullback map in the following diagram:
		
		\begin{center}
			\begin{tikzpicture}[baseline={(0,0.5)}, scale=2]
				
				\node (a) at (1,1) {$  \mc Z $};
				\node (b) at (2,1) {$ \mc Z \times_{\mc X} \mc X^2_{\mc Y} $};
				\node (c) at (3,1) {$ \mc Z $};

				\node (d) at (1,0) {$ \mc X $};
				\node (e) at (2,0) {$ \mc X^2_{\mc Y} $};
				\node (f) at (3,0) {$ \mc X $};

				\node (g) at (1.25,0.75) {$ \ulcorner $};
				\node (h) at (2.25,0.75) {$ \ulcorner $};

				\path[font=\scriptsize,>= angle 90]

				(a) edge [->] node [above] {$ \Delta_{g^*f} $} (b)
				(b) edge [->] node [above] {$ \bar{p}_2  $} (c)
				(d) edge [->] node [below] {$ \Delta_f $} (e)
				(e) edge [->] node [below] {$  p_2 $} (f)
				
				(a) edge [->] node [left] {$ g  $} (d)
				(b) edge [->] node [left] {$ G $} (e)
				(c) edge [->] node [left] {$ g $} (f);
				
			\end{tikzpicture}
		\end{center}
		We define the formal twist $\Sigma_{g^*f}$ as follows:
		\[ \Sigma_{g^*f}:= \bar{p}_2\epfs (\Delta_{g^*f})\epf  \]
		\noindent and similarly:
		\[ \Omega_{g^*f}:= \Delta_{g^*f}^!\bar{p}_2^* \]
	\end{definition}

\subsection{Formal Thom twist functoriality: after Ayoub}

Let us start with a preparatory lemma:

\begin{lemma}\label{lemma_Ex-sharp-shriek_mixed_repr_non_repr_case}
	Consider the following cartesian square of NL-stacks:
	\begin{center}
		\begin{tikzpicture}[baseline={(0,0.5)}, scale=1.5]
			\node (a) at (0,1) {$  \mc R $};
			\node (b) at (1,1) {$ \mc X$};
			\node (c) at (0,0) {$ \mc T $};
			\node (d) at (1,0) {$ \mc Y $};
			\node (e) at (0.25,0.75) {$ \ulcorner $};
			\node (e) at (0.5,0.5) {$ \Delta $};

			\path[font=\scriptsize,>= angle 90]
			
			(a) edge [->] node [above ] {$ q $} (b)
			(a) edge [->] node [left] {$ g $} (c)
			(b) edge [->] node [right] {$ f $} (d)
			(c) edge [->] node [below] {$ p $} (d);
			
		\end{tikzpicture}
	\end{center}
	\noindent Assume either that $f$ is representable lft and that $p$ is smooth (possibly non representable) or that $f$ is (possibly non representable) lft and $p$ is a smooth representable map. Then the natural transformation:
	\[ Ex_{\#!}(\Delta): p\epfs g_! \longrightarrow f_!q\epfs \]
	\noindent is an equivalence.
\end{lemma}
\begin{proof}
	Let us start with assuming that $f$ is representable lft and that $p$ is smooth (possibly non representable).
	By hypothesis we can find the following commutative cube:
	\begin{center}
		\begin{tikzpicture}[baseline={(0,1.5)}, scale=1.25]

			\node (v1) at (0,0) {$ \mc R $};
			\node (v2) at (2,0) {$ \mc X $};
			\node (v3) at (1,-1) {$ \mc T $};
			\node (v4) at (3,-1) {$ \mc Y $};
			
			\node (v5) at (0,2) {$ R $};
			\node (v6) at (2,2) {$  X $};
			\node (v7) at (1,1) {$ T $};
			\node (v8) at (3,1) {$ Y $};

			% Draw arrows in the back
			\path[font=\scriptsize,>= angle 90]
			
			(v1) edge [->] node [ right=2mm, below=0.1mm ] {$ q $} (v2)
			(v2) edge [->] node [ right=0.5mm, above=0.5mm ] {$ f $} (v4)
			
			(v6) edge [->] node [right=2mm, below=1mm] {$ x $} (v2)
			
			(v5) edge [->] node [ above ] {$ \bar{q} $} (v6)
			(v6) edge [->] node [ right ] {$ \bar{f} $} (v8)
			
			% Draw thick white lines on the front
			(v7) edge [-,line width=1mm, white] node [] {} (v3)
			(v7) edge [-, line width=1mm, white] node [] {} (v8)

			%Draw arrows in the front
			(v5) edge [->] node [ left ] {$ r $} (v1)
			(v7) edge [->] node [ left=2mm, above=1mm ] {$ t $} (v3)
			(v8) edge [->] node [ left ] {$ y $} (v4)
			
			(v5) edge [->] node [ left, below ] {$ \bar{g} $} (v7)
			(v7) edge [->] node [ below=2mm, left=0.01mm ] {$ \bar{p} $} (v8)
			
			(v1) edge [->] node [ left=0.5mm ] {$ g $} (v3)
			(v3) edge [->] node [ below ] {$ p $} (v4);

		\end{tikzpicture}
	\end{center}
	\noindent where $r,t,x,y$ are NL-atlases and such that the following are cartesian squares:\\
	\begin{minipage}{0.3\textwidth}
		\begin{center}
			\begin{tikzpicture}[baseline={(0,0.5)}, scale=1.5]
				\node (a) at (0,1) {$   R $};
				\node (b) at (1,1) {$  X$};
				\node (c) at (0,0) {$  T $};
				\node (d) at (1,0) {$  Y $};
				\node (e) at (0.25,0.75) {$ \ulcorner $};
				\node (e) at (0.5,0.5) {$ \Delta' $};

				\path[font=\scriptsize,>= angle 90]
				
				(a) edge [->] node [above ] {$ \bar q $} (b)
				(a) edge [->] node [left] {$ \bar g $} (c)
				(b) edge [->] node [right] {$ \bar f $} (d)
				(c) edge [->] node [below] {$ 	\bar p $} (d);
				
			\end{tikzpicture}
		\end{center}
	\end{minipage}
	\hfill 
	\begin{minipage}{0.3\textwidth}
		\begin{center}
			\begin{tikzpicture}[baseline={(0,0.5)}, scale=1.5]
				\node (a) at (0,1) {$   X $};
				\node (b) at (1,1) {$  Y $};
				\node (c) at (0,0) {$ \mc X$};
				\node (d) at (1,0) {$ \mc Y $};
				\node (e) at (0.25,0.75) {$ \ulcorner $};
				\node (e) at (0.65,0.5) {$ \Delta_{XY} $};

				\path[font=\scriptsize,>= angle 90]
				
				(a) edge [->] node [above ] {$ \bar{f} $} (b)
				(a) edge [->] node [left] {$ x $} (c)
				(b) edge [->] node [right] {$ y $} (d)
				(c) edge [->] node [below] {$ f $} (d);
				
			\end{tikzpicture}
		\end{center}
	\end{minipage}
	\hfill
	\begin{minipage}{0.3\textwidth}
		\begin{center}
			\begin{tikzpicture}[baseline={(0,0.5)}, scale=1.5]
				\node (a) at (0,1) {$   R $};
				\node (b) at (1,1) {$  T $};
				\node (c) at (0,0) {$ \mc R $};
				\node (d) at (1,0) {$ \mc T $};
				\node (e) at (0.25,0.75) {$ \ulcorner $};
				\node (e) at (0.65,0.5) {$ \Delta_{RT} $};

				\path[font=\scriptsize,>= angle 90]
				
				(a) edge [->] node [above ] {$ \bar{g} $} (b)
				(a) edge [->] node [left] {$ r $} (c)
				(b) edge [->] node [right] {$ t $} (d)
				(c) edge [->] node [below] {$ g $} (d);
				
			\end{tikzpicture}
		\end{center}
	\end{minipage}\\
	
	\noindent From the cube above, for $D(-)=SH(-)$, we get another cube that we can decompose as:
	
	\begin{center}
		\begin{tikzpicture}[baseline={(0,0.5)}, scale=2]
			\node (v1) at (0,1) {$ D(X) $};
			\node (v2) at (1,1) {$ D(R) $};
			\node (v3) at (2,1) {$ D(\mc R) $};
			\node (v4) at (2,2) {$ D(R) $};
			\node (v5) at (3,2) {$ D(X) $};
			\node (v6) at (3,1) {$ D(\mc X) $};
			\node (v7) at (4,1) {$ D(X) $};
			\node (v8) at (4,0) {$ D(Y) $};
			\node (v9) at (3,0) {$ D(\mc Y) $};
			\node (v10) at (3,-1) {$ D(Y) $};
			\node (v11) at (2,-1) {$ D(T) $};
			\node (v12) at (2,0) {$ D(\mc T) $};
			\node (v13) at (1,0) {$ D(T) $};
			\node (v14) at (0,0) {$ D(Y) $};
			
			\node (c1) at (0.5,0.5) { \rotatebox{-45}{$ \overset{\scriptsize Ex_{\#!}(\Delta')}{\Leftarrow} $} };
			\node (c2) at (1.5,0.5) {\rotatebox{45}{$ \overset{\scriptsize Ex_{\#!}(\Delta_{RT})}{\Rightarrow} $}};
			\node (c3) at (2.5,0.5) {\rotatebox{45}{$ \overset{\scriptsize Ex_{\#!}(\Delta)}{\Rightarrow} $}};
			\node (c4) at (3.5,0.5) {\rotatebox{-45}{$ \overset{\scriptsize Ex_{\#!}(\Delta_{XY})}{\Leftarrow} $}};
			
			\node (c5) at (2.5,1.5) { $ \circlearrowleft $};
			\node (c6) at (2.5,-0.5) { $ \circlearrowleft $};

			\path[font=\scriptsize,>= angle 90]
			% Horizontal
			(v4) edge [->] node [above ] {$ \bar{q}\epfs $} (v5)
			
			(v1) edge [<-] node [above ] {$ \bar{q}\epfs $} (v2)
			(v2) edge [->] node [above ] {$r\epfs $} (v3)
			(v3) edge [->] node [above ] {$ q\epfs $} (v6)
			(v6) edge [<-] node [above ] {$ x\epfs $} (v7)
			
			(v14) edge [<-] node [below ] {$ \bar{p}\epfs $} (v13)
			(v13) edge [->] node [below ] {$ t\epfs $} (v12)
			(v12) edge [->] node [below ] {$ p\epfs $} (v9)
			(v9) edge [<-] node [below ] {$ y\epfs $} (v8)
			
			(v11) edge [->] node [below ] {$ \bar{p}\epfs $} (v10)
			
			%Vertical
			
			(v4) edge [->] node [left ] {$ r\epfs $} (v3)
			(v5) edge [->] node [right ] {$ x\epfs $} (v6)
			
			(v1) edge [->] node [left ] {$ \bar{f}\epf  $} (v14)
			(v2) edge [->] node [left ] {$ \bar{g}\epf $} (v13)
			(v3) edge [->] node [left ] {$ g_! $} (v12)
			(v6) edge [->] node [left ] {$ f_! $} (v9)
			(v7) edge [->] node [right ] {$ \bar{f}\epf  $} (v8)
			
			(v12) edge [<-] node [left ] {$ t\epfs $} (v11)
			(v9) edge [<-] node [right ] {$ y\epfs $} (v10);
			
		\end{tikzpicture}
	\end{center}
	Using \cite[5.5.3.4]{HTT} and the conservativity of $r^*$, it is not hard to see that:
	\[ Ex_{\#!}(\Delta): p\epfs g_! \longrightarrow f_!q\epfs  \]
	\noindent is an equivalence if and only if:
	\[ Ex_{\#!}(\Delta)r\epfs: p\epfs g_!r\epfs \longrightarrow f_!q\epfs r\epfs  \]
	\noindent is an equivalence. Let us check that $Ex_{\#!}(\Delta)r\epfs$ is indeed an equivalence. Denote by $\Delta^{RX\mc X}_{TY\mc Y}$ the cartesian square obtained joining $\Delta'$ and $\Delta_{XY}$, and similarly denote as $\Delta^{R\mc R\mc X}_{T\mc T \mc Y}$ the cartesian square obtained from $\Delta_{RT}$ and $\Delta$. Then we have the following commutative diagram:
	
	\begin{center}
		\begin{tikzpicture}[baseline={(0,0.5)}, scale=2.25]
			\node (a1) at (0,1.75) {$   p\epfs g_! r\epfs $};
			\node (a2) at (0,0.5) {$  p\epfs t\epfs \bar{g}_! $};
			\node (a3) at (0,0) {$ y\epfs \bar{p}\epfs \bar{g}\epf  $};
			\node (b1) at (2,1.75) {$ f_! q\epfs r\epfs $};
			\node (b2) at (2,1.25) {$ f_! x\epfs \bar{q}\epfs $};
			\node (b3) at (2,0) {$ y\epfs \bar{f}\epf \bar{q}\epfs $};

			\path[font=\scriptsize,>= angle 90]
			
			(a1) edge [->] node [above] {$ Ex_{\#!}(\Delta)r\epfs $} (b1)
			(a3) edge [->] node [above] {$ \sim $} node [below] {$ y\epfs Ex_{\#!}(\Delta') $} (b3)

			(a1) edge [<-] node [right] {\rotatebox{90}{$ \sim $}} node [left] {\rotatebox{90}{$ p\epfs Ex_{\#!}(\Delta') $}} (a2)
			(a2) edge [double equal sign distance] node [] {$  $} (a3)
			(b1) edge [double equal sign distance] node [] {$  $} (b2)
			(b2) edge [<-] node [left] {\rotatebox{90}{$ \sim $}} node [right] {\rotatebox{90}{$ Ex_{\#!}(\Delta_{XY})\bar{q}\epfs $}} (b3)
			
			(a2) edge [->] node [above] {\rotatebox{20}{$  $}} node [below=-3mm] {\rotatebox{30}{$ Ex_{\#!}(\Delta^{R\mc R\mc X}_{T\mc T \mc Y}) $}} (b1)
			(a3) edge [->] node [left] {\rotatebox{40}{$ \sim $}} node [below=-3mm] {\rotatebox{30}{$ Ex_{\#!}(\Delta^{RX\mc X}_{TY\mc Y}) $}} (b2);
		\end{tikzpicture}
	\end{center}
	Hence we can rewrite $Ex_{\#!}(\Delta)r\epfs $ as:
	\[ y\epfs Ex_{\#!}(\Delta'): y\epfs \bar{p}\epfs \bar{g}\epf \longrightarrow y\epfs \bar{f}\epf \bar{q}\epfs \]
	\noindent Since $Ex_{\#!}(\Delta')$ comes from a cartesian square in ${Sch}$, $ Ex_{\#!}(\Delta)r\epfs $ is indeed an equivalence by \cref{Sec.4.4:_Ex-sharp-!_equiv_schematic_case} and we are done.\\
	
	The proof of other case, where $f$ is a generic lft map and $p$ is assumed to be smooth and representable, is totally analogous and left to the reader.
\end{proof}

Let us consider the following commutative diagram in $\oocatname{ASt}^{NL,lft}$:

\begin{equation}\label{Ayoub_diagram}
	\begin{tikzpicture}[baseline={(0,1.5)}, scale=2]
		
		\node (a0) at (0,2) {$ \mc X $};
		
		\node (a) at (0,1) {$ \mc W_1 $};
		\node (b) at (1, 1) {$ \mc W_3 $};
		\node (c)  at (0,0) {$  \mc Y $};
		\node (d) at (1,0) {$ \mc W_2 $};
		
		\node (d0) at (2,0) {$ \mc Z $};
		
		\node (e) at (0.2,0.75) {$ \ulcorner $};
		
		%\node (g) at (0.5,0.5) {$ \Delta $};

		\path[font=\scriptsize,>= angle 90]

		(a) edge [->] node [above ] {$ \omega_{13} $} (b)
		(a) edge [->] node [left] {$ p_{1}  $} (c)
		(b) edge[->] node [right] {$ \omega_{32} $} (d)
		(c) edge [->] node [below] {$ f_2 $} (d)
		
		(a0) edge [->] node [right] {$ f_3 $} (b)
		(a0) edge [->] node [left] {$ f_1 $} (a)
		
		(d) edge [->] node [below] {$ p_{2} $} (d0)
		(b) edge [->] node [above, right] {$ p_{3} $} (d0);
	\end{tikzpicture}
\end{equation}

\noindent where $\mc W_1\simeq \mc Y\times_{\mc W_2} \mc W_3$.

\begin{lemma}[Composition of Twists]\label{Ayoub_triangle_lemma}
	Consider the diagram \eqref{Ayoub_diagram} and assume that all the maps are lft and that $p_1,p_2,p_3$ are smooth. Moreover assume that either $f_2$ or $\omega_{32}$ is representable. Then we have a natural equivalence:
	\[ \mr{Th}(f_3,p_3)\simeq \mr{Th}(f_2,p_2)\mr{Th}(f_1,p_1) \]
\end{lemma}
\begin{proof}
	The equivalence follows from functorialities and exchange transformations between the functors involved:
	
	\begin{center}
		\begin{tikzpicture}[baseline={(0,1.5)}, scale=2]
			
			\node (a0) at (0,2) {$ \bullet $};
			
			\node (a) at (0,1) {$ \bullet $};
			\node (b) at (1, 1) {$ \bullet $};
			\node (c)  at (0,0) {$ \bullet $};
			\node (d) at (1,0) {$\bullet $};
			
			\node (d0) at (2,0) {$ \bullet $};
			
			\node (e) at (0.48,0.53) {\rotatebox{45}{$ \overset{Ex_{\#!}}{\Leftarrow} $}};
			\node (e1) at (.25,1.45) {$ \circlearrowright $};
			\node (e2) at (1.45,0.25) {$ \circlearrowright $};
			%\node (g) at (0.5,0.5) {$ \Delta $};

			\path[font=\scriptsize,>= angle 90]

			(a) edge [->] node [above ] {$ \omega_{13}\epf $} (b)
			(a) edge [->] node [left] {$ p_{1}\epfs  $} (c)
			(b) edge[->] node [right] {$ \omega_{32}\epfs $} (d)
			(c) edge [->] node [below] {$ f_2\epf $} (d)
			
			(a0) edge [->] node [right] {$ f_3\epf  $} (b)
			(a0) edge [->] node [left] {$ f_1\epf  $} (a)
			
			(d) edge [->] node [below] {$ p_{2}\epfs $} (d0)
			(b) edge [->] node [above, right] {$ p_{3}\epfs $} (d0);
		\end{tikzpicture}
	\end{center}
	Namely we have:
	\[ \mr{Th}(f_2,p_2)\mr{Th}(f_1,p_1) = p_2\epfs f_2\epf p_1\epfs f_1\epf \overset{Ex_{\#!}}{\longleftarrow} p_2\epfs \omega_{32}\epfs \omega_{13}\epf f_1\epf \simeq p_3\epfs f_3\epf = \mr{Th}(f_3,p_3)  \]
	\noindent and by our assumptions $Ex_{\#!}: \omega_{32}\epfs \omega_{13}\epf \overset{}{\rightarrow} f_1\epf p_1\epfs $ is an equivalence, as proved in\Cref{lemma_Ex-sharp-shriek_mixed_repr_non_repr_case}, therefore we can conclude.
	
\end{proof}

\begin{lemma}\label{Formal_twists_composition}
	Let $ g: \mc Y \rightarrow \mc Z $ and $ f: \mc X \rightarrow \mc Y $ be two smooth maps in $\oocatname{ASt}^{\leq 1}$, and assume $f$ is representable. 
	Then:
	\[ \Sigma_{f\circ g} \simeq \Sigma_{f^*g} \Sigma_f \]
	and:
	\[ \Omega_{f\circ g}\simeq \Omega_g\Omega_{g^*f}\]
	
\end{lemma}

\begin{proof}
	The two claims in $(i)$ are duals to each other, so it is enough to prove the first of them. Consider the following diagram:
	\begin{center}
		\begin{tikzpicture}[baseline={(0,1.5)}, scale=2]
			
			\node (a0) at (0,2) {$ \mc Z $};
			
			\node (a) at (0,1) {$ \mc Z^2_{\mc X} $};
			\node (b) at (1, 1) {$ \mc Z^2_{\mc Y} $};
			\node (c)  at (0,0) {$  \mc Z $};
			\node (d) at (1,0) {$ \mc Z\times_{\mc X} \mc X^2_{\mc Y} $};
			
			\node (d0) at (2,0) {$ \mc Z $};
			
			\node (e) at (0.2,0.75) {$ \ulcorner $};
			
			%\node (g) at (0.5,0.5) {$ \Delta $};

			\path[font=\scriptsize,>= angle 90]

			(a) edge [->] node [above ] {$  $} (b)
			(a) edge [->] node [left] {$ p_{g}  $} (c)
			(b) edge[->] node [right] {$  $} (d)
			(c) edge [->] node [below] {$ \Delta_{g^*f} $} (d)
			
			(a0) edge [->] node [right] {$ \Delta_{f\circ g} $} (b)
			(a0) edge [->] node [left] {$ \Delta_g $} (a)
			
			(d) edge [->] node [below] {$ \bar{p}_g $} (d0)
			(b) edge [->] node [above, right] {$ p_{f\circ g} $} (d0);
		\end{tikzpicture}
	\end{center}
	
	The isomorphism $ \mc Z^2_{\mc X}\simeq \mc Z \times_{\mc Z\times_{\mc X}\mc X^2_{\mc Y}} \mc Z^2_{\mc Y} $ follows from the following diagram of cartesian squares:
	
	\begin{center}
		\begin{tikzpicture}[baseline={(0,1.5)}, scale=1.5]
			
			\node (a3) at (0,3) {$ \mc Z_{\mc X}^2  $};
			\node (a2) at (0,2) {$ \mc Z_{\mc Y}^2 $};
			\node (a1) at (0,1) {$ \mc Z\times_{\mc X} \mc X^2_{\mc Y}  $};
			\node (a0) at (0,0) {$\mc Z $};
			
			\node (b3) at (1,3) {$ \mc Z  $};
			\node (b2) at (1,2) {$ \mc Z\times_{\mc X} \mc X^2_{\mc Y} $};
			\node (b1) at (1,1) {$ \mc X^2_{\mc Y} $};
			\node (b0) at (1,0) {$ \mc X  $};

			\node (c2) at (2,2) {$ \mc Z $};
			\node (c1) at (2,1) {$ \mc X $};
			\node (c0) at (2,0) {$ \mc Y $};
			
			\node (t1) at (0.25,2.75) {$ \ulcorner $};
			\node (t2) at (0.25,1.75) {$ \ulcorner $};
			\node (t3) at (0.25,0.75) {$ \ulcorner $};
			\node (t4) at (1.25,1.75) {$ \ulcorner $};
			\node (t5) at (1.25,0.75) {$ \ulcorner $};

			%\node (g) at (0.5,0.5) {$ \Delta $};

			\path[font=\scriptsize,>= angle 90]

			%horizontal
			(a3) edge [->] node [above ] {$  $} (b3)
			(a2) edge [->] node [above ] {$  $} (b2)
			(a1) edge [->] node [above ] {$  $} (b1)
			(a0) edge [->] node [above ] {$  $} (b0)
			
			(b2) edge [->] node [above ] {$  $} (c2)
			(b1) edge [->] node [above ] {$  $} (c1)
			(b0) edge [->] node [above ] {$  $} (c0)
			
			%vertical
			(a3) edge [->] node [above ] {$  $} (a2)
			(a2) edge [->] node [above ] {$  $} (a1)
			(a1) edge [->] node [above ] {$  $} (a0)
			
			(b3) edge [->] node [right ] {$ \Delta_{g^*f} $} (b2)
			(b2) edge [->] node [above ] {$  $} (b1)
			(b1) edge [->] node [above ] {$  $} (b0)
			
			(c2) edge [->] node [above ] {$  $} (c1)
			(c1) edge [->] node [above ] {$  $} (c0);
		\end{tikzpicture}
	\end{center}
	Indeed, looking at the top-left-most square and at the outer square given by composing all the three squares on the left part of the diagram, we actually see that we must have $\mc Z^2_{\mc X}\simeq \mc Z \times_{\mc Z\times_{\mc X}\mc X^2_{\mc Y}} \mc Z^2_{\mc Y}$ as we wanted. Moreover $ \Delta_{g^*f} $ and the map $ \mc Z^2_{\mc Y}\rightarrow \mc Z\times_{\mc X} \mc X^2_{\mc Y} $ are representable, since they are pullbacks of $\Delta_f$ and $g$ respectively. 
	
	Then we can apply \cref{Ayoub_triangle_lemma}, recalling the definitions of the formal twists, and obtain:
	
	\[ \Sigma_{g^*f}\Sigma_g=\mr{Th}(\Delta_{g^*f}, \bar{p}_g)\mr{Th}(\Delta_g,p_g)\overset{\text{(\Cref{Ayoub_triangle_lemma})}}{\underset{\sim}{\longrightarrow}} \mr{Th}(\Delta_{f\circ g},p_{f\circ g})=\Sigma_{f\circ g} \]
	
\end{proof}

\subsection{Homotopy Purity}

The goal for this section is to prove homotopy purity for a given representable map between smooth algebraic $NL$-stacks. As in many previous instances, the central idea is again to reduce our assertions about stacks to established results concerning schemes. 
%\textcolor{red}{We take the opportunity here to thank J. Ayoub for pointing out to us that, at least Zariski locally, one can always apply the classical deformation to the normal cone along the diagonals of schematic maps, as these diagonals are always locally closed}.\\
Our stepping stone for this section will be the following:

\begin{lemma}\label{schematic_HP_for_diagonal_after_AGV}
	Let $r: \mc X \rightarrow \mc Y$ be a smooth representable map between NL-stacks. Then $\Delta_r$ satisfies homotopy purity, i.e. we have that the natural maps $ \Pi_1(\Delta_r,p_2) $ and $\Pi_0(\Delta_r,p_2)$ of \Cref{sec.6:_definition_HP} are both equivalences, giving us an identification:
	\[ \Sigma_r \simeq \Sigma^{\mb L_r}  \]
\end{lemma}
\begin{proof}
	By NL-separation, we can assume that $r$ is actually a smooth map between schemes. For any map between schemes the diagonal is an immersion by \cite[\href{https://stacks.math.columbia.edu/tag/01KJ}{Tag 01KJ}]{stacks-project}, i.e. it is locally closed. This in particular implies that $\mr{Def}_{\Delta_r}$ is represented by a scheme. Therefore, using the same convention as in \Cref{Notation:_HP}, we have the canonical roof:
	\[ \Sigma_r:=\mr{Th}(\Delta_r,p_2) \overset{\Pi_1(\Delta_r,p_2)}{\longrightarrow} (r_S)\epfs \mr{Th}(\widehat{\Delta_r},\hat{p}_2)r_X^* \overset{\Pi_0(\Delta_r,p_2)}{\longleftarrow} q_{\#}\mr{Th}((\Delta_r)_0,(p_2)_0)=: q_{\#}\Sigma^{\mc N_{\Delta_r}} \]
	\noindent where $p_2: \mc X^2_{\mc Y}\rightarrow \mc Y $ is the projection onto the second coordinate and $ \mc N_{\Delta_r}:=\mb L_{\Delta_r}[-1] $ is the conormal sheaf of $\Delta_r$. Recall that $\mc N_{\Delta_r}\simeq \mb L_r$ using the standard fiber sequence of cotangent complexes for the composition $p_2\circ \Delta_r=Id$. Working Zariski locally, we can assume $r$ is separated, or equivalently that $\Delta_r$ is closed. In this case we already know that $ \Pi_1(\Delta_r,p_2) $ and $ \Pi_0(\Delta_r,p_2) $ are equivalences, as proved e.g. in \cite[Proposition 5.7]{Hoyois_Equiv_Six_Op}, and we are done.
\end{proof}

\begin{proposition}[Representable Purity]\label{Repr_lft_Ayoub_Purity_after_AGV}
	Let $r: \mc X \rightarrow \mc Y$ be a smooth representable map between NL-stacks. Then we have the following natural equivalence:
	
	\[ r_{\#} \overset{\sim}{\longrightarrow} r\epf \Sigma_r\simeq r\epf \Sigma^{\mb L_r}  \]
	
\end{proposition}
\begin{proof}
	Choosing a NL atlas $y: Y \rightarrow \mc Y$, we can use the conservativity of $y^*$ to reduce ourself to the case where $r$ is actually a smooth map of schemes. Then mimicking the same arguments in the proof of \cite[Theorem 4.4.29]{AGV_Rigid6FF}, one shows that:
	\[ \varphi_r: r_{\#} \longrightarrow r\epf \Sigma_r \]
	\noindent is an equivalence. We then conclude applying \Cref{schematic_HP_for_diagonal_after_AGV}.
\end{proof}

\begin{remark}
	The previous lemma and proposition are an immediate consequence of the same, already known statements in the schematic case. We just reported it here, in this slightly generalised form, for easy reference in what will follow next.
\end{remark}

Before proving the main theorem of this section, we will need a technical lemma. Let $f: X \rightarrow Y$ be a map between smooth $S$-schemes and suppose that $f$ is a global complete intersection  (from here on simply \textit{gci}), i.e. that it factors as a regular closed immersion followed by a smooth map:

\[ X \overset{\iota}{\into} M \overset{\mu}{\rightarrow} Y \]
The map $\mu$ will give rise to natural maps between deformation spaces and normal bundles:
\[ \hat{\pi}: \mr{Def}_{\iota} \rightarrow \mr{Def}_f \]
\[ \pi_0: N_{\iota} \rightarrow N_{f} \]

\begin{lemma}\label{Key_Lemma:_schematic_HP+}
	Let $f$ be the g.c.i. map in $ {Sm}_S$ given by the composition $ X  \overset{\iota}{\into} M \overset{\mu}{\rightarrow}  Y$, for $\iota$ a closed immersion and $\mu$ a smooth map. Let $p,q$ be the structure maps of $X$ and $Y$. Using the same conventions as in \Cref{Notation:_HP}, from the maps:
	
	\begin{center}
		\begin{tikzpicture}[baseline={(0,0.5)}, scale=2.5]
			\node (a) at (0,0.5) {$  X $};
			\node (d) at (0.5,0.5) {$  M $};
			\node (b) at (1,0.5) {$  Y $};
			\node (c) at (0.5,0) {$ S $};
			
			\node (a1) at (2,0.5) {$ \A^1_X $};
			\node (d1) at (2.5,0.5) {$ \mr{Def}_{\iota} $};
			\node (b1) at (3,0.5) {$ \mr{Def}_f $};
			
			\node (c1) at (2.5,0) {$ S $};
			
			\node (a2) at (4,0.5) {$ X $};
			\node (d2) at (4.5,0.5) {$ N_{\iota} $};
			\node (b2) at (5,0.5) {$ N_f $};
			\node (c2) at (4.5,0) {$ S $};
			
			\path[font=\scriptsize,>= angle 90]
			
			(a) edge [closed] node [below ] {$ \iota $} (d)
			(d) edge [->] node [below ] {$ \mu $} (b)
			(a) edge [bend left=30,->] node [above ] {$ f $} (b)
			
			(a) edge [->] node [left] {$q$} (c)
			(d) edge [->] node [left] {$  $} (c)
			(b) edge [->] node [right] {$p$} (c)

			([yshift=-2.5mm]b.east) edge [shorten >=2mm,shorten <=2mm,->] node [below] {$ i_1 $} ([yshift=-2.5mm]a1.west)

			(a1) edge [bend left=30,->] node [above ] {$ \hat{f} $} (b1)
			(a1) edge [->] node [below ] {$ \hat{\iota} $} (d1)
			(d1) edge [->] node [below ] {$ \hat{\pi} $} (b1)
			
			(a1) edge [->] node [left] {$ $} (c1)
			(d1) edge [->] node [left] {$ $} (c1)
			(b1) edge [->] node [right] {$ r_S\circ \hat p $} (c1)
			
			([yshift=-2.5mm]b1.west) edge [shorten >=5mm,shorten <=15mm,<-] node [below=2.5mm, right=2.5mm] {$ i_0 $} ([yshift=-2.5mm]a2.east)

			(a2) edge [bend left=30,->] node [above ] {$ f_0 $} (b2)
			(a2) edge [->] node [below ] {$ \iota_0 $} (d2)
			(d2) edge [->] node [below ] {$ \pi_0 $} (b2)
			
			(a2) edge [->] node [left] {$ q $} (c2)
			(d2) edge [->] node [left] {$  $} (c2)
			(d2) edge [->] node [left] {$  $} (c2)
			(b2) edge [->] node [right] {$ p_0 $} (c2);
			
		\end{tikzpicture}
	\end{center}
	
	\noindent we get a commutative diagram of the form:
	\begin{center}
		\begin{tikzpicture}[baseline={(0,1)}, scale=2]
			
			\node (a) at (0,1) {$  \mr{Th}(\iota,p\circ \mu)\Sigma^{-\iota^*\mb L_{\mu}} $};
			
			\node (b) at (2.5,1) {$ r_{ S}\epfs\mr{Th}(\hat{\iota},\hat{p}\circ \hat\pi)r_{X}^*\Sigma^{-\iota^*\mb L_{\mu}}  $};
			
			\node (c) at (5,1) {$  \mr{Th}(\iota_0,p_0\circ \pi_0)\Sigma^{-\iota^*\mb L_{\mu}} $};
			
			\node (d) at (0,0) {$  \mr{Th}(f,p) $};
			
			\node (e) at (2.5,0) {$ r_{ S}\epfs\mr{Th}(\hat{f},\hat{p})r_{X}^*  $};
			
			\node (f) at (5,0) {$  \mr{Th}(f_0,p_0) $};

			%\node (g) at (1.25,0.75) {$ \ulcorner $};
			%\node (h) at (2.25,0.75) {$ \ulcorner $};

			\path[font=\scriptsize,>= angle 90]

			(a) edge [->] node [above] {$ \Pi_1(\iota,p\circ \mu) $} (b)
			(c) edge [->] node [above] {$ \Pi_0(\iota,p\circ \mu) $} (b)
			(d) edge [->] node [below] {$ \Pi_1(f,p) $} (e)
			(f) edge [->] node [below] {$ \Pi_0(f,p) $} (e)
			(a) edge [->] node [left] {\rotatebox{-90}{$ \sim  $}} (d)
			
			(b) edge [->] node [left] {\rotatebox{-90}{$ \sim  $}} (e)
			
			(c) edge [->] node [right] {\rotatebox{-90}{$ \sim  $}} (f);
			
		\end{tikzpicture}
	\end{center}
\end{lemma}

\begin{proof}
	Let us start by looking at the cotangent complexes of $\mu,\hat\pi,\pi_0$. Let $q_0: N_f \rightarrow X$ be the projection map. Using the standard fiber sequences:
	
	\[ \iota^* \mb L_{\mu} \rightarrow \mb L_{f} \rightarrow \mb L_{\iota} \]
	\[ \pi_0^*\mb L_{q_0} \rightarrow \mb L_{\bigslant{N_{\iota}}{X}} \rightarrow \mb L_{\pi_0} \]
	\[ \hat\pi^* \mb L_{\hat p} \rightarrow \mb L_{\hat p \circ \hat\pi} \rightarrow \mb L_{\hat\pi } \]
	
	\noindent and the fact that $\mr{Def}_{\iota}$ and $\mr{Def}_f $ are flat over $\A^1_M$ and $\A^1_Y $ respectively, one can deduce that:
	\begin{equation}\label{eq_ctg_cplx_HP_proposition}
		\begin{split}
			\mb L_{\hat\pi}\simeq \hat\iota^* r_X^* \iota^*\mb L_{\mu} \\
			\mb L_{\pi_0}\simeq \pi_0^*q_0^* \iota^*\mb L_{\mu}
		\end{split}
	\end{equation}
	
	This in particular tells us that $\pi,\hat\pi,\pi_0$ are all smooth maps. Then by \Cref{Repr_lft_Ayoub_Purity_after_AGV}, we have natural equivalences:
	\begin{equation}\label{eq:_FP_for_HP}
		\varphi_{\mu}: \mu\epfs \overset{\sim}{\longrightarrow} \mu\epf \Sigma_{\mu} \ \ \ \ \ \ \ \ \ \ \ \ \ \ \ \varphi_{\hat \pi}: \hat \pi\epfs \overset{\sim}{\longrightarrow} \hat\pi\epf\Sigma_{\hat\pi}  \ \ \ \ \ \ \ \ \ \ \ \ \ \ \ \varphi_{\pi_0}: \pi_0\epfs  \overset{\sim}{\longrightarrow} \pi_0\epf \Sigma_{\pi_0}
	\end{equation}
	\noindent Denote by $ p_{2,\mu} $ (resp. $p_{2,\hat\pi}$, $p_{2,\pi_0}$) the projection to the second coordinate of the fiber product of $\mu$ (resp. $\hat\pi$, $\pi_0$). By \Cref{schematic_HP_for_diagonal_after_AGV}, we also have identifications:
	
	\begin{equation}\label{eq:_closed_HP_for_repr-lft_HP}
		\begin{split}
			\Pi_0^{-1}(\Delta_{\mu},p_{2,\mu})\Pi_1(\Delta_{\mu},p_{2,\mu}): \Sigma_{\mu} \overset{\sim}{\longrightarrow} \Sigma^{\mb L_{\mu}} \\
			\Pi_0^{-1}(\Delta_{\hat\pi},p_{2,\hat\pi})\Pi_1(\Delta_{\hat\pi},p_{2,\hat\pi}): \Sigma_{\hat\pi} \overset{\sim}{\longrightarrow} \Sigma^{\mb L_{\hat\pi}}\\
			\Pi_0^{-1}(\Delta_{\pi_0},p_{2,\pi_0})\Pi_1(\Delta_{\pi_0},p_{2,\pi_0}):\Sigma_{\pi_0} \overset{\sim}{\longrightarrow} \Sigma^{\mb L_{\pi_0}}
		\end{split}
	\end{equation}
	
	To simplify the notation let us denote from now on by $D_{\iota}$ and $D_f$ the deformation spaces relative to  $\iota$ and $f$ respectively; moreover denote $ SH(-) $ just by $D(-)$.\\
	
	Back to our main task, we will construct and show the commutativity of the square:
	
	\begin{equation}\label{eq:_commutativity_HP_pt1}
		\begin{tikzpicture}[baseline={(0,1)}, scale=2]
			
			\node (a) at (0,1) {$  \mr{Th}(\iota,p\circ \mu)\Sigma^{-\iota^*\mb L_{\mu}} $};
			
			\node (b) at (2.5,1) {$ r_{ S}\epfs\mr{Th}(\hat{\iota},\hat{p}\circ \hat\pi)r_{X}^*\Sigma^{-\iota^*\mb L_{\mu}}  $};
			
			%	\node (c) at (5,1) {$  \mr{Th}(\iota_0,p_0\circ \pi_0)\Sigma^{-\mb L_{\pi}} $};
			
			\node (d) at (0,0) {$  \mr{Th}(f,p) $};
			
			\node (e) at (2.5,0) {$ r_{ S}\epfs\mr{Th}(\hat{f},\hat{p})r_{X}^*  $};
			
			%	\node (f) at (5,0) {$  \mr{Th}(f_0,p_0) $};

			%\node (g) at (1.25,0.75) {$ \ulcorner $};
			%\node (h) at (2.25,0.75) {$ \ulcorner $};

			\path[font=\scriptsize,>= angle 90]

			(a) edge [->] node [above] {$ \Pi_1(\iota,p\circ \mu) $} (b)
			%	(c) edge [->] node [above] {$ \Pi_0(\iota,p\circ \pi) $} (b)
			(d) edge [->] node [below] {$ \Pi_1(f,p) $} (e)
			%	(f) edge [->] node [below] {$ \Pi_0(f,p) $} (e)
			(a) edge [->] node [left] {\rotatebox{-90}{$ \sim  $}} (d)
			
			(b) edge [->] node [left] {\rotatebox{-90}{$ \sim  $}} (e);
			
			%	(c) edge [->] node [right] {\rotatebox{-90}{$ \sim  $}} (f);
			
		\end{tikzpicture}
	\end{equation}
	\noindent by filling the following diagram below:
	
	\begin{equation}\label{eq:_HP__BIG_diagram_with_Double_Def}
		\begin{tikzpicture}[baseline={(0,2.5)}, scale=1.25]
			\node (v1) at (0,0) {$  D(\A^1_X) $};
			\node (v4) at (2,1) {$ D(X) $};
			\node (v2) at (3,-1) {$  D(D_f) $};
			\node (v5) at (5,0) {$ D(Y) $};
			\node (v7) at (0,2) {$  D(\A^1_X) $};
			\node (v10) at (2,3) {$  D(X) $};
			\node (v8) at (3,1) {$  D(D_{\iota}) $};
			\node (v11) at (5,2) {$ D(M) $};
			\node (v3) at (6,-2) {$  D(S) $};
			\node (v6) at (8,-1) {$ D(S)$};
			%\node (v12) at (8,1) {$  D(S) $};
			%\node (v9) at (6,0) {$ D(S) $};
			
			%\node (v17) at (-1,2) {$  D(X) $};

			%back wing vertices
			\node (v15) at (3.5,3.75) {$ D(X)   $};
			\node (v16) at (6.5,2.75) {$ D(M)   $};
			%front wing
			\node (v13) at (-1.5,1.25) {$ D(\A^1_X) $};
			\node (v14) at (1.5,0.25) {$ D(D_{\iota}) $};

			\path[font=\scriptsize,>= angle 90]
			%Draw arrow back wing
			(v15) edge [->] node [ left ] {$ \Sigma^{\iota^*\mb L_{\mu}} $} (v10)
			(v16) edge [->] node [ above ] {$ \Sigma_{\mu} $} (v11)
			(v15) edge [->] node [ above ] {$ \iota_{!} $} (v16)
			(v16) edge [->] node [ right ] {$ \mu{\#} $} (v5)
			
			% Draw arrows in the back
			
			(v1) edge [->] node [ left=2mm, above=0.1mm ] {$ i_{1,X}^* $} (v4)
			(v4) edge [->] node [ right=0.5mm, above=0.5mm ] {$ f_! $} (v5)
			(v10) edge [double equal sign distance] node [right=2mm, below=1mm] {$  $} (v4)
			(v7) edge [->] node [ above ] {$ i_{1,X}^* $} (v10)
			(v10) edge [->] node [ above ] {$ \iota_! $} (v11)
			% Draw thick white lines on the front
			(v8) edge [-,line width=1mm, white] node [] {} (v2)
			(v8) edge [-, line width=1mm, white] node [] {} (v11)
			(v16) edge [bend left=10,-,line width=1mm, white] node [] {} (v6)
			%Draw arrows in the front
			(v7) edge [double equal sign distance] node [ left ] {$  $} (v1)
			(v8) edge [->] node [ left=2mm, above=1mm ] {$ \hat \pi_! $} (v2)
			(v11) edge [->] node [ left ] {$ \mu_! $} (v5)
			(v7) edge [->] node [ left, below ] {$ \hat\iota_! $} (v8)
			(v8) edge [->] node [ left=2mm ] {$ i_{1,M}^* $} (v11)
			(v1) edge [->] node [ below=2mm, left=2mm ] {$ \hat f_! $} (v2)
			(v2) edge [->] node [ below ] {$ i_{1,Y}^* $} (v5)
			%Arrow in the back second cube
			(v16) edge [bend left=10,->] node [ right=7mm, below=-2mm ] {$ (p\circ \mu)\epfs $} (v6)
			(v5) edge [->] node [ right=2mm, below=1mm ] {$ p\epfs $} (v6)
			% Draw thick white lines on the front
			(v14) edge [bend left=15,-,line width=1mm, white] node [] {} (v3)
			%(v9) edge [-, line width=1mm, white] node [] {} (v3)
			%Arrow in the front second cube
			(v2) edge [->] node [ right=-2mm, below=-0.5mm ] {$ \hat{p}\epfs $} (v3)
			(v14) edge [bend left=15,->] node [ left=0mm, above=2mm ] {$ (\hat p\circ \hat \pi)\epfs $} (v3)
		%	(v9) edge [double equal sign distance] node [ right=2mm, below=0.1mm ] {$  $} (v12)
		%	(v9) edge [double equal sign distance] node [ left=3mm, below=0.1mm ] {$ $} (v3)
			(v3) edge [double equal sign distance] node [ right=2mm, below=0.1mm ] {$  $} (v6)
		%	(v12) edge [double equal sign distance] node [ right=2mm, below=0.1mm ] {$  $} (v6)
			% Draw thick white lines on the front wing
			%(v15) edge [-,line width=1mm, white] node [] {} (v13)
			(v13) edge [-,line width=1mm, white] node [] {} (v14)
			(v14) edge [-,line width=1mm, white] node [] {} (v8)
			(v14) edge [-,line width=1mm, white] node [] {} (v2)
			% Draw arrows front wing
			(v10) edge [bend right=30,->] node [above] {$ r_X^* $} (v7)
			(v15) edge [bend right=45,->] node [above] {$ r_X^* $} (v13)
			
			(v13) edge [->] node [ below=2mm,right=-1.5mm] {$ \Sigma^{r_X^*\iota^*\mb L_{\mu}} $} (v7)
			(v13) edge [->] node [ below=2mm, right=-4mm ] {$ \hat\iota_{!} $} (v14)
			(v14) edge [->] node [ below=2mm, right=-1mm ] {$ \Sigma_{\hat\pi} $} (v8)
			(v14) edge [->] node [ below ] {$ \hat \pi\epfs $} (v2);
		\end{tikzpicture}
	\end{equation}
	
	In order to fill \eqref{eq:_HP__BIG_diagram_with_Double_Def}, we will split it into smaller pieces. Let us start from the central rectangular parallelepipedon; one can decompose and fill it as:
	
	\begin{center}
		\begin{tikzpicture}[baseline={(0,0.5)}, scale=2]
			\node (v1) at (0,1) {$ D(M) $};
			\node (v2) at (1,1) {$ D(X) $};
			\node (v3) at (2,1) {$ D(X) $};
			\node (v4) at (2,2) {$ D(X) $};
			\node (v5) at (3,2) {$ D(M) $};
			\node (v6) at (3,1) {$ D(Y) $};
			\node (v7) at (4,1) {$ D(M) $};
			\node (v8) at (4,0) {$ D(D_{\iota}) $};
			\node (v9) at (3,0) {$ D(D_f) $};
			\node (v10) at (3,-1) {$ D(D_{\iota}) $};
			\node (v11) at (2,-1) {$ D(\A^1_X) $};
			\node (v12) at (2,0) {$ D(\A^1_X) $};
			\node (v13) at (1,0) {$ D(\A^1_X) $};
			\node (v14) at (0,0) {$ D(D_{\iota}) $};
			
			\node (c1) at (0.5,0.5) { \rotatebox{45}{$ \overset{\scriptsize Ex_{!}^*}{\Rightarrow} $} };
			\node (c2) at (1.5,0.5) {$ \circlearrowleft $};
			\node (c3) at (2.5,0.5) {\rotatebox{-45}{$ \overset{\scriptsize Ex_{!}^*}{\Leftarrow} $}};
			\node (c4) at (3.5,0.5) {\rotatebox{-45}{$ \overset{\scriptsize Ex_{!}^*}{\Rightarrow} $}};
			
			\node (c5) at (2.5,1.5) { $ \circlearrowleft $};
			\node (c6) at (2.5,-0.5) { $ \circlearrowleft $};

			\path[font=\scriptsize,>= angle 90]
			% Horizontal
			(v4) edge [->] node [above ] {$ \iota_! $} (v5)
			
			(v1) edge [<-] node [above ] {$ \iota_! $} (v2)
			(v2) edge [double equal sign distance] node [above ] {$  $} (v3)
			(v3) edge [->] node [above ] {$ f_! $} (v6)
			(v6) edge [<-] node [above ] {$ \mu_! $} (v7)
			
			(v14) edge [<-] node [below ] {$ \hat{\iota}_! $} (v13)
			(v13) edge [double equal sign distance] node [below ] {$  $} (v12)
			(v12) edge [->] node [below ] {$ \hat{f}_! $} (v9)
			(v9) edge [<-] node [below ] {$ \hat{\pi}_! $} (v8)
			
			(v11) edge [->] node [below ] {$ \hat{\iota}_! $} (v10)
			
			%Vertical
			
			(v4) edge [double equal sign distance] node [left ] {$  $} (v3)
			(v5) edge [->] node [right ] {$ \mu_! $} (v6)
			
			(v1) edge [->] node [left ] {$ i_{1,M}^*  $} (v14)
			(v2) edge [->] node [left ] {$ i_{1,X}^* $} (v13)
			(v3) edge [->] node [left ] {$ i_{1,X}^* $} (v12)
			(v6) edge [->] node [left ] {$ i_{1,Y}^* $} (v9)
			(v7) edge [->] node [right ] {$ i_{1,M}^* $} (v8)
			
			(v12) edge [double equal sign distance] node [left ] {$  $} (v11)
			(v9) edge [<-] node [right ] {$ \hat{\pi}_! $} (v10);
			
		\end{tikzpicture}
	\end{center}
	\noindent The winged square on the floor in \eqref{eq:_HP__BIG_diagram_with_Double_Def}, and adjacent to the parallelepipedon, can be filled as:

	\begin{center}
		\begin{tikzpicture}[baseline={(0,0.5)}, scale=2]
		%	\node (v1) at (0,1) {$ D(S) $};
		%	\node (v2) at (1,1) {$ D(M) $};
			\node (v3) at (1,1) {$ D(Y) $};
			\node (v4) at (2,1) {$ D(M) $};
		%	\node (v5) at (3,2) {$ D(S) $};
			\node (v6) at (1,0) {$ D(S) $};
		%	\node (v7) at (4,1) {$ D(S) $};
		%	\node (v8) at (4,0) {$ D(S) $};
			\node (v9) at (0,0) {$ D(S) $};
		%	\node (v10) at (3,-1) {$ D(S) $};
			\node (v11) at (-1,1) {$ D(D_{\iota}) $};
			\node (v12) at (0,1) {$ D(D_f) $};
		%	\node (v13) at (1,0) {$ D(D_{\iota}) $};
		%	\node (v14) at (0,0) {$ D(S) $};
			
		%	\node (c1) at (0.5,0.5) { \rotatebox{45}{$ \overset{\scriptsize Ex_{\#}^*}{\Leftarrow} $} };
		%	\node (c2) at (1.5,0.5) {\rotatebox{-45}{$ \overset{\scriptsize Ex_{!}^*}{\Rightarrow} $}};
			\node (c3) at (0.5,0.5) {\rotatebox{45}{$ \overset{\scriptsize Ex_{\#}^*}{\Leftarrow} $}};
			\node (c4) at (-0.35,0.65) {$ \circlearrowleft $};
			\node (c4) at (1.35,0.65) {$ \circlearrowleft $};
		%	\node (c5) at (2.5,1.5) {\rotatebox{45}{$ \overset{\scriptsize Ex_{\#!}}{\Rightarrow} $}};
		%	\node (c6) at (2.5,-0.5) {\rotatebox{-45}{$ \overset{\scriptsize Ex_{\#!}}{\Rightarrow} $}};

			\path[font=\scriptsize,>= angle 90]
			% Horizontal
			(v4) edge [bend left=20,->] node [right ] {$ (p\circ \mu)\epfs $} (v6)
			
		%	(v1) edge [<-] node [above ] {$ (p\circ \pi)\epfs  $} (v2)
		%	(v2) edge [->] node [above ] {$ \pi_! $} (v3)
			(v3) edge [->] node [left ] {$ p\epfs $} (v6)
		%	(v6) edge [double equal sign distance] node [above ] {$  $} (v7)
			
		%	(v14) edge [<-] node [below ] {$ (\hat p\circ \hat \pi)\epfs $} (v13)
		%	(v13) edge [->] node [below ] {$ \hat{\pi}_! $} (v12)
			(v12) edge [->] node [left ] {$ \hat p\epfs $} (v9)
		%	(v9) edge [double equal sign distance] node [below ] {$  $} (v8)
			
			(v11) edge [bend right=20,->] node [left ] {$ (\hat p\circ \hat \pi)\epfs $} (v9)
			
			%Vertical
			
			(v4) edge [->] node [above] {$ \mu\epfs $} (v3)
		%	(v5) edge [double equal sign distance] node [right ] {$  $} (v6)
			
		%	(v1) edge [double equal sign distance] node [left ] {$  $} (v14)
		%	(v2) edge [->] node [left=3mm,above=1mm ] {$ i_{1,M}^* $} (v13)
			(v3) edge [<-] node [above] {$ i_{1,Y}^* $} (v12)
			(v6) edge [double equal sign distance] node [left ] {$  $} (v9)
		%	(v7) edge [double equal sign distance] node [right ] {$  $} (v8)
			
			(v12) edge [<-] node [above ] {$ \hat{\pi}\epfs $} (v11);
		%	(v9) edge [double equal sign distance] node [right ] {$  $} (v10);
			
		\end{tikzpicture}
	\end{center}
	The small lateral triangles in diagram \eqref{eq:_HP__BIG_diagram_with_Double_Def} involving $\Sigma_{\pi}$ and $\Sigma_{\hat{\pi}}$ can be decomposed and filled using the natural ambidexterity equivalences $\varphi_{\pi}$ and $\varphi_{\hat\pi}$ of \Cref{Repr_lft_Ayoub_Purity_after_AGV}.\\
	Now we only need to show the commutativity of the last three squares left in \eqref{eq:_HP__BIG_diagram_with_Double_Def}, namely:\\
	\begin{minipage}{0.3\textwidth}
		\begin{equation*}
			\begin{tikzpicture}[baseline={(0,0.5)}, scale=1.5]
				\node (a) at (0,1) {$   D(X) $};
				\node (b) at (1,1) {$  D(M)  $};
				\node (c) at (0,0) {$   D(X)  $};
				\node (d) at (1,0) {$   D(M)  $};
				\node (e) at (0.25,0.75) {$  $};
				\node (e) at (0.5,0.5) {$  $};

				\path[font=\scriptsize,>= angle 90]
				
				(a) edge [->] node [above ] {$ \iota_! $} (b)
				(a) edge [->] node [left] {$ \Sigma^{\iota^*\mb L_{\mu}} $} (c)
				(b) edge [->] node [right] {$ \Sigma_{\mu} $} (d)
				(c) edge [->] node [below] {$ \iota_!  $} (d);
				
			\end{tikzpicture}
		\end{equation*}
	\end{minipage}
	\hfill 
	\begin{minipage}{0.2\textwidth}
		\begin{equation*}
			\begin{tikzpicture}[baseline={(0,0.5)}, scale=1.5]
				\node (a) at (0,1) {$   D(X) $};
				\node (b) at (1,1) {$  D(\A^1_X) $};
				\node (c) at (0,0) {$  D(X) $};
				\node (d) at (1,0) {$  D(\A^1_X) $};
				\node (e) at (0.25,0.75) {$  $};
				\node (e) at (0.65,0.5) {$  $};

				\path[font=\scriptsize,>= angle 90]
				
				(a) edge [->] node [above ] {$ r_X^* $} (b)
				(a) edge [->] node [left] {$ \Sigma^{\iota^*\mb L_{\mu}}  $} (c)
				(b) edge [->] node [right] {$ \Sigma^{r_X^*\iota^*\mb L_{\mu}}  $} (d)
				(c) edge [->] node [below] {$ r_X^* $} (d);
				
			\end{tikzpicture}
		\end{equation*}
	\end{minipage}
	\hfill
	\begin{minipage}{0.35\textwidth}
		\begin{equation}\label{eq:_Last_3_Squares_HP}
			\begin{tikzpicture}[baseline={(0,0.5)}, scale=1.5]
				\node (a) at (0,1) {$   D(A^1_X) $};
				\node (b) at (1,1) {$  D(D_{\iota}) $};
				\node (c) at (0,0) {$ D(\A^1_X) $};
				\node (d) at (1,0) {$ D(D_{\iota})$};
				\node (e) at (0.25,0.75) {$   $};
				\node (e) at (0.65,0.5) {$  $};

				\path[font=\scriptsize,>= angle 90]
				
				(a) edge [->] node [above ] {$ \hat{\iota}\epf $} (b)
				(a) edge [->] node [left] {$ \Sigma^{r_X^*\iota^*\mb L_{\mu}} $} (c)
				(b) edge [->] node [right] {$ \Sigma_{\hat{\mu}} $} (d)
				(c) edge [->] node [below] {$ \iota_! $} (d);
				
			\end{tikzpicture}
		\end{equation}
	\end{minipage}\\
	\noindent For $\bullet \in \set{\mu,\hat{\pi},\pi_0}$, denote by $T_{\bullet}$ the vector bundle associated with $\mb L_{\bullet}$, by $ q_{\bullet} $ the relative projection onto the base, and by $s_{\bullet}$ the relative zero section; let $ q_{\iota^*\mu} $ and $ s_{\iota^*\mu} $ be the projection onto the base and the zero section relative to $\iota^*T_{\mu}$. Recall that we have canonical identifications between the normal bundles of the diagonal $ N_{\Delta_{\bullet}} $ and $ T_{\bullet} $. Using the identifications in \eqref{eq_ctg_cplx_HP_proposition} and the natural equivalences in \eqref{eq:_closed_HP_for_repr-lft_HP}, we can then decompose and fill the squares in \eqref{eq:_Last_3_Squares_HP} as:
	
	\begin{center}
		\begin{minipage}{0.4\textwidth}
			\begin{center}
				\begin{tikzpicture}[baseline={(0,0.5)}, scale=2]
					\node (a1) at (0,1) {$   D(X) $};
					\node (a2) at (1,1) {$  D(M)  $};
					\node (a3) at (3,1) {$  D(M)  $};
					\node (b1) at (0,0) {$   D(\iota^*T_{\mu})  $};
					\node (b2) at (1,0) {$   D(N_{\Delta_{\mu}})  $};
					\node (c1) at (0,-1) {$ D(X) $};
					\node (c2) at (1,-1) {$ D(M) $};
					\node (c3) at (3,-1) {$  D(M)  $};
					
					\node (v0) at (0.5,0.5) {$  \circlearrowleft  $};
					\node (v1) at (0.5,-0.5) {\rotatebox{45}{$ \overset{\scriptsize Ex_{\#!}}{\Leftarrow} $} };
					\node (v2) at (2,0) { \rotatebox{-45}{\scalebox{0.65}{$ \Pi_0^{-1}(\Delta_{\mu},p_{2,\mu})\Pi_1(\Delta_{\mu},p_{2,\mu}) $}}  };
					
					\path[font=\scriptsize,>= angle 90]
					
					(a2) edge [double equal sign distance] node [above ] {$  $} (a3)
					(c2) edge [double equal sign distance] node [above ] {$  $} (c3)
					(a3) edge [->] node [right ] {$ \Sigma_{\pi} $} (c3)
					
					(a1) edge [->] node [above ] {$ \iota_! $} (a2)
					(b1) edge [->] node [above ] {$  $} (b2)
					(a1) edge [->] node [left] {$ (s_{\iota^*\mu})\epf $} (b1)
					(b1) edge [->] node [left] {$ (q_{\iota^*\mu})\epfs $} (c1)
					(a2) edge [->] node [right] {$ s_{\mu}\epf $} (b2)
					(b2) edge [->] node [right] {$ q_{\mu}\epfs $} (c2)
					(c1) edge [->] node [below] {$ \iota_!  $} (c2);
					
				\end{tikzpicture}
			\end{center}
		\end{minipage}
		\hfill
		\begin{minipage}{0.4\textwidth}
			\begin{center}
				\begin{tikzpicture}[baseline={(0,0.5)}, scale=2]
					\node (a1) at (0,1) {$   D(X) $};
					\node (a2) at (1,1) {$  D(\A^1_X)  $};
					\node (b1) at (0,0) {$   D(\iota^*T_{\pi})  $};
					\node (b2) at (1,0) {$   D(r_X^*\iota^*T_{\pi})  $};
					\node (c1) at (0,-1) {$ D(X) $};
					\node (c2) at (1,-1) {$ D(\A^1_X) $};
					
					\node (v0) at (0.5,0.5) {\rotatebox{45}{$ \overset{\scriptsize Ex_{!}^*}{\Rightarrow} $}};
					\node (v1) at (0.5,-0.5) {\rotatebox{45}{$ \overset{\scriptsize Ex_{\#!}}{\Leftarrow} $} };

					\path[font=\scriptsize,>= angle 90]
					
					(a1) edge [->] node [above ] {$ r_X^* $} (a2)
					(b1) edge [->] node [above ] {$  $} (b2)
					(a1) edge [->] node [left] {$ (s_{\iota^*\mu})\epf $} (b1)
					(b1) edge [->] node [left] {$ (q_{\iota^*\mu})\epfs $} (c1)
					(a2) edge [->] node [right] {$  $} (b2)
					(b2) edge [->] node [right] {$  $} (c2)
					(c1) edge [->] node [below] {$ r_X^*  $} (c2);
					
				\end{tikzpicture}
			\end{center}
		\end{minipage}
	\end{center}
	
	\begin{center}
		\begin{tikzpicture}[baseline={(0,0.5)}, scale=2]
			\node (a1) at (0,1) {$   D(\A^1_X) $};
			\node (a2) at (1,1) {$  D(M)  $};
			\node (a3) at (3,1) {$  D(M)  $};
			\node (b1) at (0,0) {$   D(r_X^*\iota^*T_{\pi})  $};
			\node (b2) at (1,0) {$   D(N_{\Delta_{\hat\pi}})  $};
			\node (c1) at (0,-1) {$ D(\A^1_X) $};
			\node (c2) at (1,-1) {$ D(M) $};
			\node (c3) at (3,-1) {$  D(M)  $};
			
			\node (v0) at (0.5,0.5) {$  \circlearrowleft  $};
			\node (v1) at (0.5,-0.5) {\rotatebox{45}{$ \overset{\scriptsize Ex_{\#!}}{\Leftarrow} $} };
			\node (v2) at (2,0) { \rotatebox{-45}{\scalebox{0.65}{$ \Pi_0^{-1}(\Delta_{\hat\pi},p_{2,\hat\pi})\Pi_1(\Delta_{\hat\pi},p_{2,\hat\pi}) $}}  };
			
			\path[font=\scriptsize,>= angle 90]
			
			(a2) edge [double equal sign distance] node [above ] {$  $} (a3)
			(c2) edge [double equal sign distance] node [above ] {$  $} (c3)
			(a3) edge [->] node [right ] {$ \Sigma_{\mu} $} (c3)
			
			(a1) edge [->] node [above ] {$ \hat\iota_! $} (a2)
			(b1) edge [->] node [above ] {$  $} (b2)
			(a1) edge [->] node [left] {$  $} (b1)
			(b1) edge [->] node [left] {$  $} (c1)
			(a2) edge [->] node [right] {$ s_{\hat\pi}\epf $} (b2)
			(b2) edge [->] node [right] {$ q_{\hat\pi}\epfs $} (c2)
			(c1) edge [->] node [below] {$ \hat\iota_!  $} (c2);
			
		\end{tikzpicture}
	\end{center}
	\noindent where the natural transformation $\Pi_0^{-1}(\Delta_{\hat\pi},p_{2,\hat\pi})\Pi_1(\Delta_{\hat\pi},p_{2,\hat\pi})$ actually involved the double deformation space $ \mr{Def}_{\Delta_{\hat\pi}} $.\\
	Putting everything together, our decomposition of \eqref{eq:_HP__BIG_diagram_with_Double_Def} and the naturality of the exchange transformations gives us a way to fill the following commutative diagram:
	
	\begin{center}
		\begin{tikzpicture}[baseline={(0,0)}, scale=2]
			\node (a1) at (0,2.5) {$ \mr{Th}(\iota,p\circ \mu) $};
			\node (a2) at (0,2.25) {\rotatebox{90}{$=$}};
			\node (a3) at (0,2) {$ p\epfs \mu\epfs\iota\epf  $};
			\node (a4) at (0,1) {$ p\epfs\mu\epf \Sigma_{\mu}\iota_! $};
			\node (a5) at (0,0.75) {\rotatebox{90}{$\simeq$}};
			\node (a6) at (0,0.5) {$ p\epfs \mu_!\iota_!\Sigma^{\iota^*\mb L_{\mu}} $};
			\node (a7) at (0,0.25) {\rotatebox{90}{$=$}};
			\node (a8) at (0,0) {$ \mr{Th}(f,p)\Sigma^{\iota^*\mb L_{\mu}} $};
			
			\node (b1) at (3,2.5) {$ r_S\epfs \mr{Th}(\hat\iota,\hat p\circ \hat \pi)r_X^* $};
			\node (b2) at (3,2.25) {\rotatebox{90}{$=$}};
			\node (b3) at (3,2) {$ r_S\epfs \hat  p\epfs \hat \pi\epfs \hat \iota\epf r_X^* $};
			\node (b4) at (3,1) {$ r_S\epfs \hat  p\epfs \hat \pi\epf \Sigma_{\hat\pi} \hat \iota\epf r_X^* $};
			\node (b5) at (3,0.75) {\rotatebox{-90}{$\simeq$}};
			\node (b6) at (3,0.5) {$ r_S\epfs \hat  p\epfs \hat \pi\epf \hat \iota\epf r_X^*\Sigma^{\iota^*\mb L_{\mu}} $};
			\node (b7) at (3,0.25) {\rotatebox{90}{$=$}};
			\node (b8) at (3,0) {$ r_S\epfs \mr{Th}(\hat f, \hat p)r_X^*\Sigma^{\iota^*\mb L_{\mu}} $};
			
			\node (c) at (1.5,1.25) {$ \circlearrowleft $};

			\path[font=\scriptsize,>= angle 90]

			(a1) edge [->] node [above] {$ \Pi_1(\iota,p\circ \mu) $} (b1)
			(a3) edge [->] node [left] {\rotatebox{-90}{$\sim$}} node [right] {$ \varphi_{\mu}$} (a4)
			(b3) edge [->] node [right] {\rotatebox{-90}{$\sim$}}  node [left] {$ \varphi_{\hat \pi}$} (b4)
			(a8) edge [->] node [below] {$ \Pi_1(f,p) $} (b8);
			
		\end{tikzpicture}
	\end{center}
	\noindent Applying $\Sigma^{-\iota^*\mb L_{\mu}}$ to the diagram above, we get a new commutative diagram:
	\begin{center}
		\begin{tikzpicture}[baseline={(0,1)}, scale=2]
			
			\node (a) at (0,1) {$  \mr{Th}(\iota,p\circ \mu)\Sigma^{-\iota^*\mb L_{\mu}} $};
			
			\node (b) at (2.5,1) {$ r_{ S}\epfs\mr{Th}(\hat{\iota},\hat{p}\circ \hat\pi)r_{X}^*\Sigma^{-\iota^*\mb L_{\mu}}  $};
			
			%	\node (c) at (5,1) {$  \mr{Th}(\iota_0,p_0\circ \pi_0)\Sigma^{-\mb L_{\pi}} $};
			
			\node (d) at (0,0) {$  \mr{Th}(f,p) $};
			
			\node (e) at (2.5,0) {$ r_{ S}\epfs\mr{Th}(\hat{f},\hat{p})r_{X}^*  $};
			
			%	\node (f) at (5,0) {$  \mr{Th}(f_0,p_0) $};

			%\node (g) at (1.25,0.75) {$ \ulcorner $};
			%\node (h) at (2.25,0.75) {$ \ulcorner $};

			\path[font=\scriptsize,>= angle 90]

			(a) edge [->] node [above] {$ \Pi_1(\iota,p\circ \mu) $} (b)
			%	(c) edge [->] node [above] {$ \Pi_0(\iota,p\circ \pi) $} (b)
			(d) edge [->] node [below] {$ \Pi_1(f,p) $} (e)
			%	(f) edge [->] node [below] {$ \Pi_0(f,p) $} (e)
			(a) edge [->] node [left] {\rotatebox{-90}{$ \sim  $}} (d)
			
			(b) edge [->] node [left] {\rotatebox{-90}{$ \sim  $}} (e);
			
			%	(c) edge [->] node [right] {\rotatebox{-90}{$ \sim  $}} (f);
			
		\end{tikzpicture}
	\end{center}
	\noindent that is exactly what we wanted.\\
	
	\noindent To prove the existence of a commutative diagram of the form:
	
	\begin{center}
		\begin{tikzpicture}[baseline={(0,1)}, scale=2]
			
			\node (b) at (2.5,1) {$ r_{ S}\epfs\mr{Th}(\hat{\iota},\hat{p}\circ \hat\pi)r_{X}^*\Sigma^{-\mb L_{\mu}}  $};
			
			\node (c) at (5,1) {$  \mr{Th}(\iota_0,p_0\circ \pi_0)\Sigma^{-\mb L_{\mu}} $};
			
			\node (e) at (2.5,0) {$ r_{ S}\epfs\mr{Th}(\hat{f},\hat{p})r_{X}^*  $};
			
			\node (f) at (5,0) {$  \mr{Th}(f_0,p_0) $};
			
			\path[font=\scriptsize,>= angle 90]
			
			(c) edge [->] node [above] {$ \Pi_0(\iota,p\circ \mu) $} (b)
			
			(f) edge [->] node [below] {$ \Pi_0(f,p) $} (e)
			
			(b) edge [->] node [left] {\rotatebox{-90}{$ \sim  $}} (e)
			
			(c) edge [->] node [right] {\rotatebox{-90}{$ \sim  $}} (f);
			
		\end{tikzpicture}
	\end{center}
	\noindent one has to replace diagram \eqref{eq:_HP__BIG_diagram_with_Double_Def} with a similar and appropriate diagram involving $N_{\iota}$ and $N_{f}$ and then follow the same step as we did before. To avoid burdening the reader with an even longer proof, we leave the remaining details for them to work out.
\end{proof}

Now we are ready to prove the main theorem of this section:

\begin{theorem}[Homotopy Purity] \label{Thm:_Repr-lft_HP}
	Let $f: \mc X \longrightarrow \mc Y$ be a representable map between $S$-smooth NL stacks, with $p,q$ their structure maps. Then using the same conventions of \Cref{Notation:_HP}, we have that the natural transformations:
	\[ 	\mr{Th}(f,p) \overset{\Pi_1(f,p)}{\longrightarrow} (r_S)\epfs \mr{Th}(\hat{f},\hat{p})r_X^* \overset{\Pi_0(f,p)}{\longleftarrow} q_{\#}\mr{Th}(f_0,\pi_0) \]
	\noindent are both equivalences.
\end{theorem}

\begin{proof}[Proof of \Cref{Thm:_Repr-lft_HP}]
	Choose an atlas $y: Y \rightarrow \mc Y$ of $\mc Y$ and let $x: X:=Y\times_{\mc Y}\mc X \rightarrow \mc X$ be the atlas on $\mc X$ obtained via the fiber product. Denote by $g$ the natural map $X \rightarrow Y$. By the conservativity of $y^*$, we know that it is enough to prove that:
	\[ y^*\Pi_1(f,p) \ \ \ \ \ \ \ \ y^*\Pi_0(f,p) \]
	\noindent are both equivalences. But one can check that this corresponds to prove that:
	\[ \Pi_1(g,p\circ y)x^* \ \ \ \ \ \ \ \ \Pi_0(g,p\circ y)x^* \]
	\noindent are equivalences. So without loss of generality we can assume that $f$ is just a map between smooth $S$-schemes $ X=\mc X  \rightarrow Y=\mc Y$. By \cite[\href{https://stacks.math.columbia.edu/tag/02FV}{Tag 02FV}]{stacks-project} and \cite[\href{https://stacks.math.columbia.edu/tag/0E9K}{Tag 0E9K}]{stacks-project}, the map $f$ is then automatically lci. Working Zariski locally, we can further assume that the map is actually gci, i.e. that $f$ factors as:
	\[ X \overset{\iota}{\into} M \overset{\mu}{\rightarrow} Y \]
	\noindent for some smooth $Y$-scheme $M$. Now we can apply \Cref{Key_Lemma:_schematic_HP+}, and hence we can find a commutative diagram:
	
	\begin{center}
		\begin{tikzpicture}[baseline={(0,1)}, scale=2]
			
			\node (a) at (0,1) {$  \mr{Th}(\iota,p\circ \mu)\Sigma^{-\mb L_{\mu}} $};
			
			\node (b) at (2.5,1) {$ r_{ S}\epfs\mr{Th}(\hat{\iota},\hat{p}\circ \hat\pi)r_{X}^*\Sigma^{-\mb L_{\mu}}  $};
			
			\node (c) at (5,1) {$  \mr{Th}(\iota_0,p_0\circ \pi_0)\Sigma^{-\mb L_{\mu}} $};
			
			\node (d) at (0,0) {$  \mr{Th}(f,p) $};
			
			\node (e) at (2.5,0) {$ r_{ S}\epfs\mr{Th}(\hat{f},\hat{p})r_{X}^*  $};
			
			\node (f) at (5,0) {$  \mr{Th}(f_0,p_0) $};

			%\node (g) at (1.25,0.75) {$ \ulcorner $};
			%\node (h) at (2.25,0.75) {$ \ulcorner $};

			\path[font=\scriptsize,>= angle 90]

			(a) edge [->] node [above] {$ \Pi_1(\iota,p\circ \mu) $} (b)
			(c) edge [->] node [above] {$ \Pi_0(\iota,p\circ \mu) $} (b)
			(d) edge [->] node [below] {$ \Pi_1(f,p) $} (e)
			(f) edge [->] node [below] {$ \Pi_0(f,p) $} (e)
			(a) edge [->] node [left] {\rotatebox{-90}{$ \sim  $}} (d)
			
			(b) edge [->] node [left] {\rotatebox{-90}{$ \sim  $}} (e)
			
			(c) edge [->] node [right] {\rotatebox{-90}{$ \sim  $}} (f);
			
		\end{tikzpicture}
	\end{center}
	\noindent where all the vertical arrows are invertible. But the top row is formed by equivalences too by \cite[Proposition 5.7]{Hoyois_Equiv_Six_Op}. Hence $\Pi_1(f,p) $ and $ \Pi_0(f,p)  $ must be equivalences themselves and we are done.
\end{proof}

	\subsection{Ambidexterity and Purity}
%The formal Thom transformation associated to a smooth map with a (representable) section is invertible. Hence, instead of considering $\varphi_f$ we can equivalently consider:
%
%\[ \mf{fp_f}:= \varphi_f\Omega_f:  f\epfs \Omega_{f} \longrightarrow f\epf   \]

By mimicking \cite[Proposition 1.7.3]{Ayoub_6FF_vol1}, we are going to show the following:

\begin{proposition}\label{Thm:_Ayoub_1.7.3+}
	Let:
	\[ \mc X \overset{f}{\longrightarrow} \mc Y \overset{g}{\longrightarrow} \mc Z \]
	\noindent be two composable smooth maps where $f$ and $gf$ are representable. Then we get a commutative diagram of the form:

	\begin{center}
		\begin{tikzpicture}[baseline={(0,0.5)}, scale=2]
			
			\node (a) at (0.75,1) {$  (g f)\epfs $};
			
			\node (c) at (3,1) {$ (g f)_! \Sigma_{gf} $};

			\node (d) at (0.75,0) {$ g\epfs  f\!\epfs $};
			\node (e) at (2,0) {$ g\epf  \Sigma_g f\epfs   $};
			\node (f) at (3,0) {$ g\epf \Sigma_g f_!\Sigma_f  $};

			%\node (g) at (1.25,0.75) {$ \ulcorner $};
			%\node (h) at (2.25,0.75) {$ \ulcorner $};

			\path[font=\scriptsize,>= angle 90]

			(a) edge [->] node [above] {$ \varphi_{g f} $} (c)
			
			(d) edge [->] node [below] {$ \varphi_g $} (e)
			(e) edge [->] node [below] {$  \varphi_f $} (f)
			
			(a) edge [double equal sign distance] node [left] {$  $} (d)
			
			(c) edge [<-] node [left] {\rotatebox{-90}{$ \sim  $}} (f);
			
		\end{tikzpicture}
	\end{center}
	\noindent as well as the analogous diagram where the bottom row is replaced by:
	\[ g\epfs  f\epfs\underset{\varphi_f}{\longrightarrow} g\epfs  f_! \Sigma_f \underset{\varphi_g}{\longrightarrow} g_!\Sigma_g f_!\Sigma-f \]
\end{proposition}
\begin{proof}
	It is not difficult to see that with slight modifications we can adapt the proof of \cite[1.7.3]{Ayoub_6FF_vol1} to our case. We will show what modifications are in order just for the reader convenience, but we make no claim of originality in this proof.\\
	Let $ p_{f,i}, p_{g,i}, p_{gf,i} $ be the projections onto the first and second coordinates for $i=1,2$ on $\mc X^2_{\mc Y}, \mc Y_{\mc Z}^2, \mc X^2_{\mc Z}$ respectively. Similarly, we will denote by $p_{f^*g,i}$ the map obtained as a pullback from $p_{g,i}$ along $f$.
	The natural transformation $\varphi_{g f}$ is defined as:
	
	\[ (g f)\epfs = (g f)\epfs(p_{gf,1})\epf \Delta_{gf}\epf \overset{Ex_{\#!}}{\longrightarrow} (gf)_!\Sigma_{gf} \]
	
	Now consider the following collection of cartesian squares:
	
	\begin{center}
		\begin{tikzpicture}[baseline={(0,1.5)}, scale=2]

			\node (a2) at (0,2) {$ \mc X_{\mc Z}^2 $};
			\node (a1) at (0,1) {$ \mc X\times_{\mc Z} \mc Y  $};
			\node (a0) at (0,0) {$\mc X $};

			\node (b2) at (1,2) {$ \mc Y\times_{\mc Z} \mc X $};
			\node (b1) at (1,1) {$ \mc Y^2_{\mc Z} $};
			\node (b0) at (1,0) {$ \mc Y  $};

			\node (c2) at (2,2) {$ \mc X $};
			\node (c1) at (2,1) {$ \mc Y $};
			\node (c0) at (2,0) {$ \mc Z $};

			\node (t1) at (0.25,1.75) {$ \ulcorner $};
			\node (t3) at (0.25,0.75) {$ \ulcorner $};
			\node (t2) at (1.25,1.75) {$ \ulcorner $};
			\node (t4) at (1.25,0.75) {$ \ulcorner $};

			\node (T1) at (0.5,1.5) {$ \Delta_2 $};
			\node (T3) at (0.5,0.5) {$ \Delta_1 $};
			\node (T2) at (1.5,1.5) {$ \Delta_3 $};
			\node (T4) at (1.5,0.5) {$ \Delta_4 $};

			%\node (g) at (0.5,0.5) {$ \Delta $};

			\path[font=\scriptsize,>= angle 90]

			%horizontal
			
			(a2) edge [->] node [above ] {$ v_2 $} (b2)
			(a1) edge [->] node [above ] {$ r_2 $} (b1)
			(a0) edge [->] node [below ] {$ f $} (b0)
			
			(b2) edge [->] node [above ] {$ p_{f^*g,2} $} (c2)
			(b1) edge [->] node [above ] {$ p_{g,2} $} (c1)
			(b0) edge [->] node [below ] {$ g $} (c0)
			
			%vertical
			
			(a2) edge [->] node [left ] {$ v_1 $} (a1)
			(a1) edge [->] node [left ] {$ q_1 $} (a0)

			(b2) edge [->] node [left ] {$ r_1 $} (b1)
			(b1) edge [->] node [left ] {$ p_{g,1} $} (b0)
			
			(c2) edge [->] node [right ] {$ f $} (c1)
			(c1) edge [->] node [right ] {$ g $} (c0);
		\end{tikzpicture}
	\end{center}
	
	Let $\Delta_{Big}$ be the outer cartesian square. The naturality of $Ex_{\#!}$ will fill the following diagram:
	
	\begin{center}
		\begin{tikzpicture}[baseline={(0,0.5)}, scale=2.5]
			
			\node (a) at (0,1) {$ (gf)\epfs $};
			\node (b) at (4,1) {$ (gf)\epf \Sigma_{gf} $};
			
			\node (c1) at (0,0) {$ g\epfs f\epfs q_1\epf v_1\epf \Delta_{gf}\epf $};
			\node (c2) at (1,-1) {$ g\epfs (p_{g,1})\epf r_2\epfs v_1\epf \Delta_{gf}\epf  $};
			\node (c3) at (2,0) {$ g\epfs (p_{g,1})\epf r_1\epf v_2\epfs \Delta_{gf}\epf  $};
			\node (c4) at (3,-1) {$ g\epf (p_{g,2})\epfs  r_1\epf v_2\epfs \Delta_{gf}\epf  $};
			\node (c5) at (4,0) {$ g\epf f\epf  q_2\epfs v_2\epfs \Delta_{gf}\epf  $};

			\path[font=\scriptsize,>= angle 90]

			(a) edge [->] node [above] {$ Ex_{\#|}(\Delta_{Big}) $} (b)
			(c1) edge [->] node [left] {$ Ex_{\#!}(\Delta_1) $} (c2)
			(c2) edge [->] node [left] {$ Ex_{\#!}(\Delta_2) $} (c3)
			(c3) edge [->] node [right] {$ Ex_{\#!}(\Delta_3)  $} (c4)
			(c4) edge [->] node [right ] {$Ex_{\#!}(\Delta_4)  $} (c5)
			
			(a) edge [double equal sign distance] node [above] {$  $} (c1)
			(b) edge [double equal sign distance] node [above] {$  $} (c5);
			
		\end{tikzpicture}
	\end{center}
	Said it differently, denoting by $D(-)$ the functor $SH(-)$, we have that the composition represented by the diagram:
	
	\begin{center}
		\begin{tikzpicture}[baseline={(0,-1)}, scale=2.5]

			\node (a) at (0,1) {$ D\left(\mc X\times_{\mc Z} \mc X\right) $};
			\node (b) at (1, 1) {$ D\left(\mc X\right) $};
			\node (c)  at (0,0) {$  D\left(\mc X\right) $};
			\node (d) at (1,0) {$ D\left(\mc Z\right) $};
			\node (e) at (0.2,0.75) {$  $};
			\node (f) at (-0.75,1.5) {$ \mc X $};
			\node (g) at (0.5,0.5) {\rotatebox{45}{$ \overset{Ex_{\#!}}{\Rightarrow} $}};

			\path[font=\scriptsize,>= angle 90]

			(a) edge [->] node [above ] {$ (p_{gf,2})\epfs $} (b)
			(a) edge [->] node [left] {$ (p_{gf,1})\epf  $} (c)
			(b) edge[->] node [right] {$ (gf)_! $} (d)
			(c) edge [->] node [below] {$ (gf)\epfs $} (d)
			(f) edge [dashed, bend right=-30,,->] node [below] {$ \Sigma_{gf} $} (b)
			(f) edge [bend left=-30, double equal sign distance] node [below] {$  $} (c)
			(f) edge [->] node [above, right] {$ \Delta_{gf}\epf  $} (a);
		\end{tikzpicture}
	\end{center}
	\noindent can be decomposed as:
	\begin{equation}\label{eq_formal_purity_diagram_1}
		\begin{tikzpicture}[baseline={(0,2.5)}, scale=2.5]
			
			\node (d0) at (-0.5,2.5) {$ D\left(\mc X\right) $};
			
			\node (a2) at (0,2) {$ D\left(\mc X_{\mc Z}^2\right) $};
			\node (a1) at (0,1) {$ D\left(\mc X\times_{\mc Z} \mc Y\right)  $};
			\node (a0) at (0,0) {$ D\left(\mc X\right) $};

			\node (b2) at (1,2) {$ D\left(\mc Y\times_{\mc Z} \mc X\right) $};
			\node (b1) at (1,1) {$ D\left(\mc Y^2_{\mc Z}\right) $};
			\node (b0) at (1,0) {$ D\left(\mc Y\right)  $};

			\node (c2) at (2,2) {$ D\left(\mc X\right) $};
			\node (c1) at (2,1) {$ D\left(\mc Y\right) $};
			\node (c0) at (2,0) {$ D\left(\mc Z\right) $};
			
			\node (s1) at (0.5,1.5) {\rotatebox{45}{$ \overset{Ex_{\#!}}{\Rightarrow} $}};
			\node (s2) at (1.5,1.5) {\rotatebox{45}{$ \overset{Ex_{\#!}}{\Rightarrow} $}};
			\node (s3) at (0.5,0.5) {\rotatebox{45}{$ \overset{Ex_{\#!}}{\Rightarrow} $}};
			\node (s4) at (1.5,0.5) {\rotatebox{45}{$ \overset{Ex_{\#!}}{\Rightarrow} $}};
			
			%\node (g) at (0.5,0.5) {$ \Delta $};

			\path[font=\scriptsize,>= angle 90]

			(d0) edge [dashed, bend left=30,->] node [above ] {$ \Sigma_{gf} $} (c2)
			(d0) edge [->] node [above, right ] {$ \Delta_{gf}\epf  $} (a2)
			(d0) edge [bend right=30,double equal sign distance] node [above ] {$  $} (a0)
			
			%horizontal
			
			(a2) edge [->] node [above ] {$ v_2\epfs $} (b2)
			(a1) edge [->] node [above ] {$ r_2\epfs $} (b1)
			(a0) edge [->] node [above ] {$ f\epfs $} (b0)
			
			(b2) edge [->] node [above ] {$ \left(p_{f^*g,2}\right)\epfs $} (c2)
			(b1) edge [->] node [above ] {$ p_{g,2}\epfs $} (c1)
			(b0) edge [->] node [above ] {$ g\epfs $} (c0)
			
			%vertical
			
			(a2) edge [->] node [left ] {$ v_1\epf $} (a1)
			(a1) edge [->] node [left ] {$ q_1\epf $} (a0)

			(b2) edge [->] node [left ] {$ r_1\epf $} (b1)
			(b1) edge [->] node [left ] {$ p_{g,1}\epf $} (b0)
			
			(c2) edge [->] node [right ] {$ f\epf $} (c1)
			(c1) edge [->] node [right ] {$ g\epf $} (c0);
		\end{tikzpicture}
	\end{equation}
	\noindent Now notice that we can decompose $\Delta_{gf}$ as a composite of $\Delta_f$ plus another map, given by the universal property of $\mc X^2_{\mc Z}$, that we will denote as $\widetilde{\Delta}$:
	\[ \mc X \overset{\Delta_f}{\longrightarrow} \mc X^2_{\mc Y} \overset{\widetilde{\Delta}}{\longrightarrow} \mc X^2_{\mc Z} \]
	the map $\widetilde{\Delta}$ fits in a cube made of cartesian squares of the form:
	
	\begin{center}
		\begin{tikzpicture}[baseline={(0,1.5)}, scale=1]

			\node (v1) at (0,0) {$ \mc X $};
			\node (v2) at (2,0) {$ \mc Y $};
			\node (v3) at (1,-1) {$ \mc X\times_{\mc Z}\mc Y $};
			\node (v4) at (3,-1) {$ \mc Y^2_{\mc Z} $};
			
			\node (v5) at (0,2) {$ \mc X^2_{\mc Y} $};
			\node (v6) at (2,2) {$ \mc X $};
			\node (v7) at (1,1) {$ \mc X^2_{\mc Z} $};
			\node (v8) at (3,1) {$ \mc Y\times_{\mc Z}\mc X $};

			% Draw arrows in the back
			\path[font=\scriptsize,>= angle 90]
			
			(v1) edge [->] node [ right=2mm, below=0.1mm ] {$f$} (v2)
			(v2) edge [->] node [ right=0.5mm, above=0.5mm ] {$\Delta_g$} (v4)
			
			(v6) edge [->] node [right=2mm, below=1mm] {$ f $} (v2)
			
			(v5) edge [->] node [ above ] {$ p_{f,2} $} (v6)
			(v6) edge [->] node [ right ] {$ \Delta_{f^*g} $} (v8)
			
			% Draw thick white lines on the front
			(v7) edge [-,line width=1mm, white] node [] {} (v3)
			(v7) edge [-, line width=1mm, white] node [] {} (v8)

			%Draw arrows in the front
			(v5) edge [->] node [ left ] {$ p_{f,1} $} (v1)
			(v7) edge [->] node [ left=2mm, above=1mm ] {$ v_1 $} (v3)
			(v8) edge [->] node [ left ] {$ r_1 $} (v4)
			
			(v5) edge [->] node [ left, below ] {$ \widetilde{\Delta} $} (v7)
			(v7) edge [->] node [ below=2mm, left=0.01mm ] {$ v_2 $} (v8)
			
			(v1) edge [->] node [ left=0.5mm ] {$ \Delta_{f^*g}' $} (v3)
			(v3) edge [->] node [ below ] {$ r_2 $} (v4);

		\end{tikzpicture}
	\end{center}
	
	Using the cube above, we can further decompose our diagram \eqref{eq_formal_purity_diagram_1} as:
	
	\begin{equation}\label{eq_formal_purity_diagram_2}
		\begin{tikzpicture}[baseline={(0,1)}, scale=1.75]
			
			\node (v0) at (-2,3) {$ D\left(\mc X\right) $};
			
			\node (v1) at (-0.5,-0) {$ D\left(\mc X\right) $};
			\node (v2) at (1.5,-0) {$ D\left(\mc Y\right) $};
			\node (v3) at (1,-1) {$ D\left(\mc X\times_{\mc Z}\mc Y\right) $};
			\node (v4) at (3,-1) {$ D\left(\mc Y^2_{\mc Z}\right) $};
			
			\node (v5) at (-0.5,2) {$ D\left(\mc X^2_{\mc Y}\right) $};
			\node (v6) at (1.5,2) {$ D\left(\mc X\right) $};
			\node (v7) at (1,1) {$ D\left(\mc X^2_{\mc Z}\right) $};
			\node (v8) at (3,1) {$ D\left(\mc Y\times_{\mc Z}\mc X\right) $};
			
			\node (v9) at (5,1) {$ D\left(\mc X\right) $};
			\node (v10) at (5,-1) {$ D\left(\mc Y \right) $};
			\node (v11) at (5,-3) {$ D\left(\mc Z\right) $};
			\node (v12) at (3,-3) {$ D\left( \mc Y\right) $};
			\node (v13) at (1,-3) {$ D\left(\mc X\right) $};

			\node (c1) at (0.5,1) {\rotatebox{45}{$ \scriptsize{\overset{Ex_{\#!}}{\Rightarrow}} $}};
			\node (c2) at (1.25,1.5) {\rotatebox{60}{$ \scriptsize{\overset{Ex_{\#!}}{\Rightarrow}} $}};
			\node (c3) at (2,0.15) {\rotatebox{45}{$ \scriptsize{\overset{Ex_{\#!}}{\Rightarrow}} $}};
			\node (c4) at (1.25,-0.5) {\rotatebox{60}{$ \scriptsize{\overset{Ex_{\#!}}{\Rightarrow}} $}};
			\node (c5) at (4,0.15) {\rotatebox{45}{$ \scriptsize{\overset{Ex_{\#!}}{\Rightarrow}} $}};
			\node (c6) at (2,-2) {\rotatebox{45}{$ \scriptsize{\overset{Ex_{\#!}}{\Rightarrow}} $}};
			\node (c7) at (4,-2) {\rotatebox{45}{$ \scriptsize{\overset{Ex_{\#!}}{\Rightarrow}} $}};

			% Draw arrows in the back
			\path[font=\scriptsize,>= angle 90]
			
			(v1) edge [->] node [ right=2mm, below=0.1mm ] {$f\epfs $} (v2)
			(v2) edge [->] node [ right=0.5mm, above=0.5mm ] {$\Delta_g\epf $} (v4)
			
			(v6) edge [->] node [left=2mm, above=1mm] {$ f \epf $} (v2)
			
			(v5) edge [->] node [ above ] {$ \left(p_{f,2}\right)\epfs $} (v6)
			(v6) edge [->] node [ right ] {$ \left(\Delta_{f^*g}\right)\epf  $} (v8)
			
			%equal arrows in the back
			(v2) edge [double equal sign distance] node [  ] {$  $} (v12)
			(v2) edge [dashed, ->] node [ above=2mm, right=-1mm ] {$ \Sigma_g $} (v10)
			
			% Draw thick white lines on the front
			(v7) edge [-,line width=1mm, white] node [] {} (v3)
			(v7) edge [-, line width=1mm, white] node [] {} (v8)
			(v3) edge [-,line width=1mm, white] node [] {} (v4)
			(v8) edge [-, line width=1mm, white] node [] {} (v4)

			%Draw arrows in the front
			(v5) edge [->] node [ left ] {$ \left(p_{f,1}\right)\epf $} (v1)
			(v7) edge [->] node [ left=2mm, above=1mm ] {$ v_1\epf $} (v3)
			(v8) edge [->] node [ left ] {$ r_1\epf $} (v4)
			
			(v5) edge [->] node [ left, below ] {$ \widetilde{\Delta}\epf $} (v7)
			(v7) edge [->] node [ above=2mm, right=0.01mm ] {$ \left(v_2\right)\epfs $} (v8)
			
			(v1) edge [->] node [ left=0.5mm, below=1mm ] {$ \left(\Delta_{f^*g}' \right)\epf $} (v3)
			(v3) edge [->] node [ above=2mm, right=-1mm ] {$ r_2\epfs $} (v4)
			
			% Diagonal maps
			(v0) edge [->] node [ left=1mm, below=1mm ] {$ \Delta_f\epf $} (v5)
			(v0) edge [dashed,->] node [ above ] {$ \Sigma_f $} (v6)
			(v0) edge [dashed, ->, bend left=30] node [ above ] {$ \Sigma_{gf} $} (v9)
			(v0) edge [double equal sign distance, bend right=30] node [  ] {$  $} (v13)
			(v6) edge [dashed, ->] node [ above ] {$ \Sigma_{f^*g} $} (v9)
			
			%horizontal maps
			(v4) edge [->] node [ below ] {$ (p_{g,2})\epfs $} (v10)
			(v8) edge [->] node [ above=2mm, left=-6mm ] {$ \left(p_{f^*g,2}\right)\epfs $} (v9)
			(v12) edge [->] node [  ] {$  $} (v11)
			(v13) edge [->] node [  ] {$  $} (v12)
			
			%vertical maops
			(v3) edge [->] node [  ] {$  $} (v13)
			(v4) edge [->] node [  ] {$  $} (v12)
			(v9) edge [->] node [  ] {$  $} (v10)
			(v10) edge [->] node [  ] {$  $} (v11)
			
			% Equal arrows
			(v0) edge [double equal sign distance] node [  ] {$  $} (v1)
			(v1) edge [double equal sign distance] node [  ] {$  $} (v13);

		\end{tikzpicture}
	\end{equation}
	\noindent Notice that the maps involved in the cube inside diagram \eqref{eq_formal_purity_diagram_2}:
	
	\begin{center}
		\begin{tikzpicture}[baseline={(0,1.5)}, scale=1.75]

			\node (v1) at (-0.5,-0) {$ D\left(\mc X\right) $};
			\node (v2) at (1.5,-0) {$ D\left(\mc Y\right) $};
			\node (v3) at (1,-1) {$ D\left(\mc X\times_{\mc Z}\mc Y\right) $};
			\node (v4) at (3,-1) {$ D\left(\mc Y^2_{\mc Z}\right) $};
			
			\node (v5) at (-0.5,2) {$ D\left(\mc X^2_{\mc Y}\right) $};
			\node (v6) at (1.5,2) {$ D\left(\mc X\right) $};
			\node (v7) at (1,1) {$ D\left(\mc X^2_{\mc Z}\right) $};
			\node (v8) at (3,1) {$ D\left(\mc Y\times_{\mc Z}\mc X\right) $};

			\node (c1) at (0.5,1) {\rotatebox{45}{$ \scriptsize{\overset{Ex_{\#!}}{\Rightarrow}} $}};
			\node (c2) at (1.25,1.5) {\rotatebox{60}{$ \scriptsize{\overset{Ex_{\#!}}{\Rightarrow}} $}};
			\node (c3) at (2,0.15) {\rotatebox{45}{$ \scriptsize{\overset{Ex_{\#!}}{\Rightarrow}} $}};
			\node (c4) at (1.25,-0.5) {\rotatebox{60}{$ \scriptsize{\overset{Ex_{\#!}}{\Rightarrow}} $}};

			% Draw arrows in the back
			\path[font=\scriptsize,>= angle 90]
			
			(v1) edge [->] node [ right=2mm, below=0.1mm ] {$f\epfs $} (v2)
			(v2) edge [->] node [ right=0.5mm, above=0.5mm ] {$\Delta_g\epf $} (v4)
			
			(v6) edge [->] node [left=2mm, above=1mm] {$ f \epf $} (v2)
			
			(v5) edge [->] node [ above ] {$ \left(p_{f,2}\right)\epfs $} (v6)
			(v6) edge [->] node [ right ] {$ \left(\Delta_{f^*g}\right)\epf  $} (v8)

			% Draw thick white lines on the front
			(v7) edge [-,line width=1mm, white] node [] {} (v3)
			(v7) edge [-, line width=1mm, white] node [] {} (v8)
			(v3) edge [-,line width=1mm, white] node [] {} (v4)
			(v8) edge [-, line width=1mm, white] node [] {} (v4)

			%Draw arrows in the front
			(v5) edge [->] node [ left ] {$ \left(p_{f,1}\right)\epf $} (v1)
			(v7) edge [->] node [ left=2mm, above=1mm ] {$ v_1\epf $} (v3)
			(v8) edge [->] node [ left ] {$ r_1\epf $} (v4)
			
			(v5) edge [->] node [ left, below ] {$ \widetilde{\Delta}\epf $} (v7)
			(v7) edge [->] node [ above=2mm, right=0.01mm ] {$ \left(v_2\right)\epfs $} (v8)
			
			(v1) edge [->] node [ left=0.5mm, below=1mm ] {$ \left(\Delta_{f^*g}' \right)\epf $} (v3)
			(v3) edge [->] node [ above=2mm, right=-1mm ] {$ r_2\epfs $} (v4);
			
		\end{tikzpicture}
	\end{center}

	\noindent  are all representable, hence all the exchange transformations on its faces are actually equivalences.
	The diagram \eqref{eq_formal_purity_diagram_2} can be slightly simplified as:
	\begin{equation}\label{eq:_final_complete_diagram_formal_purity}
		\begin{tikzpicture}[baseline={(2,1.65)}, scale=1.75]
			
			\node (v0) at (-2,3) {$ D\left(\mc X\right) $};
			
			\node (v1) at (-0.5,-0) {$ D\left(\mc X\right) $};
			\node (v2) at (1.5,-0) {$ D\left(\mc Y\right) $};
			
			\node (v4) at (3,-1) {$ D\left(\mc Y^2_{\mc Z}\right) $};
			
			\node (v5) at (-0.5,2) {$ D\left(\mc X^2_{\mc Y}\right) $};
			\node (v6) at (1.5,2) {$ D\left(\mc X\right) $};
			\node (v7) at (1,1) {$ D\left(\mc X^2_{\mc Z}\right) $};
			\node (v8) at (3,1) {$ D\left(\mc Y\times_{\mc Z}\mc X\right) $};
			
			\node (v9) at (5,1) {$ D\left(\mc X\right) $};
			\node (v10) at (5,-1) {$ D\left(\mc Y \right) $};
			\node (v11) at (5,-3) {$ D\left(\mc Z\right) $};
			\node (v12) at (3,-3) {$ D\left( \mc Y\right) $};
			\node (v13) at (1,-3) {$ D\left(\mc X\right) $};

			\node (c1) at (0.5,1) {\rotatebox{45}{$ \scriptsize{\overset{Ex_{\#!}}{\Rightarrow}} $}};
			\node (c2) at (1.25,1.5) {\rotatebox{60}{$ \scriptsize{\overset{Ex_{\#!}}{\Rightarrow}} $}};

			\node (c5) at (4,0.15) {\rotatebox{45}{$ \scriptsize{\overset{Ex_{\#!}}{\Rightarrow}} $}};
			
			\node (c7) at (4,-2) {\rotatebox{45}{$ \scriptsize{\overset{Ex_{\#!}}{\Rightarrow}} $}};

			% Draw arrows in the back
			\path[font=\scriptsize,>= angle 90]
			
			(v1) edge [->] node [ right=2mm, below=0.1mm ] {$f\epfs $} (v2)
			(v2) edge [->] node [ right=0.5mm, above=0.5mm ] {$\Delta_g\epf $} (v4)
			
			(v6) edge [->] node [left=2mm, above=1mm] {$ f \epf $} (v2)
			
			(v5) edge [->] node [ above ] {$ \left(p_{f,2}\right)\epfs $} (v6)
			(v6) edge [->] node [ right ] {$ \left(\Delta_{f^*g}\right)\epf  $} (v8)
			
			%equal arrows in the back
			(v2) edge [double equal sign distance] node [  ] {$  $} (v12)
			(v2) edge [dashed, ->] node [ above=2mm, right=-1mm ] {$ \Sigma_g $} (v10)
			
			% Draw thick white lines on the front
			
			(v7) edge [-, line width=1mm, white] node [] {} (v8)
			
			(v8) edge [-, line width=1mm, white] node [] {} (v4)

			%Draw arrows in the front
			(v5) edge [->] node [ left ] {$ \left(p_{f,1}\right)\epf $} (v1)
			
			(v8) edge [->] node [ left ] {$ r_1\epf $} (v4)
			
			(v5) edge [->] node [ left, below ] {$ \widetilde{\Delta}\epf $} (v7)
			(v7) edge [->] node [ above=2mm, right=0.01mm ] {$ \left(v_2\right)\epfs $} (v8)

			% Diagonal maps
			(v0) edge [->] node [ left=1mm, below=1mm ] {$ \Delta_f\epf $} (v5)
			(v0) edge [dashed,->] node [ above ] {$ \Sigma_f $} (v6)
			(v0) edge [dashed, ->, bend left=30] node [ above ] {$ \Sigma_{gf} $} (v9)
			
			(v6) edge [dashed, ->] node [ above ] {$ \Sigma_{f^*g} $} (v9)
			
			%horizontal maps
			(v4) edge [->] node [ below ] {$ (p_{g,2})\epfs $} (v10)
			(v8) edge [->] node [ above=2mm, left=-6mm ] {$ \left(p_{f^*g,2}\right)\epfs $} (v9)
			(v12) edge [->] node [ below ] {$ g\epfs $} (v11)
			(v13) edge [->] node [ below ] {$ f\epfs $} (v12)
			
			%vertical maps
			
			(v4) edge [->] node [  ] {$  $} (v12)
			(v9) edge [->] node [ right ] {$ f_! $} (v10)
			(v10) edge [->] node [ right ] {$ g_! $} (v11)
			
			% Equal arrows
			(v0) edge [double equal sign distance] node [  ] {$  $} (v1)
			(v1) edge [double equal sign distance] node [  ] {$  $} (v13);

		\end{tikzpicture}
	\end{equation}
	\noindent This last diagram encodes exactly the functoriality we were looking for:
	
	\begin{center}
		\begin{tikzpicture}[baseline={(0,0.5)}, scale=2]
			
			\node (a) at (0.75,1) {$  (g f)\epfs $};
			
			\node (c) at (3,1) {$ (g f)_! \Sigma_{gf} $};

			\node (d) at (0.75,0) {$ g\epfs  f\!\epfs $};
			\node (e) at (2,0) {$ g\epf  \Sigma_g f\epfs   $};
			\node (f) at (3,0) {$ g\epf \Sigma_g f_!\Sigma_f  $};

			%\node (g) at (1.25,0.75) {$ \ulcorner $};
			%\node (h) at (2.25,0.75) {$ \ulcorner $};

			\path[font=\scriptsize,>= angle 90]

			(a) edge [->] node [above] {$ \varphi_{g f} $} (c)
			
			(d) edge [->] node [below] {$ \varphi_g $} (e)
			(e) edge [->] node [below] {$  \varphi_f $} (f)
			
			(a) edge [double equal sign distance] node [left] {$  $} (d)
			
			(c) edge [->] node [left] {\rotatebox{-90}{$ \sim $}} (f);
			
		\end{tikzpicture}
	\end{center}
	
	\noindent Indeed from the upper triangular face of \eqref{eq:_final_complete_diagram_formal_purity}, as we did in \cref{Formal_twists_composition}, we get the equivalence:
	\[ \Sigma_{gf}\overset{\sim}{\longrightarrow} \Sigma_{f^*g}\Sigma_f \]
	
	\noindent From \Cref{lemma_Ex-sharp-shriek_mixed_repr_non_repr_case} and from the subdiagram of \eqref{eq:_final_complete_diagram_formal_purity} given by:
	
	\begin{center}
		\begin{tikzpicture}[baseline={(2,3)}, scale=1.5]

			\node (v2) at (1.5,-0) {$ D\left(\mc Y\right) $};
			
			\node (v4) at (3,-1) {$ D\left(\mc Y^2_{\mc Z}\right) $};

			\node (v6) at (1.5,2) {$ D\left(\mc X\right) $};
			
			\node (v8) at (3,1) {$ D\left(\mc Y\times_{\mc Z}\mc X\right) $};
			
			\node (v9) at (5,1) {$ D\left(\mc X\right) $};
			\node (v10) at (5,-1) {$ D\left(\mc Y \right) $};

			\node (c5) at (4,0.15) {\rotatebox{45}{$ \scriptsize{\overset{Ex_{\#!}}{\Rightarrow}} $}};

			% Draw arrows in the back
			\path[font=\scriptsize,>= angle 90]

			(v2) edge [->] node [ right=0.5mm, above=0.5mm ] {$\Delta_g\epf $} (v4)
			
			(v6) edge [->] node [left=2mm, above=1mm] {$ f \epf $} (v2)

			(v6) edge [->] node [ right ] {$ \left(\Delta_{f^*g}\right)\epf  $} (v8)
			
			(v2) edge [dashed, ->] node [ above ] {$ \Sigma_{g} $} (v10)
			% Draw thick white lines on the front

			(v8) edge [-, line width=1mm, white] node [] {} (v4)

			%Draw arrows in the front

			(v8) edge [->] node [ left ] {$ r_1\epf $} (v4)

			% Diagonal maps

			(v6) edge [dashed, ->] node [ above ] {$ \Sigma_{f^*g} $} (v9)
			
			%horizontal maps
			(v4) edge [->] node [ below ] {$ (p_{g,2})\epfs $} (v10)
			(v8) edge [->] node [ above=2mm, left=-6mm ] {$ \left(p_{f^*g,2}\right)\epfs $} (v9)
			(v9) edge [->] node [ right ] {$ f_! $} (v10);

		\end{tikzpicture}
	\end{center}
	
	\noindent we get that:
	\[ \Sigma_g f\epf \overset{\sim}{\longrightarrow}f\epf \Sigma_{f^*g} \]
	\noindent The naturality of exchange transformations $Ex_{\#!}$ in diagram \eqref{eq:_final_complete_diagram_formal_purity} will then fill the following:
	
	\begin{center}
		\begin{tikzpicture}[baseline={(0,0.5)}, scale=2]
			
			\node (a) at (0.75,2) {$  (g f)\epfs $};
			
			\node (c) at (3,2) {$ (g f)_! \Sigma_{gf} $};
			
			\node (c1) at (3,1) {$ (g f)_! \Sigma_{f^*g}\Sigma_{f} $};

			\node (d) at (0.75,0) {$ g\epfs  f\!\epfs $};
			\node (e) at (2,0) {$ g\epf  \Sigma_g f\epfs   $};
			\node (f) at (3,0) {$ g\epf \Sigma_g f_!\Sigma_f  $};

			%\node (g) at (1.25,0.75) {$ \ulcorner $};
			%\node (h) at (2.25,0.75) {$ \ulcorner $};

			\path[font=\scriptsize,>= angle 90]

			(a) edge [->] node [above] {$ \varphi_{g f} $} node [below] {$ \sim $}  (c)
			
			(d) edge [->] node [below] {$ \varphi_g $} (e)
			(e) edge [->] node [below] {$  \varphi_f $} node [above] {$ \sim $} (f)
			
			(a) edge [double equal sign distance] node [left] {\rotatebox{-90}{$  $}} (d)
			
			(c) edge [->] node [left] {\rotatebox{-90}{$ \sim $}} (c1)
			(c1) edge [<-] node [left] {\rotatebox{-90}{$ \sim $}} (f);
			
		\end{tikzpicture}
	\end{center}
	\noindent where $\varphi_f, \varphi_{gf}$ are equivalence by \Cref{Repr_lft_Ayoub_Purity_after_AGV}, and this concludes the proof.
\end{proof}

\begin{theorem}[Non-Representable Ambidexterity]\label{Thm:_stacky_Ambidexterity}
	Let $f: \mc X \longrightarrow \mc Y$ be a smooth map between NL-stacks. Then the following natural transformation:
	\[ \varphi_f: f\epfs \overset{}{\longrightarrow} f_!\Sigma_f \]
	\noindent is an equivalence.
\end{theorem}
\begin{proof}
	Let $x: X \rightarrow \mc X$ be a NL-atlas, where $X$ can be chosen to be a scheme. Then $g:=fx$ is representable since the source lives in ${Sch}$ (cf. \cite[Corollary 3.2.3]{Alper_Book}).  By \cite[5.5.3.4]{HTT} and by the conservativity of $x^*$, it is easy to see that $\varphi_f$ is an equivalence if and only if:
	\[ \varphi_f x\epfs: f\epfs x\epfs \longrightarrow f_! \Sigma_f x\epfs \]
	\noindent is an equivalence. But by \Cref{Thm:_Ayoub_1.7.3+}, $\varphi_f x\epfs $ sits in a commutative diagram where all the other natural transformations are equivalences and therefore it must be an equivalence itself, and we are done.
\end{proof}
And finally we are now able to prove our main result:
\begin{corollary}[Non-Representable Purity]\label{Sec.4:_Stacky_Purity}
	Let $f: \mc X \longrightarrow \mc Y$ be a smooth map between NL-stacks. Then:
	\[ \varphi_f: f\epfs \overset{}{\longrightarrow} f_!\Sigma^{\mb L_f} \]
	\noindent is an equivalence.
\end{corollary}
\begin{proof}
	By \Cref{Thm:_stacky_Ambidexterity}, the only part left to prove is the identification of the formal twist $\Sigma_f$ with $\Sigma^{\mb L_f}$. Let $\Delta_f$ be the diagonal of $f$ and let $p_2: \mc X^2_{\mc Y}\rightarrow \mc X$ be the second projection. As we have seen at the beginning in \eqref{eq:_Homotopy_Purity_Roof}, unravelling the definitions, we have a roof of maps:
	\begin{equation}\label{eq:_final_HP_roof_diagonal}
		\Sigma_{f}=\mr{Th}(\Delta_f,p_2) \overset{\Pi_1(\Delta_{f},p_2)}{\longrightarrow} (r_S)\epfs \mr{Th}(\hat{\Delta_f},\hat{p_2})r_X^* \overset{\Pi_0(\Delta_{f},p_2)}{\longleftarrow} q_{\#}\Sigma^{\mb L_f}
	\end{equation}
	Let $x: X \rightarrow \mc X$ be any NL-atlas for $\mc X$. Let $ \Delta_{x^*f} $ and $\pi_2$ be the pullbacks along $x$ of $\Delta_f$ and $p_2$ respectively. Then it is not difficult to see that applying $x^*$ to \eqref{eq:_final_HP_roof_diagonal} we get:
	
	\[ \Sigma_{x^*f}:=\mr{Th}(\Delta_{x^*f},\pi_2) \overset{\Pi_1(\Delta_{x^*f},\pi_2)}{\longrightarrow} (r_S)\epfs \mr{Th}(\hat{\Delta_{x^*f}},\hat{\pi_2})r_X^* \overset{\Pi_0(\Delta_{x^*f},\pi_2)}{\longleftarrow} q_{\#}x_{\#}\Sigma^{x^*\mb L_f} \]
	
	\noindent By \Cref{Thm:_Repr-lft_HP} both $\Pi_1(\Delta_{x^*f},\pi_2)$ and $\Pi_0(\Delta_{x^*f},\pi_2)$ must be equivalences. But by the conservativity of $x^*$, this implies that the maps $\Pi_i(\Delta_{f},p_2)$ in \eqref{eq:_final_HP_roof_diagonal} are also equivalences and hence:
	
	\[ \Sigma_f \simeq \Sigma^{\mb L_f} \]
	\noindent as we wanted.
	
\end{proof}

Let us finish now with a recap of what we proved up to now for the stable motivic category of algebraic stacks:

\begin{theorem}\label{Final_recap_thm_1-alg_stk}
    We have a functor:
	\[ \mr{SH}^*_!: Corr\left( \oocatname{ASt}^{\leq 1} \right)_{lft,all} \longrightarrow \oocatname{Pr}_{stb}^{\mr L} \]
	\noindent extending the analogous functor defined on schemes. This functor encodes the following data:
	
	\begin{enumerate}
		\item For every $ \mc X $ algebraic stack, we have the tensor- hom adjunction in $ \mr{SH}(\mc X) $:
		\[ -\otimes - \dashv \iMap_{\mr{SH}(\mc X)}(-,-) \]
		\item For any map $ f: \mc X \longrightarrow \mc Y $ in $ \oocatname{ASt}^{\leq 1} $, we have a pair of adjoint functors:
		\[ f^*\dashv f^* \]
		\noindent and if $ f $ is smooth we also have:
		\[ f_{\#}\dashv f^* \dashv f_* \]
		\item For $ f: \mc X \longrightarrow \mc Y $ a lft map in $ \oocatname{ASt}^{\leq 1} $, we have another pair of adjoint functors:
		\[ f_!\dashv f^! \]
        \item Given a cartesian diagram in $ \oocatname{ASt}^{\leq 1} $: 
		\begin{center}
		\begin{tikzpicture}[baseline={(0,-1)}, scale=1.25]
			\node (a) at (0,1) {$ \mc W $};
			\node (b) at (1, 1) {$ \mc Y $};
			\node (c)  at (0,0) {$  \mc Z$};
			\node (d) at (1,0) {$ \mc X $};
			\node (e) at (0.2,0.75) {$ \ulcorner $};
			\node (f) at (0.5,0.5) {$ \Delta $};

			\path[font=\scriptsize,>= angle 90]
			
			(a) edge [->] node [above ] {$ g $} (b)
			(a) edge [->] node [left] {$ q $} (c)
			(b) edge[->] node [right] {$ p $} (d)
			(c) edge [->] node [below] {$ f $} (d);
		\end{tikzpicture}
	\end{center}
        \noindent where $p$ is lft, we have the base change equivalence:
        \[ p_!f^*\overset{Ex^*_!}{\simeq} g^*q_! \]
        \item For any lft map $f$, the projection formula holds:
        \[ f_!(-)\otimes - \simeq f_!( - \otimes f^*(-))  \]
        \item For any smooth map $f$, the smooth projection formula holds:
        \[ f\epfs(-)\otimes - \simeq f\epfs( - \otimes f^*(-))  \]
	\end{enumerate}
These functors satisfy the usual compatibilities given by exchange transformations:
$$ Ex^*_!,Ex^!_*, Ex^*_{\#}, Ex^*_*,Ex_{\# *}, Ex_{!*}, Ex_{!\#}, Ex^{*!} $$
\noindent like in \Cref{App:_Exchange_Transf}, and they give rise to localisation fiber sequences like in \cite[Proposition 4.2.1]{Chowdhury}. Moreover, for a given smooth map $f$, we have the following natural purity equivalence:
\[ f\epfs \longrightarrow f_!\Sigma^{\mb L_f} \]
\end{theorem}

\begin{proof}
    The functor $ \SH^*_! $ was constructed in \cref{Sect.3.2:_SH*!_Corr_functor}, and, as explained in Appendix \ref{App.A.2:_6FF}, it encodes the existence of the six functors, and the projection formula in (5). The existence of $f\epfs$, for a given smooth map $f$, follows from our construction of $\SH_{cl}$. The smooth base change and smooth projection formula were proved in \cref{Sec.5:_smooth_cl_BC} and \cref{smooth_projection_formula}, respectively. The exchange transformations were recorded in \cref{App:_Exchange_Transf}, and purity was proved in \cref{Sec.4:_Stacky_Purity}.
\end{proof}

\section{Final Remarks}\label{Applications}
\subsection{Higher Derived Algebraic Stacks}
As mentioned in the introduction, all the material presented here can easily be adapted to work for higher derived algebraic stacks. We will give here a brief overview of the key ingredients needed to extend by induction all the material previously presented. Since we were already quite detailed in what we did before, we will now give details only on the less trivial parts of the extension of our theorems to higher derived algebraic stacks. The reader can fill the gaps replacing in the arguments we gave \textquotedblleft \textit{representable} \textquotedblright by \textquotedblleft $(n-1)$-representable\textquotedblright  and \textquotedblleft \textit{algebraic stack}\textquotedblright  by \textquotedblleft $n$-algebraic (derived) stack\textquotedblright, and use induction where necessary. All the \textit{underived} classical results we used for motivic spaces and for the motivic homotopy category have their derived counterparts statements worked out in \cite{Adeel_PhD}. But as we will see below, one can easily reduce to the underived case taking truncations.

 We will assume a basic knowledge of derived geometry as developed in \cite{SAG, HAG-II}. Introduction to the subject with all the relevant definitions can also be found in \cite{Khan-Intro-DAG, Hekking-Rydh-Savvas}.\\

We will denote the category of $n$-algebraic derived stacks as $\oocatname{dASt}^{\leq n}$. We have a natural inclusion of $n$-algebraic stacks into derived ones $\oocatname{ASt}^{\leq n}$. Such inclusion admits a right adjoint given by the truncation:

\[ \begin{array}{ccccc}
	\iota: & \oocatname{ASt}^{\leq n}& \rightleftarrows & \oocatname{dASt}^{\leq n}& :(-)_{cl}\\
	& \mc X_{cl} & \leftarrow & \mc X &
\end{array} \]

We will denote by $\times^{h}$ the homotopy pullback in $\oocatname{dASt}^{\leq n}$. Let us first start with a derived version of our \Cref{NL-ASt=ASt}:

\begin{theorem}
	Any $n$-algebraic derived stack $\mc X\in \oocatname{dASt}^{\leq n}$ admits a NL-cover.
\end{theorem}
\begin{proof}
	Notice that one can define NL-maps and atlases as done in the underived case. As in the classical case, we can interchange the NL-topology and NL-atlases with smooth-Nis ones: one can reduce to the truncations of our derived object to check that the two definition agree. We will proceed in steps as in the proof of \Cref{S2:_Alper_LMB}.\\
	
	Let us first suppose that $\mc X=X $ is a derived algebraic space. Let us consider an étale cover $p: U \rightarrow X$ where $U$ is a disjoint union of affine derived schemes and let us suppose that we are given a $x \in X(\sk)$. We want to find a smooth-Nis cover, so we want to lift $x$ to some schematic cover of $X$. Denote by $U_x$ the following fiber product:
	\begin{equation}\label{Alper_LMB_diagram}
		\begin{tikzpicture}[baseline={(0,0)}, scale=1.5]
			\node (a) at (0,1) {$ U_x $};
			\node (b) at (1, 1) {$  U $};
			\node (c)  at (0,0) {$ \spec(\sk) $};
			\node (d) at (1,0) {$\mc X$};
			
			\node (e) at (0.25,0.75) {$\ulcorner $};
			
			\path[font=\scriptsize,>= angle 90]
			
			(a) edge [->] node [above ] {$  $} (b)
			(a) edge [->] node [left] {$  $} (c)
			(b) edge [->] node [right] {$  $} (d)
			(c) edge [->] node [below] {$ x $} (d);
			
		\end{tikzpicture}
	\end{equation}

	Now we have that $U_x$ is smooth (actually \'etale) over $\spec(\sk)$ and hence it is a classical (underived) scheme. Then by \cite[\href{https://stacks.math.columbia.edu/tag/056U}{Tag 056U}]{stacks-project}, there exists a closed point $ u: \spec(\kappa(u))\into U_x $ such that $\bigslant{\kappa(u)}{\sk}$ is a finite separable extension of some degree d. Let $\left(U/ X\right)^d$ be the $d$\textit{th}-fiber product equipped with the permuting action of the symmetric group $\mf{S}_d$. Notice that $\left(U/ X\right)^d$ coincides with the homotopical $d$\textit{th}-fiber product since $p$ is in particular flat. Then, as in the proof of \Cref{S2:_Alper_LMB}, the point $u \in U_x(\kappa(u))$ gives us a lift of $x$:
	
	\begin{center}
		\begin{tikzpicture}[baseline={(0,0)}, scale=1.5]
			\node (a) at (1,0.7) {$  \mr{\Acute{E}T}^d\!\!\left(U/ X\right) $};
			\node (b) at (1, 0) {$   X  $};
			\node (c)  at (0,0) {$\spec(\sk) $};

			\path[font=\scriptsize,>= angle 90]
			
			(a) edge [->] node [right ] {$  $} (b)
			(c) edge [->] node [below] {$ x $} (b)
			(c) edge [dashed, ->] node [left] {$  $} (a);
			
		\end{tikzpicture}
	\end{center}
	
	 Choosing a faithful representation $\mf S_d \sseq GL_n$, we can rewrite $\mr{\Acute{E}T}^d\!\!\left(U/ X\right)\simeq \left[\bigslant{V_d(\bigslant{U}{ X})}{GL_n}\right]$, where $V_d(\bigslant{U}{ X}):= (\bigslant{U}{ X})^d\times^{h,\mf S_d} GL_n $. Since $GL_n$ is special, the point $v \in \mr{\Acute{E}T}^d\!\!\left(U/ X\right)$ admits a lifting to $V_d(\bigslant{U}{ X})$:
	\begin{center}
		\begin{tikzpicture}[baseline={(0,0)}, scale=1.5]
			\node (a) at (1,0.7) {$ V_d(\bigslant{U }{ X})$};
			\node (b) at (1, 0) {$   X  $};
			\node (c)  at (0,0) {$\spec(\sk) $};

			\path[font=\scriptsize,>= angle 90]
			
			(a) edge [->] node [right ] {$  $} (b)
			(c) edge [->] node [below] {$ x $} (b)
			(c) edge [dashed, ->] node [left] {$  $} (a);
			
		\end{tikzpicture}
	\end{center}
	
	One can check that $V_d(\bigslant{U }{ X})$ is actually a scheme. Indeed, $V_d(\bigslant{U }{ X})$ is a scheme if and only if its classical truncation is a scheme. But since $\mf S_d$ is étale, we have that:
		\begin{center}
		\begin{tikzpicture}[baseline={(0,0)}, scale=2]
			\node (a) at (0,1) {$ (U^d_X)_{cl}\times GL_n $};
			\node (b) at (1.5, 1) {$  U^d_X\times^{h} GL_n $};
			\node (c)  at (0,0) {$ V_d(\bigslant{U }{\mc X})_{cl} $};
			\node (d) at (1.5,0) {$ V_d(\bigslant{U }{\mc X}) $};
			
			\node (e) at (0.25,0.75) {$\ulcorner $};
			
			\path[font=\scriptsize,>= angle 90]
			
			(a) edge [closed] node [above ] {$  $} (b)
			(a) edge [->] node [left] {$  $} (c)
			(b) edge [->] node [right] {$  $} (d)
			(c) edge [closed] node [below] {$ $} (d);
			
		\end{tikzpicture}
	\end{center}
	\noindent is a homotopy pullback. Using the fact that $ (U^d_X)_{cl}\times GL_n \rightarrow V_d(\bigslant{U }{\mc X})_{cl} $ is a $\mf S_d$-torsor, we must have that:
	\[ V_d(\bigslant{U }{\mc X})_{cl}\simeq (U^d_X)_{cl}\times^{\mf S_d} GL_n  \]
	\noindent By \cite[\href{https://stacks.math.columbia.edu/tag/07S6}{Tag 07S6}]{stacks-project}, we get that $V_d(\bigslant{U }{\mc X})_{cl}$ is represented by a scheme and hence $ V_d(\bigslant{U }{\mc X}) $ is actually a derived scheme as we wanted. Now to get the smooth-Nis atlas, one considers the family $ \set{V_n(\bigslant{U }{ X})}_{n \in \mb N} $\\
	
	For the second step, consider $\mc X \in \oocatname{dASt}^{\leq 1}$ a $1$-algebraic derived stack. Take any smooth cover $p: U \rightarrow \mc X$ and, given any point $x\in \mc X(\sk)$, consider the same diagram as in \eqref{Alper_LMB_diagram}. The fiber product $U_x$ will now be a derived algebraic space, smooth over $\sk$. But by the first step of the proof, we can assume that $U_x$ is actually a derived scheme (up to taking further coverings). So, as we did before we can find an element in $ V_d(\bigslant{U }{ \mc X})(\sk) $ lifting $x\in \mc X(\sk)$. But now $ V_d(\bigslant{U }{\mc  X}) $ is an algebraic space and hence we can further reduce by step one to the case where $V_d(\bigslant{U }{\mc  X})$ is actually represented by a scheme. And again one concludes considering the family $\set{V_n(\bigslant{U }{ \mc X})}_{n \in \mb N}$.\\
	
	For a general $n$-algebraic derived stack $\oocatname{dASt}^{\leq n}$ with a given $\sk$-point $x \in \mc X(\sk)$, by induction on $n$, one repeats the same argument as in step two in order to lift $x$ to $ V_d(\bigslant{U }{\mc  X}) $, that is now $(n-1)$-algebraic and derived. So by induction one can suppose that $V_d(\bigslant{U }{\mc  X})$ is actually a derived scheme and conclude as in all the previous steps.
\end{proof}

One can construct motivic spaces (and the associated stable motivic homotopy category) for $\mc X \in \oocatname{dASt}^{\leq n}$ considering $\A^1$-invariant, NL-sheaves on the category $\mr{Lis}_{\bigslant{}{\mc X}}$ of $\mc X$-smooth $n$-algebraic derived stacks, mimicking what we did in Section 3. By induction on the algebraicity degree $n$, one can show, as we did for $1$-algebraic stacks, that this motivic categories satisfy NL-descent and hence they coincide with their Kan extensions along NL-\v Cech covers. As for derived schemes, $\mc H(\mc X)$ and $\SH(\mc X)$ will be nil-invariant:

\begin{proposition}[Nil-invariance]
	Let $\mc X\in \oocatname{dASt}^{\leq n}$. Then we have natural equivalences of $\infty$-categories:
	\[ \mc H(\mc X_{cl})\simeq \mc H(\mc X) \]
	\[ \SH(\mc X_{cl})\simeq \SH(\mc X) \]
	\noindent induced by the closed immersion $\mc X_{cl }\into \mc X$.
\end{proposition}
\begin{proof}
	It is a simple consequence of induction and NL-descent. Let us do the case $n=0$: suppose that $\mc X=X$ is a derived algebraic space. Then using NL-descent, we can see that $ \H(X) $ and $\H(X_{cl})$ (resp. $ \SH(X) $ and $\SH(X_{cl})$) are Kan extended from $\H(-)$ on derived schemes and hence they must coincide by \cite[Corollary 1.3.5]{Adeel_PhD} and by uniqueness of Kan extensions. In other words, if $U \rightarrow X$ we get a chain of equivalences:
	\[ \H(X)\simeq \lim_{n \in \Delta} \H(U^n_X)\simeq \lim_{n \in \Delta} \H(U^n_{cl, X_{cl}})\simeq \H(X_{cl}) \]
	\noindent where the middle equivalence is a consequence of \cite[Corollary 1.3.5]{Adeel_PhD} (or of a very special case of \cite[Corollary 5.9]{Khan-Ravi_Generalised_Coh_Stacks}).\\
	Now suppose we know our claim for all $k$-algebraic derived stacks such that $k<n$. If $\mc X$ is $n$-algebraic, we can find a NL-atlas $U\rightarrow \mc X$ that will be authomatically $(n-1)$-representable. Then as we did before, we get by induction hypothesis a chain of equivalences:
	\[ \H(\mc  X)\simeq \lim_{n \in \Delta} \H(U^n_{\mc X})\simeq \lim_{n \in \Delta} \H(U^n_{cl, \mc X_{cl}})\simeq \H(\mc X_{cl}) \]
	\noindent and hence our claim. The same argument clearly works also for $\SH(-)$.
\end{proof}

In the end we get that:

\begin{theorem}
	We have a functor:
	\[ \mr{SH}^*_!: Corr\left( \oocatname{dASt}^{\leq n} \right)_{lft,all} \longrightarrow \oocatname{Pr}_{stb}^{\mr L} \]
	\noindent extending the analogous functor defined on derived schemes, together with a factorisation of $SH^*_!$:
	\[ Corr\left( \oocatname{dASt}^{\leq n} \right)_{lft,all} \longrightarrow Corr\left( \oocatname{ASt}^{\leq n} \right)_{lft,all} \stackrel{\restrict{SH^*_!}{\oocatname{ASt}^{\leq n}}}{\longrightarrow} \oocatname{Pr}_{stb}^{\mr L} \]
	\noindent induced by the truncation functor. This new functors will have all the analogous properties as in \Cref{Final_recap_thm_1-alg_stk}.
\end{theorem}
\begin{proof}
	Using nil-invariance, we can already suppose to work with underived $n$-algebraic stacks. Then one proceeds by induction on $n$, using the same argument we already gave in the previous sections of this paper, to show all the relevant properties. There is one exception to this: the homotopy purity identification for $f$ smooth of $\Sigma_f$ with $\Sigma^{\mb L_f}$ follows from \Cref{s5:_htpy_purity_n-algebraic} instead.
\end{proof}

\begin{theorem}\label{s5:_htpy_purity_n-algebraic}
	Let $n>1$ be an integer and let $g: \mc X \rightarrow \mc Y$ be a $(n-1)$-representable map of derived $n$-algebraic stacks, with $\mc X, \mc Y$ smooth over some base $\mc S\in \oocatname{dASt}^{\leq n}$. Denote by $q,p$ the structure maps of $\mc X$ and $\mc Y$ respectively. Then we have that:
	\[ \mr{Th}(g,p):=p\epfs g_!\simeq q\epfs \Sigma^{\mc N_g}  \]
	\noindent where $\mc N_g:= \mb L_g[-1]$ is the conormal complex.
\end{theorem}
\begin{remark}
	Recall that the main ingredient of our proof of homotopy purity for 0-representable (underived) maps was \Cref{Key_Lemma:_schematic_HP+}. There we used purity for 0-representable maps and homotopy purity for closed immersions. But closed immersions are affine maps, that can be thought as $(-1)$-representable maps. If HP(k) stands for \textquotedblleft homotopy purity for $k$-representable maps\textquotedblright, A(h) stands for \textquotedblleft ambidexterity for $h$-representable maps\textquotedblright and P(l) stands for \textquotedblleft  purity for $l$-representable maps\textquotedblright, then our proof of purity for 1-algebraic stacks can be summarised as:
	\[ (\mr{HP}(-1)\wedge \mr{P}(0)\Rightarrow \mr{HP}(0))\wedge (A(0)\Rightarrow A(1)) \Rightarrow P(1) \]
	The induction strategy then is to get the following implications:
	\[  (\mr{HP}(n-2)\wedge \mr{P}(n-1)\Rightarrow \mr{HP}(n-1))\wedge (A(n-1)\Rightarrow A(n)) \Rightarrow P(n) \]
	The upcoming proof shows exactly that:
	\[ \mr{HP}(n-2)\wedge \mr{P}(n-1)\Rightarrow \mr{HP}(n-1) \]
\end{remark}
\begin{proof}
	We will go by induction on $n$, using purity: we already know purity for $n=0,1$ by nil-invariance and \Cref{Sec.4:_Stacky_Purity} and we will assume by induction hypothesis to know that purity holds for smooth $(n-1)$-representable maps (and that homotopy purity holds for $(n-2)$-representable maps between smooth algebraic stacks).\\
	
	By nil-invariance we can reduce ourselves to the case where $\mc X,\mc Y,\mc S$ are all underived. By NL-separation, we can assume that $\mc X$ is $(n-1)$-algebraic and that $\mc Y=Y, \mc S=S$ are schemes. Let us consider $x: X \rightarrow \mc X$ a NL atlas and consider the following diagrams coming from the functoriality of the deformation space construction in \cite{Aranha-Pstragowski}:
		\begin{center}
		\begin{tikzpicture}[baseline={(0,0.5)}, scale=2.5]
			\node (a) at (0.5,1) {$  X $};
			\node (d) at (0.5,0.5) {$  \mc X $};
			\node (b) at (1,0.5) {$  Y $};
			\node (c) at (0.5,0) {$ S $};
			\node (e) at (1,1) {$  Y $};
			
			\node (f) at (0.75,0.75) {$  \circlearrowleft $};
			
			\node (a1) at (2.5,1) {$ \A^1_X $};
			\node (d1) at (2.5,0.5) {$ \A^1_{\mc X} $};
			\node (b1) at (3,0.5) {$ \mr{Def}_g $};
			
			\node (c1) at (2.5,0) {$ S $};
			\node (e1) at (3,1) {$  \mr{Def}_{\omega} $};
			
			\node (f1) at (2.75,0.75) {$  \circlearrowleft $};

			\node (a2) at (4.5,1) {$ X $};
			\node (d2) at (4.5,0.5) {$ \mc X $};
			\node (b2) at (5,0.5) {$ N_g $};
			\node (c2) at (4.5,0) {$ S $};
			\node (e2) at (5,1) {$  N_{\omega}  $};
			
			\node (f2) at (4.75,0.75) {$  \circlearrowleft $};
			
			\path[font=\scriptsize,>= angle 90]
			
			(a) edge [->] node [left ] {$ x $} (d)
			(d) edge [->] node [below ] {$ g $} (b)
			(a) edge [->] node [above ] {$ \omega $} (e)
			(e) edge [double equal sign distance] node [above ] {$  $} (b)
			
			%(a) edge [->] node [left] {$q$} (c)
			(d) edge [->] node [left] {$ q $} (c)
			(b) edge [->] node [right=2mm, below=-0.5mm] {$p$} (c)

			([yshift=0mm]b.east) edge [shorten >=8mm,shorten <=8mm,->] node [below] {$ i_1 $} ([yshift=0mm]d1.west)

			(a1) edge [->] node [above ] {$ \hat{\omega} $} (e1)
			(e1) edge [->] node [right ] {$ \delta $} (b1)
			(a1) edge [->] node [left ] {$ \xi $} (d1)
			(d1) edge [->] node [below ] {$ \hat{g} $} (b1)
			
			%(a1) edge [->] node [left] {$ $} (c1)
			(d1) edge [->] node [left] {$ $} (c1)
			(b1) edge [->] node [right] {$ r_S\circ \hat p $} (c1)
			
			([yshift=0mm]b1.west) edge [shorten >=11mm,shorten <=18mm,<-] node [below=2.5mm, right=2.5mm] {$ i_0 $} ([yshift=0mm]d2.east)

			(a2) edge [->] node [above ] {$ \omega_0 $} (e2)
			(e2) edge [->] node [right ] {$ \gamma $} (b2)
			(a2) edge [->] node [left ] {$ x $} (d2)
			(d2) edge [->] node [below ] {$ g_0 $} (b2)
			
			%(a2) edge [->] node [left] {$ q $} (c2)
			(d2) edge [->] node [left] {$  $} (c2)
			(d2) edge [->] node [left] {$  $} (c2)
			(b2) edge [->] node [right] {$ p_0 $} (c2);
			
		\end{tikzpicture}
	\end{center}
	
	Our claim will follow once we show that the natural transformations (defined as in Section 4):
	\begin{equation}\label{S5:_HP_higher_eq1}
		p\epfs g_!=:\mr{Th}(g,p) \overset{\Pi_1(g,p)}{\longrightarrow} (r_S)\epfs \mr{Th}(\hat{g},\hat{p})r_{\mc X}^* \overset{\Pi_0(f,p)}{\longleftarrow} \mr{Th}(g_0,p_0)\simeq q\epfs \Sigma^{\mc N_g}
	\end{equation}
	\noindent are both equivalences. Using the conservativity of $x^!$ and by \cite[5.5.3.4]{HTT}, one can show that $\Pi_1(g,p), \Pi_0(g,p)$ are equivalences if and only if the following natural transformations are equivalences:
	\begin{equation}\label{S5:_HP_higher_eq2}
		p\epfs \omega_!\simeq\mr{Th}(g,p)x_! \overset{\Pi_1(g,p)x_!}{\longrightarrow} (r_S)\epfs \mr{Th}(\hat{g},\hat{p})r_{\mc X}^*x_! \overset{\Pi_0(f,p)x_!}{\longleftarrow} q_{\#}\mr{Th}(g_0,p_0)x_!\simeq q\epfs \Sigma^{\mc N_g}x_!
	\end{equation}
	By induction hypothesis, we know that purity holds for $x$:
	\[ x_!\simeq x\epfs \Sigma^{-\mb L_{x}}\simeq  x\epfs \Sigma^{\mc N_{x}} \]
	 \noindent So we can rewrite the RHS of \eqref{S5:_HP_higher_eq2} as:
	\[ q\epfs \Sigma^{\mc N_g}x_!\simeq q\epfs \Sigma^{\mc N_g}x\epfs \Sigma^{\mc N_x}\simeq q\epfs x\epfs \Sigma^{x^*\mc N_g}\Sigma^{\mc N_x}\simeq q\epfs x\epfs \Sigma^{\mc N_{\omega}}    \]
	\noindent where the last equivalence follows from the functoriality of the Borel J-homomorphism and from the standard fiber sequence in $\mr{K}(X)$:
	\[ x^*\mc N_g \rightarrow \mc N_{\omega} \rightarrow \mc N_x \]
	Let us denote, for later use, by $\Phi_x$ the composition of equivalences we showed before:
	\[ \Phi_x: q\epfs \Sigma^{\mc N_g}x_!\overset{\sim}{\underset{\varphi_x}{\longrightarrow}} q\epfs \Sigma^{\mc N_g}x\epfs \Sigma^{\mc N_x}\simeq q\epfs x\epfs \Sigma^{x^*\mc N_g}\Sigma^{\mc N_x}\simeq q\epfs x\epfs \Sigma^{\mc N_{\omega}}  \]
	From the diagrams:\\
	\begin{minipage}{0.5\textwidth}
		\begin{center}
			\begin{tikzpicture}[baseline={(0,1)}, scale=1.25]
				
				\node (a) at (0,1) {$  \A^1_X $};
				
				\node (b) at (1,1) {$ X  $};
				
				%	\node (c) at (5,1) {$  \mr{Th}(\iota_0,p_0\circ \pi_0)\Sigma^{-\mb L_{\pi}} $};
				
				\node (d) at (0,0) {$  \A^1_{\mc X} $};
				
				\node (e) at (1,0) {$ \mc X  $};
				
				%	\node (f) at (5,0) {$  \mr{Th}(f_0,p_0) $};

				\node (g) at (0.25,0.75) {$ \ulcorner $};
				%\node (h) at (2.25,0.75) {$ \ulcorner $};

				\path[font=\scriptsize,>= angle 90]

				(a) edge [->] node [above] {$ r_X $} (b)
				%	(c) edge [->] node [above] {$ \Pi_0(\iota,p\circ \pi) $} (b)
				(d) edge [->] node [below] {$ r_{\mc X} $} (e)
				%	(f) edge [->] node [below] {$ \Pi_0(f,p) $} (e)
				(a) edge [->] node [left] {$\xi$} (d)
				
				(b) edge [->] node [right] {$x$} (e);
				
				%	(c) edge [->] node [right] {\rotatebox{-90}{$ \sim  $}} (f);
				
			\end{tikzpicture}
		\end{center}
	\end{minipage}
	\begin{minipage}{0.2\textwidth}
		\begin{center}
			\begin{tikzpicture}[baseline={(0,1)}, scale=1.25]
				
				\node (a) at (0,1) {$  \A^1_X $};
				
				\node (b) at (1,1) {$ \mr{Def}_{\omega} $};
				
				%	\node (c) at (5,1) {$  \mr{Th}(\iota_0,p_0\circ \pi_0)\Sigma^{-\mb L_{\pi}} $};
				
				\node (d) at (0,0) {$  \A^1_{\mc X} $};
				
				\node (e) at (1,0) {$ \mr{Def}_g  $};
				
				%	\node (f) at (5,0) {$  \mr{Th}(f_0,p_0) $};

				\node (g) at (0.5,0.5) {$ \circlearrowleft $};
				%\node (h) at (2.25,0.75) {$ \ulcorner $};

				\path[font=\scriptsize,>= angle 90]

				(a) edge [->] node [above] {$ \hat{\omega} $} (b)
				%	(c) edge [->] node [above] {$ \Pi_0(\iota,p\circ \pi) $} (b)
				(d) edge [->] node [below] {$ \hat{g} $} (e)
				%	(f) edge [->] node [below] {$ \Pi_0(f,p) $} (e)
				(a) edge [->] node [left] {$\xi$} (d)
				
				(b) edge [->] node [right] {$ \delta $} (e);
				
				%	(c) edge [->] node [right] {\rotatebox{-90}{$ \sim  $}} (f);
				
			\end{tikzpicture}
		\end{center}
	\end{minipage}\\
	\noindent we get that the middle term in \eqref{S5:_HP_higher_eq2} becomes:
	\[ r_S\epfs \hat{p}\epfs \hat{g}_!r_{\mc X}^*x_! \stackrel{Ex^*_!}{\simeq} r_S\epfs \hat{p}\epfs \hat{g}_!\xi_!r_{X}^*\simeq  r_S\epfs \hat{p}\epfs \delta_! \hat{\omega}_! r_{X}^* \]
	\noindent Notice that $\delta: \mr{Def}_{\omega}\rightarrow \mr{Def}_g$ is a $(n-1)$-representable, smooth surjection (cf. \cite{Aranha-Pstragowski}). Hence, by induction hypothesis, we can use purity and get:
	\[ \delta_!\Sigma^{\mb L_{\delta}}\simeq \delta\epfs \]
	But by \Cref{Relative_ctg_Def}, we actually have that $\delta_!\simeq\delta\epfs $. Hence the middle term in \eqref{S5:_HP_higher_eq2} becomes:
	\[ r_S\epfs \hat{p}\epfs \hat{g}_!r_{\mc X}^*x_! \simeq  r_S\epfs \hat{p}\epfs \delta\epfs \hat{\omega}_! r_{X}^*=: r_S\epfs \mr{Th}(\hat{\omega}, \delta\circ \hat{p}) r_X^* \]
	\noindent We will denote by $ \Phi_{\delta} $ the composition of equivalences we just showed:
	\[ \Phi_{\delta}: r_S\epfs \hat{p}\epfs \hat{g}_!r_{\mc X}^*x_! \stackrel{Ex^*_!}{\simeq} r_S\epfs \hat{p}\epfs \hat{g}_!\xi_!r_{X}^*\simeq  r_S\epfs \hat{p}\epfs \delta_! \hat{\omega}_! r_{X}^* \stackrel{\sim}{\underset{\varphi_{\delta}}{\longrightarrow}} r_S\epfs \hat{p}\epfs \delta\epfs \hat{\omega}_! r_{X}^*=: r_S\epfs \mr{Th}(\hat{\omega}, \delta\circ \hat{p}) r_X^* \]
	With similar arguments to the ones used in the proof of \Cref{Key_Lemma:_schematic_HP+}, one can show that \eqref{S5:_HP_higher_eq2} and $ \Phi_x, \Phi_{\delta}$ fit into a commutative diagram of the following form:
		\begin{center}
		\begin{tikzpicture}[baseline={(0,1)}, scale=2]
			
			\node (a) at (0,1) {$  \mr{Th}(g,p)x_!$};
			
			\node (b) at (2.5,1) {$ r_{ S}\epfs\mr{Th}(\hat{g},\hat{p})r_{X}^*x_! $};
			
			\node (c) at (5,1) {$  \mr{Th}(g_0,p_0)x_!\simeq q\epfs \Sigma^{\mc N_{g}}x_! $};
			
			\node (d) at (0,0) {$  \mr{Th}(\omega,p) $};
			
			\node (e) at (2.5,0) {$ r_{ S}\epfs\mr{Th}(\hat{\omega},\delta\circ \hat{p})r_{X}^*  $};
			
			\node (f) at (5,0) {$  \mr{Th}(\omega_0,\gamma\circ p_0)\simeq q\epfs x\epfs \Sigma^{\mc N_{\omega}}  $};

			%\node (g) at (1.25,0.75) {$ \ulcorner $};
			%\node (h) at (2.25,0.75) {$ \ulcorner $};

			\path[font=\scriptsize,>= angle 90]

			(a) edge [->] node [above] {$ \Pi_1(g,p)x_! $} (b)
			(c) edge [->] node [above] {$ \Pi_0(g,p)x_! $} (b)
			(d) edge [->] node [below] {$ \Pi_1(\omega,p) $} (e)
			(f) edge [->] node [below] {$ \Pi_0(\omega,p) $} (e)
			(a) edge [double equal sign distance] node [left] {}  (d)
			
			(b) edge [->] node [left] {\rotatebox{-90}{$ \sim  $}} node [right] {$\Phi_{\delta}$} (e)
			
			(c) edge [->] node [left] {\rotatebox{-90}{$ \sim  $}} node [right] {$\Phi_x$}(f);
			
		\end{tikzpicture}
	\end{center}
	\noindent But this implies that we can identify the natural transformations in \eqref{S5:_HP_higher_eq2} with $ \Pi_1(\omega,p) $ and $\Pi_1(\omega,p)$ that we know are equivalences by \Cref{Key_Lemma:_schematic_HP+}, and hence we are done.
	
\end{proof}

\begin{lemma}\label{Relative_ctg_Def}
	Let $g: \mc X\rightarrow Y$ be a map between $S$-smooth $(n-1)$-algebraic stacks, such that $S,Y$ are represented by schemes, and let $x: X \rightarrow \mc X$ be a NL-atlas with $X\in \op{Sch}$. Let $\delta: \mr{Def}_{\omega}\rightarrow \mr{Def}_g$ be the map induced by the following commutative diagram:
		\begin{center}
		\begin{tikzpicture}[baseline={(0,1)}, scale=1.25]
			
			\node (a) at (0,1) {$ X $};
			
			\node (b) at (1,1) {$ Y  $};

			\node (d) at (0,0) {$  \mc X $};
			
			\node (e) at (1,0) {$ Y  $};

			\node (g) at (0.25,0.75) {$  $};

			\path[font=\scriptsize,>= angle 90]

			(a) edge [->] node [above] {$ \omega $} (b)
			(d) edge [->] node [below] {$g $} (e)
			(b) edge [double equal sign distance] node [left] {$$} (e)
			
			(a) edge [->] node [right] {$x$} (d);

		\end{tikzpicture}
	\end{center}
	\noindent Then we have a natural identification:
	\[ \Sigma^{\mb L_{\delta}}\simeq Id_{\SH(\mr{Def}_{\omega})} \]
\end{lemma}
\begin{proof}
	Let us recall that by construction we have the following cartesian diagram:
	
		\begin{center}
		\begin{tikzpicture}[baseline={(0,0.5)}, scale=2]
			
			\node (a) at (1,1) {$  N_{\omega} $};
			\node (b) at (2,1) {$ \mr{Def}_{\omega} $};
			\node (c) at (3,1) {$ \mb G_{m,Y} $};

			\node (d) at (1,0) {$ N_g $};
			\node (e) at (2,0) {$ \mr{Def}_g $};
			\node (f) at (3,0) {$ \mb G_{m,Y} $};

			\node (g) at (1.25,0.75) {$ \ulcorner $};
			\node (h) at (2.75,0.75) {$ \urcorner $};

			\path[font=\scriptsize,>= angle 90]

			(a) edge [closed] node [above] {$ \nu_{\omega} $} (b)
			(c) edge [open'] node [above] {$ j_{\omega}  $} (b)
			(d) edge [closed] node [below] {$ \nu_g $} (e)
			(f) edge [open'] node [below] {$  j_g $} (e)
			
			(a) edge [->] node [left] {$  \gamma $} (d)
			(b) edge [->] node [left] {$ \delta $} (e)
			(c) edge [double equal sign distance] node [left] {$  $} (f);
			
		\end{tikzpicture}
	\end{center}
	
	\noindent Since $\delta: \mr{Def}_{\omega}\rightarrow \mr{Def}_{g}$ is a smooth surjection, in particular we get that the squares are actually homotopy cartesian. This implies that:
	\[ \nu_{\omega}^*\mb L_{\delta}\simeq \mb L_{\gamma} \quad \quad j_{\omega}^*\mb L_{\delta}\simeq 0 \]
	
	Denote by $\psi_{\omega}: N_{\omega}\rightarrow X$ and $\psi_{g}: N_g \rightarrow \mc X$ the projection maps. Then denoting $\mc N_{\bullet}:=\mb L_{\bullet}[-1]$ and looking at the standard fiber sequences of cotangent complexes, we get:
	\begin{equation}\label{HP_twist_lemma_eq}
		\begin{split}
			\psi_{\omega}^*\mb L_{\omega}\rightarrow \mb L_{\bigslant{N_{\omega}}{Y}} \rightarrow \psi_{\omega}^*\mc N_{\omega}\\
			\psi_{g}^*\mb L_{g}\rightarrow \mb L_{\bigslant{N_{g}}{Y}} \rightarrow \psi_{g}^*\mc N_{g}
		\end{split}
	\end{equation}

	The commutativity of the following diagram:
		\begin{center}
		\begin{tikzpicture}[baseline={(0,1)}, scale=1.25]
			
			\node (a) at (0,1) {$ N_{\omega} $};
			
			\node (b) at (1,1) {$ X  $};
			
			%	\node (c) at (5,1) {$  \mr{Th}(\iota_0,p_0\circ \pi_0)\Sigma^{-\mb L_{\pi}} $};
			
			\node (d) at (0,0) {$ N_g $};
			
			\node (e) at (1,0) {$ \mc X  $};
			
			%	\node (f) at (5,0) {$  \mr{Th}(f_0,p_0) $};

			\node (g) at (0.5,0.5) {$ \circlearrowleft  $};
			%\node (h) at (2.25,0.75) {$ \ulcorner $};

			\path[font=\scriptsize,>= angle 90]

			(a) edge [->] node [above] {$ \psi_{\omega} $} (b)
			%	(c) edge [->] node [above] {$ \Pi_0(\iota,p\circ \pi) $} (b)
			(d) edge [->] node [below] {$ \psi_{g} $} (e)
			%	(f) edge [->] node [below] {$ \Pi_0(f,p) $} (e)
			(a) edge [->] node [left] {$ \gamma $} (d)
			
			(b) edge [->] node [right] {$x$} (e);
			
			%	(c) edge [->] node [right] {\rotatebox{-90}{$ \sim  $}} (f);
			
		\end{tikzpicture}
	\end{center}
	\noindent tells us that $\gamma^*\psi_g^*\simeq \psi_{\omega}^*x^* $. Therefore the fiber sequences in \eqref{HP_twist_lemma_eq} fit into a diagram of fiber sequences of the form:
	
		\begin{center}
		\begin{tikzpicture}[baseline={(0,1.5)}, scale=2]

			\node (a2) at (0,1.5) {$ \psi_{\omega}^*x^*\mb L_g $};
			\node (a1) at (0,0.75) {$ \psi_{\omega}^*\mb L_{\omega}  $};
			\node (a0) at (0,0) {$ \psi_{\omega}^* \mb L_x $};

			\node (b2) at (1,1.5) {$ \gamma^* \mb L_{\bigslant{N_g}{Y}} $};
			\node (b1) at (1,0.75) {$ \mb L_{\bigslant{N_{\omega}}{Y}} $};
			\node (b0) at (1,0) {$ \mb L_{\gamma}  $};

			\node (c2) at (2,1.5) {$ \psi_{\omega}^*x^*\mc N_g $};
			\node (c1) at (2,0.75) {$ \psi_{\omega}^* \mc N_{\omega}$};
			\node (c0) at (2,0) {$ \psi_{\omega}^*\mc N_x $};

			\node (t1) at (0.25,1.75) {$  $};
			\node (t3) at (0.25,0.75) {$  $};
			\node (t2) at (1.25,1.75) {$  $};
			\node (t4) at (1.25,0.75) {$  $};

			\node (T1) at (0.5,1.5) {$  $};
			\node (T3) at (0.5,0.5) {$  $};
			\node (T2) at (1.5,1.5) {$  $};
			\node (T4) at (1.5,0.5) {$  $};

			%\node (g) at (0.5,0.5) {$ \Delta $};

			\path[font=\scriptsize,>= angle 90]

			%horizontal
			
			(a2) edge [->] node [above ] {$  $} (b2)
			(a1) edge [->] node [above ] {$  $} (b1)
			(a0) edge [->] node [below ] {$  $} (b0)
			
			(b2) edge [->] node [above ] {$  $} (c2)
			(b1) edge [->] node [above ] {$  $} (c1)
			(b0) edge [->] node [below ] {$ $} (c0)
			
			%vertical
			
			(a2) edge [->] node [left ] {$  $} (a1)
			(a1) edge [->] node [left ] {$  $} (a0)

			(b2) edge [->] node [left ] {$  $} (b1)
			(b1) edge [->] node [left ] {$  $} (b0)
			
			(c2) edge [->] node [right ] {$  $} (c1)
			(c1) edge [->] node [right ] {$  $} (c0);
		\end{tikzpicture}
	\end{center}
	\noindent from which, looking at the last row, we deduce that:
	\[ \Sigma^{\nu_{\omega}^*\mb L_{\delta}}\simeq \Sigma^{\mb L_{\gamma}}\simeq Id \]
	But by \cite[Theorem 4.5.1(6)]{Chowdhury}, we have that:
	\[ \nu_{\omega}\epf \nu_{\omega}^! \rightarrow Id \rightarrow j_{\omega}{}_*j_{\omega}^* \]
	\noindent is a fiber sequence. Applying $\Sigma^{\mb L_{\delta}}$ to this we get another fiber sequence:
	\[ \nu_{\omega}\epf \nu_{\omega}^!\Sigma^{\mb L_{\delta}} \rightarrow \Sigma^{\mb L_{\delta}} \rightarrow j_{\omega}{}_*j_{\omega}^*\Sigma^{\mb L_{\delta}} \]
	\noindent But the LHS and the RHS are respectively equivalent to:
	\[ \nu_{\omega}\epf \nu_{\omega}^!\Sigma^{\mb L_{\delta}}\simeq \nu_{\omega}\epf \Sigma^{\nu_{\omega}^*\mb L_{\delta}}\nu_{\omega}^!\simeq \nu_{\omega}\epf \Sigma^{\mb L_{\gamma}}\nu_{\omega}^! \simeq \nu_{\omega}\epf \nu_{\omega}^!  \]
	\[ j_{\omega}{}_*j_{\omega}^*\Sigma^{\mb L_{\delta}}\simeq j_{\omega}{}_*\Sigma^{j_{\omega}^*\mb L_{\delta}}j_{\omega}^*\simeq  j_{\omega}{}_*j_{\omega}^*  \]
	\noindent since $ \Sigma^{\mb L_{\gamma}}\simeq Id $ and $ \Sigma^{j_{\omega}^*\mb L_{\delta}}\simeq Id $. Again by \cite[Theorem 4.5.1(6)]{Chowdhury}, this implies that:
	\begin{align*}
		\Sigma^{\mb L_{\delta}}&\simeq \op{Fib}\left(  j_{\omega}{}_*\Sigma^{j_{\omega}^*\mb L_{\delta}}j_{\omega}^* \rightarrow  \nu_{\omega}\epf \Sigma^{\mb L_{\gamma}}\nu_{\omega}^![+1]  \right)\simeq \\
		&\simeq \op{Fib}\left( j_{\omega}{}_*j_{\omega}^*\rightarrow \nu_{\omega}\epf \nu_{\omega}^![+1] \right)\simeq \\
		&\simeq Id
	\end{align*}
	\noindent exactly as we wanted.
\end{proof}

\subsection{Universality after Drew-Gallauer}

Let us notice that the notion of (complete) coefficient system of \cite[Definition 7.5,7.7]{Drew-Gallauer} makes completely sense replacing schemes with derived schemes. Moreover, by definition a coefficient system on derived schemes will satisfy the localisation property and this will force nil-invariance by \cite[Corollary 5.9]{Khan-Ravi_Generalised_Coh_Stacks}. So a coefficient system from derived schemes:
\[ D^*: \op{dSch}^{op}\longrightarrow \op{CAlg}(\op{Cat}_{\infty}^{st})  \]
\noindent uniquely corresponds to the coefficient system induced by the truncation functor $(-)_{cl}: \op{dSch}\rightarrow \op{Sch}$:
\[ D^*: \op{Sch}^{op}\longrightarrow \op{CAlg}(\op{Cat}_{\infty}^{st})   \]

\begin{proposition}\label{S5:_Universality_SH}
	Let:
	\[ D^*: \left(\oocatname{dASt}^{\leq n}\right)^{op}\longrightarrow \op{CAlg}(\op{Cat}_{\infty}^{st}) \]
	\noindent be a functor that satisfies NL-descent and such that its restriction to derived schemes is a complete coefficient system in the sense of \cite{Drew-Gallauer}. Then there exists a unique (up to contractible choice) map:
	\[ R_D: \SH^* \longrightarrow D^* \in \op{Fun}\left(\left(\oocatname{dASt}^{\leq n}\right)^{op}, \op{CAlg}(\op{Cat}_{\infty}^{st})\right)^{\Delta^1} \]
	\noindent that commutes with arbitrary pullbacks, with \#-functoriality, tensors and inner-hom's, and with exceptional functors arising from smooth maps.
\end{proposition}
\begin{proof}
	It is a very easy consequence of \cite[Theorem 7.14]{Drew-Gallauer}. Indeed, by \textit{loc. cit.}, nil-invariance and by hypothesis we get a unique (up to contractible choice) functor:
	 \[ R_D: \SH^* \longrightarrow D^* \in \op{Fun}\left(\op{dSch}^{op}, \op{CAlg}(\op{Cat}_{\infty}^{st})\right)^{\Delta^1} \]
	 By NL-descent, we know that both $SH^*$ and $D^*$ are Kan extended from derived schemes to $\oocatname{dASt}^{\leq n}$, hence we get a unique (up to contractible choice) functor:
	 \[ R_D: \SH^* \longrightarrow D^* \in \op{Fun}\left(\left(\oocatname{dASt}^{\leq n}\right)^{op}, \op{CAlg}(\op{Cat}_{\infty}^{st})\right)^{\Delta^1} \]
	 The commutativity of $R_D$ with the various functors follows from nil-invariance together with the schematic case proved in \cite{Ayoub-Betti}.
\end{proof}

\begin{remark}
	Let us notice that our proof of purity for algebraic stacks did not really rely on the particular six functor formalism given by $\SH$. Indeed, any complete coefficient system can be extended via NL-descent to algebraic stacks and all the formal property used in our proofs will be satisfied also by the extension of the complete coefficient system. In particular, one will get purity for smooth maps of algebraic stacks in any complete coefficient system extended via NL-descent. Another way to see this is to use \Cref{S5:_Universality_SH}, since for smooth maps the functor $R_D$ behaves well with the six functors.
\end{remark}

\subsection{Grothendieck-Verdier Duality}

We will conclude with an easy consequence of the existence of exceptional functors for non representable maps: we have now a good set up to study a motivic version of Grothendieck-Verdier duality for algebraic stacks.

\begin{definition}
	Let $\mc X$ be a $n$-algebraic derived stack. We will say that $\mb E \in \SH(\X)$ is \textit{compact} if for all NL-covers $x: X \rightarrow \mc X$ we have that $x^!\mb E \in \SH(X)$ is compact. We will denote the subcategory of compact objects as $\SH_c(\X)$.
\end{definition}
\begin{remark}
	One can show, using Nis descent via distinguished squares on the level of schemes, that $\mb E \in \SH(\X)$ is compact if and only if there exists a NL-atlas such that $\mb E \in \SH_c(X)$.
\end{remark}

\begin{definition}
	Let $\X \in \oocatname{dASt}^{\leq n}_{\bigslant{}{S}}$ and denote by $\pi_{\X}: \X \rightarrow S$ the structure map. 
	\begin{enumerate}
		\item The \textit{dualising sheaf} of $\X$ is defined as:
		\[ \omega_{\bigslant{\X}{S}}:=\pi_{\X}^!\mbbm 1_S \]
		\item Define the \textit{dualising functor} $\mb D_{\bigslant{\X}{S}}$ as:
		\[ \mb D_{\bigslant{\X}{S}}(-):=\iMap(-, \omega_{\bigslant{\X}{S}}) \]
	\end{enumerate}

\end{definition}

\begin{proposition}\label{S5:_comm_Duality}
	Let $f: \mc X \rightarrow \Y \in \oocatname{dASt}^{\leq n}_{\bigslant{}{S}}$ a map between $n$-algebraic derived stacks. Then we have that:
	\[ \mb D_{\bigslant{\X}{S}} \circ f^* \simeq f^! \circ \mb D_{\bigslant{\Y}{S}} \]
\end{proposition}
\begin{proof}
	This is an easy consequence of the projection formula $f_!(-\otimes f^*-)\simeq f_!-\otimes - $. Indeed, the projection formula can be equivalently rewritten as:
	\[ \iMap(f^*-,f^!-)\simeq f^!\iMap(-,-) \]
	But denoting by $\pi_{\bullet}$ the structure maps of $\X$ and $\Y$, this implies that:
	\begin{align*}
		\mb D_{\bigslant{\X}{S}}\circ f^*&=\iMap(f^*-, \pi_{\X}^!\mbbm 1_S)=\\
		&=\iMap(f^*-,f^!\pi_{\Y}^!\mbbm 1_S )\simeq\\
		&\simeq f^!\iMap(-, \omega_{\bigslant{Y}{S}})=\\
		&=f^!\circ \mb D_{\bigslant{\Y}{S}}
	\end{align*}
\noindent as we wanted.
\end{proof}

\begin{proposition}
	Let $\sk$ be a field with exponential characteristic  $\op{expchar}(\sk)=e$ and let $\X \in \oocatname{dASt}^{\leq n}_{\bigslant{}{S}}$, for $S$ a lft $\sk$-scheme. Then the natural map:
	\[ Id \longrightarrow \mb D_{\bigslant{\X}{S}}\circ \mb D_{\bigslant{\X}{S}} \]
	\noindent is an equivalence in $\SH_c(\X)\left[\dfrac{1}{e}\right]$.
\end{proposition}
\begin{proof}
	Let $x: X\rightarrow \mc X$ be a NL-atlas. By conservativity of $x^!$, we have that:
	\[ Id \longrightarrow \mb D_{\bigslant{\X}{S}}\circ \mb D_{\bigslant{\X}{S}} \]
	\noindent is an equivalence if and only if:
	\[ x^! \longrightarrow x^!\mb D_{\bigslant{\X}{S}}\circ \mb D_{\bigslant{\X}{S}} \]
	\noindent is an equivalence. But for any $\mb E \in \SH_c(\X)\left[ \dfrac{1}{e} \right]$, we have that:
	\begin{equation}\label{eq:_GV_Duality}
		x^!\mb E \longrightarrow x^!\mb D_{\bigslant{\X}{S}}\circ \mb D_{\bigslant{\X}{S}} (\mb E)\simeq \mb D_{\bigslant{X}{S}}\circ \mb D_{\bigslant{X}{S}} (x^!\mb E )
	\end{equation}
	\noindent by \Cref{S5:_comm_Duality}. But by definition $x^!\mb E$ is constructible in $\SH(X)$ and hence \eqref{eq:_GV_Duality} is an equivalence by \cite[Theorem 3.1.1]{Elmanto-Khan}, and we are done.
\end{proof}

\newpage
\begin{appendices}\label{Appendices}

    \section{Six-functor formalisms.}
    In this section, we recall the notion and prove the relevant lemmas, propositions and  theorems that we need in order to extend the six-functor formalisms from schemes to algebraic stacks.  The main references for this section are \cite[A.5]{mann2022padic} and \cite[Section 6.1]{liu2017enhanced}.
    \subsection{The $\infty$-category of correspondences.}
    \begin{definition}\cite[Definition A.5.1]{mann2022padic}
        A \textit{geometric setup} is a pair $(\Ca,\E)$ where $\Ca$ is an $\infty$-category and $E$ is a homotopy class of edges in $\Ca$ satisfying
        \begin{enumerate}
            \item $E$ contains all isomorphism and is stable under compositions,

           \item Pullbacks of $E$ exist and remain in $E$.
        \end{enumerate}
    \end{definition}
    \begin{definition}\cite[Definition A.5.2]{mann2022padic}
        \begin{enumerate}
            \item Let $C(\Delta^n) \subset \Delta^n \times (\Delta^n)^{op}$ be the full-subcategory spanned by $([i],[j])$ with $i \le j$.
            \item An edge in $C(\Delta^n)$ is \textit{vertical} (resp. horizontal) if the projection to the second (first) factor is degenerate.
            \item A square in $C(\Delta^n)$ is \textit{exact} if it is both a pullback and a pushout square.
            \item For any simplicial set $K$, we define 
            \begin{equation*}
                C(K): = \op{colim}_{(n,\sigma) \in \Delta_{/K}}C(\Delta^n).
            \end{equation*}
            Thus $C$ can be visualized as an endofunctor on the category of simplicial sets. It admits a right adjoint 
            \begin{equation*}
                B: \sset \to \sset
            \end{equation*}
            defined by 
            \begin{equation*}
                K \mapsto B(K) =\{B(K)_n:= \op{Hom}_{\sset}(C(\Delta^n),K) \}.
            \end{equation*}
            \item Given a geometric setup $(\Ca,\E)$, define 
            \begin{equation*}
                \op{Corr}(\Ca)_{E,\op{all}} \subset B(\Ca)
            \end{equation*}
            to be the sub-simplicial set whose $n$ simplices are maps $C(\Delta^n) \to \Ca$ which sends 
            \begin{enumerate}
                \item vertical edges to $E$,
                \item exact squares to pullback squares.
            \end{enumerate}
        \end{enumerate}
    \end{definition}
    \begin{remark}
     For every $n \ge 1$, we have maps \begin{equation}
       \gamma_n : (\Delta^n)^{op} \to C(\Delta^n)  \quad ,\quad \gamma_n' : \Delta^n \to C(\Delta^n)
     \end{equation}
     given by the inclusion of the top row and the right most columns. These map induces maps :
     \begin{equation}\label{Cboundarymaps}
         \gamma_n : (\Lambda^n_0)^{op} \to C(\partial\Delta^n) \quad;\quad \gamma_n' : \Lambda^n_0 \to C(\Delta^n)
     \end{equation}
    \end{remark}
    \begin{remark}
Let $(\Ca,\E)$ be a geometric setup. Then the lower simplices of $\op{Corr}(\Ca)_{\E,\op{all}}$ look like as follows:
     \begin{itemize}
         \item The $0$-simplices are objects of $\Ca$.
         \item A $1$-simplex i.e an edge between $X_0$ and $X_1$ where $X_0,X_1 \in \Ca$ is a diagram of the form :
         \begin{equation*}
             \begin{tikzcd}
                 X_0 & X_{01}\arrow[l] \arrow[d,"f"] \\
                 {} & X_1
             \end{tikzcd}
         \end{equation*}
         where $f \in E$.
         \item A $2$-simplex in $\op{Corr}(\Ca)_{\E,\op{all}}$ looks like as follows:
         \begin{equation*}
             \begin{tikzcd}
                 X_0 & X_{01} \arrow[l] \arrow[d,"f"]  \arrow[dr, phantom, "\square"]  & X_{02} \arrow[l] \arrow[d,"g"] \\
                 {} & X_{11}  & X_{12} \arrow[d,"h"] \arrow[l] \\
                 {} & {} & X_{22}
             \end{tikzcd}
         \end{equation*}
         where $f,g,h\in E$ and $\square$ is a pullback square.
     \end{itemize}
     Notice that edges in $Corr(\Ca)_{\E,\op{all}}$ can be composed. In particular if we have two $1$ -simplices 
     \begin{equation*}
         \begin{tikzcd}
             X & X' \arrow[l] \arrow[d,"f"] \\
             {} & Y
         \end{tikzcd}
         ~~~~~~\text{and}~~~~~~
         \begin{tikzcd}
             Y & Y'\arrow[l] \arrow[d,"h"] \\
             {} & Z,
         \end{tikzcd}
     \end{equation*}
     we can compose these two $1$-simplices to form a $1$-simplex which is the outer roof of the following $2$-simplex
     \begin{equation*}
         \begin{tikzcd}
             X & X'\arrow[l] \arrow[d,"f"] & X' \times_{Y} Y' \arrow[l] \arrow[d,"g"] \\
             {} & Y & Y'\arrow[l] \arrow[d,"h"] \\
             {} & {} & Z.
         \end{tikzcd}
     \end{equation*}
     This is possible because we can take pullback of edges in $\E$ as $(\Ca,\E)$ is a geometric setup. 
    \end{remark}
    The above remark shows  that $\op{Corr}(\Ca)_{\E,\op{all}}$ admits lift along inclusions $\Lambda^2_1 \to \Delta^2$. This leads us to the following claim that $\op{Corr}(\Ca)_{\E,\op{all}}$ is an $\infty$-category:
    \begin{proposition}
        Let $(\Ca,\E)$ be a geometric setup. Then the simplicial set $\op{Corr}(\Ca)_{\E,\op{all}}$ is an $\infty$-category.
    \end{proposition}
    \begin{proof}
        By \cite[Corollary 2.3.2.2]{HTT}, it is enough to show that the map 
        \begin{equation*}
            \op{Fun}(\Delta^2,\op{Corr}(\Ca)_{\E,\op{all}}) \to \op{Fun} (\Lambda^2_1,\op{Corr}(\Ca)_{\E,\op{all}})
        \end{equation*}
        is a trivial fibration of simplicial sets.  Let $\K': = \op{Fun}(\Delta^2,\op{Corr}(\Ca)_{\E,\op{all}})$ and $\K:=\op{Fun} (\Lambda^2_1,\op{Corr}(\Ca)_{\E,\op{all}})$. We have obvious inclusions $\K' \subset \op{Fun}(C(\Delta^2),\Ca)$ and $\K \subset \op{Fun}(C(\Lambda^2_1),\Ca)$.  By definition of the functor $C$, we see that 
        \begin{equation*}
            C(\Delta^2) =C(\Lambda^2_1)^{\triangleleft}.
        \end{equation*}
        Now we notice the following observations :
        \begin{itemize}
            \item  The simplicial set $\K$ is the full subcategory of $\op{Fun}(C(\Lambda^2_1),\Ca)$ spanned by functors $F_0: C(\Lambda^2_1) \to \Ca$ which admit a limit i.e. $F_0$ can be extended to $C(\Delta^2)=C(\Lambda^2_1)^{\triangleleft} \to \Ca$. This is because as the vertical edges are spanned in $E$ and $(\Ca,\E)$ is a geometric setup.
            \item The simplicial set $\K'$ is the full subcategory of $\op{Fun}(C(\Delta^2),\Ca)$ spanned by functors $F: C(\Delta^2) \to \Ca$ which are right Kan extended from $F|_{C(\Lambda^2_1)}$. This is also true because $(\Ca,\E)$ is a geometric setup and exact squares get mapped to pullback squares in the definition of $\op{Corr}(\Ca)_{\E,\op{all}}$.
        \end{itemize}
   
    Thus one can apply the dual version of \cite[Proposition 4.3.2.15]{HTT} to $\K,\K'$ where $\Ca_0= C(\Lambda^2_1),\Ca = C(\Delta^2), \D = \Ca$ and $\D'=\Delta^0$. Thus we get that the map $\K' \to \K$ is a trivial fibration which completes our proof.
 
    \end{proof}

    Before upgrading the notion of correspondences in operadic version, we introduce some new notations on how the Correspondences behave with functor categories. This shall play a key role in extending six-functor formalisms from one geometric setup to another.
\begin{notation}
\begin{itemize}

    \item     Let $(\Ca,\E)$ be a geometric setup and $K$ be a simplicial set.  Then the functor category $(\op{Fun}(K,\Ca),\tilde{\E})$ is a geometric setup where $\tilde{\E}$ is set of edges $\Delta^1 \times K \to \Ca$ such that for all $k \in K$, the edge $\Delta^1 \times\{k\} \to \Ca$ is in $E$. For any set of edges $P$ stable under pullbacks and compositions in $\Ca$, let $\op{Fun}^P(K,\Ca)$ is full subcategory spanned by functors $f: K \to \Ca$ such that $f: K \to \Ca_P \to \Ca$.

\item Let $\op{Corr}^{\op{all-cart}}(\op{Fun}^{\E}(K,\Ca))_{\tilde{\E},\op{all}}$ be the full subcategory spanned by edges 
\begin{equation}\label{Corrfunctorialmap}
    \begin{tikzcd}
        f &  \arrow[l,"\sigma"] f' \arrow[d,"\tau"] \\
        {} & g
    \end{tikzcd}
\end{equation}
where $\tau \in \tilde{E}$ and for all $\gamma: \Delta^1 \to K$, the square $\sigma_{\gamma} : \Delta^1 \times \Delta^1 \to \Ca$ is cartesian. 
\item Using the notations from above, we have an canonical functor \[ \alpha_{K,\Ca,\E} : \op{Corr}^{\op{all-cart}}(\op{Fun}^{\E}(K,\Ca))_{\tilde{\E},\op{all}} \to \op{Fun}(K,\op{Corr}(\Ca)_{\E,\op{all}}) \]
As $C$ is a colimit preserving functor, it is enough to define the map for $K=\Delta^m$ where $m \ge 0$.
On the level of $n$-simplices, the map is defined as follows:

Let $\sigma_n$ be an $n$-simplex of $\op{Corr}^{\op{all-cart}}(\op{Fun}^{\E}(\Delta^m,\Ca))_{\tilde{\E},\op{all}}$, which is a map $C(\Delta^n) \times \Delta^m \to \Ca$. 
To obtain an $n$-simplex of RHS, we need to define a map $C(\Delta^n) \times C(\Delta^m) \to \Ca$. The map is defined by the following composition 
\begin{equation*}
    C(\Delta^n) \times C(\Delta^m) \to C(\Delta^n) \times \Delta^m \xrightarrow{\sigma_n}\Ca
\end{equation*}
the conditions of the the simplex $\sigma_n$ ensures that the morphism $C(\Delta^n) \times C(\Delta^m) \to \Ca$ can be upgraded to the a morphism $\Delta^n \times \Delta^m \to \op{Corr}(\Ca)$.

\item There is dual map 
   \begin{equation}\label{Corrfunctorialmapdual}
       \alpha'_{K,\Ca,\E} : \op{Corr}^{\op{\E-cart}}(\op{Fun}(K,\Ca))_{\tilde{E},\op{all}} \to \op{Fun}(K^{\op{op}},\op{Corr}(\Ca)_{\E,\op{all}}) .
   \end{equation}
\end{itemize}
\end{notation}
\begin{remark}
    The $\infty$-category $\op{Corr}(\Ca)_{\E,\op{all}}$ admits the following canonical maps :
    \begin{equation*}
        \Ca_{\E} \to \op{Corr}(\Ca)_{\E,\op{all}} \quad,\quad \Ca^{op} \to \op{Corr}(\Ca)_{\E,\op{all}}
    \end{equation*}
\end{remark}

\begin{proposition}\label{preservingcoproductsCor}
    Let $(\Ca,\E)$ be a geometric setup where $\Ca$ admits finite products. Then the map :
    \begin{equation*}
        \Ca^{op} \to \CrrCpEal
    \end{equation*}
    preserves finite coproducts.
\end{proposition}
\begin{proof}
    Let $X,Y \in \CrrCpEal$, we claim that the object $X \times Y$ with maps 
    \begin{equation*}
        \begin{tikzcd}
            X & \arrow[l] X \times Y \arrow[d,"\op{id}"] \\ {} & X \times Y
        \end{tikzcd}
        \quad ; \quad
        \begin{tikzcd}
            Y & \arrow[l] X \times Y \arrow[d,"\op{id}"] \\ {} & X \times Y
        \end{tikzcd}
    \end{equation*}
    makes $X \times Y$ as a coproduct of $X$ and $Y$.

    Let $F: \Delta^0 \coprod \Delta^0 \to \CrrCpEal$ be the diagram given by $X$ and $Y$. Let $F': (\Delta^0 \coprod \Delta^0)^{\triangleright} = \Lambda^2_0 \to \CrrCpEal$ be the diagram given by the two morphisms above. In order to show this is a colimit diagram, we need to show the following :
    \begin{itemize}
        \item         For~$n \ge 2$ and $H: (\Delta^0 \coprod \Delta^0) \boldsymbol{*}\partial\Delta^{n-1} \to \CrrCpEal$ be a diagram such that $H|_{[0]}=F'$. Then  $H~\text{extends to}~H':(\Delta^0 \coprod\Delta^0) \boldsymbol{*} \Delta^{n-1} \to \CrrCpEal$.
    \end{itemize}
The morphism $H$ is a map $(H_1,H_2):C(\Lambda^n_0) \coprod_{C(\partial\Delta^n)} C(\Lambda^n_0) \to \Ca$. Let $\{X_{i1}\}_{i=1}^m, \{X_{2i}\}_{i=1}^n$ be images of $H_1$ and $H_2$ along the first row of $C(\Lambda^n_0)$. Taking products $\{X_{i1} \times X_{i2}\}_{i=1}^n$, the morphism $H_1,H_2$ amalgamate to a morphism :
\begin{equation*}
    H_3: C(\Lambda^n_0)\coprod_K C(\Lambda^n_0) \to \Ca
\end{equation*}
where $K: = C(\Lambda^n_0)/\{(0,0)\}$. Here $H_3(0,i)$ for $i=1,2.\cdots,n$ is $X_{i1} \times X_{i2}$. For $n=2$ a pictorial description of $H_3$ is the following diagram :
\begin{equation}\label{nequals2diagram}
		\begin{tikzpicture}[baseline={(0,3)}, scale=2]
			
			\node (a1) at (0,2) {$ Y $};
			\node (a2) at (0,1.5) {$ X $};
			
			\node (b1) at (1,1.5) {$ X\times Y $};
			\node (b2) at (1,1) {$ X\times Y $};
			
			\node (c1) at (2,1.5) {$ X_{11}\times X_{22} $};
			\node (c2) at (2,1) {$  $};
			\node (c3) at (2,0.5) {$ Z $};
			
			\path[font=\scriptsize,>= angle 90]
			
			(b1) edge [->] node [above ] {$  $} (a1)
			(b1) edge [->] node [above ] {$  $} (a2)
			(b1) edge [double equal sign distance] node [above ] {$  $} (b2)
			
			(c1) edge [bend right=30,-, line width=1mm, white] node [] {} (a2)
			
			(c1) edge [bend right=30,->] node [above ] {$  $} (a1)
			(c1) edge [bend right=30,->] node [above ] {$  $} (a2)
			(c1) edge [bend left=30,->] node [above ] {$  $} (c3);
		\end{tikzpicture}
\end{equation}

%\begin{equation}
%    \begin{tikzcd}
%        Y  & {} & {} \\
%        X   &  X \times Y \arrow[l] \arrow[ul] \arrow[d,"\op{id}"] & X_{11} \times X_{12} \arrow[ll, bend right=70] \arrow[ull, bend right = 80]   \arrow[dd, bend left =40] \\
%        {} & X \times Y  & {} \\
%        {} & {} & Z
%        \end{tikzcd}
%\end{equation}
Restricting the morphism to the first row we get a map $h_3 : (\Lambda^n_0)^{op} \coprod_{(\partial\Delta^n)^{op}} (\Lambda^n_0)^{op}  \cong (\Delta^0 \coprod \Delta^0)\boldsymbol{*}\partial \Delta^n \to \Ca$ where restricting to $0$ vertex is the diagram $ X \leftarrow X \times Y \rightarrow Y$. As this is a limit diagram in $\Ca$, this  extends to a diagram \[h_3': (\Delta^0 \coprod \Delta^0)\boldsymbol{*} \Delta^n \to \Ca \] 

Amalgamating $h_3'$ with $H_3$, we see that $H_3$ extends to a morphism 
\begin{equation*}
    H_3' : C(\Lambda^n_0)' \coprod_{K}C(\Lambda^n_0)' \to \Ca
\end{equation*}
where the $'$ denotes including the first row of $C(\Delta^n)$ to the simplicial set. The rest of the proof follows from the following claim :
\begin{claim}
    The morphism $H_3'$ extends to a morphism \[H' : C(\Delta^n) \coprod_{C(\Delta^{n-1})} C(\Delta^n) \to \Ca.\]
\end{claim}
\begin{proof}[Proof of the claim]
We prove the claim in the following steps:
\begin{itemize}
    \item \textbf{Case n=2:} For $n=2$, we see that \cref{nequals2diagram}
 can be extended to the following diagram:
 
 \begin{equation*}
 	\begin{tikzpicture}[baseline={(0,3)}, scale=2]
 		
 		\node (a1) at (0,2) {$ Y $};
 		\node (a2) at (0,1.5) {$ X $};
 		
 		\node (b1) at (1,1.5) {$ X\times Y $};
 		\node (b2) at (1,1) {$ X\times Y $};
 		
 		\node (c1) at (2.25,1.5) {$ X_{11}\times X_{22} $};
 		\node (c2) at (2.25,1) {$ X_{11}\times X_{12} $};
 		\node (c3) at (2.25,0.25) {$ Z $};
 		
 		\path[font=\scriptsize,>= angle 90]
 		
 		(b1) edge [->] node [above ] {$  $} (a1)
 		(b1) edge [->] node [above ] {$  $} (a2)
 		(b1) edge [double equal sign distance] node [above ] {$  $} (b2)
 		(c1) edge [double equal sign distance] node [above ] {$  $} (c2)
 		(c2) edge [->] node [above ] {$  $} (b2)
 		(c1) edge [->] node [above ] {$  $} (b1)
 		
 		%ghost line for 3d-mensionality
 		(c1) edge [bend right=30,-, line width=1mm, white] node [] {} (a2)
 		
 		(c1) edge [bend right=30,->] node [above ] {$  $} (a1)
 		(c1) edge [bend right=30,->] node [above ] {$  $} (a2)
 		(c2) edge [->] node [above ] {$  $} (c3);
 	\end{tikzpicture}
 \end{equation*}
 
%\begin{equation}
%    \begin{tikzcd}
%        Y & {} & {}\\
%        X & X \times Y \arrow[ul] \arrow[l] \arrow[d,"\op{id}"] & X_{11} \times X_{12} \arrow[ull, bend right = 80]     \arrow[ll, bend right=70]  \arrow[d,"\op{id}"]  \arrow[l]\\
%        {} & X \times Y & X_{11} \times X_{12} \arrow[l] \arrow[d] \\
%        {} & {} & Z.
%    \end{tikzcd}
%\end{equation}
 which is the morphism $H'$ for $n=2$. For the rest of proof, hence we assume $n >2$.
 \item  We claim the  morphism $H_3'$ when restricted to $\widetilde{h_3}:=H_3':C(\partial\Delta^{n-1}) \subset K' \to \Ca $ can be extended to $C(\Delta^{n-1})$. Notice that the first row of $C(\partial\Delta^{n-1})$ is $\Lambda^n_0$. As we know that $H_3'$ is a morphism already is a map from each of $C(\Lambda^n_0)'$. Using \cref{isofibrationfunclift} twice, we see that $\widetilde{h_3}$ lifts to 
 \begin{equation*}
     \widetilde{h_3}' : C(\partial\Delta^{n-1})' \to \Ca
 \end{equation*}
 \item As $H_3'$ is derived from a morphism in $\CrrCpEal$, we see that the vertical edged between the first two rows are all identities (as the squares are cartesian). As equivalences are coCartesian morphisms from $\Ca \to \Delta^0$, we see that $\widetilde{h_3}'$ extends to a morphism 
 \begin{equation*}
     \widetilde{h_3}'' : C(\partial\Delta^{n-1)})'' \to \Ca
 \end{equation*}
 where $''$ means adjoining the top row and the right most vertical column.
 \item The morphism $\widetilde{h_3}''$ now extends to the whole simplex $C(\Delta^{n-1})$. The case $n=3$ can figured out pictorially. For higher $n$, we see that that commutativity of each exact squares follows from the existence of $C(\Delta^m)$ for $m <n-1$.
 \item The similar arguments also allow us to extend each of $C(\Lambda^n_0)''$ to $C(\Delta^n)$. This completes the claim that $H_3'$ extends to 
 \begin{equation*}
     H' : C(\Delta^n) \coprod_{C(\Delta^{n-1})} C(\Delta^n) \to \Ca.
 \end{equation*}
 
 \end{itemize}
    
\end{proof}
   \textit{Back to the main proof :} Once we have the existence of $H'$, the fact that all exact squares are pullback squares follows easily that lower simplices map to $\CrrCpEal$. Thus this shows $H'$ is indeed a morphism:
  \begin{equation*}
      H: (\Delta^0 \coprod \Delta^0)\boldsymbol{*} \Delta^{n-1} \to \CrrCpEal.
  \end{equation*}
  This completes the proof of the proposition.
\end{proof}
\begin{notation}
    Let $(\Ca,\E)$ be a  geometric setup. We define a geometric setup on the $\infty$-category $(\Ca^{\op{op}})^{\coprod,\op{op}}$. We write an edge $f$ of $\Ccc$ in the form $\{Y_j\}_{1\le j \le n} \to \{X_i\}_{1 \le i \le m}$ lying over $\alpha : [m] \to [n]$. We define two sets of $\E^{+},\E^{-}$ as follows:
    \begin{itemize}
        \item $\E^+$ consists of $f$ such that the induced age $Y_{\alpha(i)} \to X_i$ belongs to $\E$ for every $i \in \alpha^{-1}(\langle n \rangle ^0)$,
        \item $\E^-$ is subset of $\E^+$ where $\alpha$ is degenerate.
    \end{itemize}
    We shall denote :
    \begin{equation*}
       \op{Corr}(\Ca)^{\otimes}_{\E,\op{all}} := \op{Corr}(\Ccc)_{\E^{-},\op{all}}
    \end{equation*}
\end{notation}

The following lemma address which edges in $\Ccc$ are coCartesian. This is essentially \cite[Lemma 6.1.4]{liu2017enhanced}. We provide the proof for the sake of completeness.

\begin{lemma}\label{coCaredginCor}
    Let $f$ be an edge in $\CrrCopEal$ of the form :
    \begin{equation*}
    \begin{tikzcd}
               \{X_j\}_{1\le j \le m}  & \arrow[l,"f_1"] \{Y_i\}_{1 \le i \le n} \arrow[d,"f_2"] \\
        {} & \{Z_i\}_{1 \le i \le n} 
    \end{tikzcd}
 \end{equation*}
 lying over $\alpha : \langle m \rangle \to \langle n \rangle$, then $f$ is p-coCartesian if and only if: 
 \begin{enumerate}
     \item for every $1 \le i \le n$, the morphism $Y_i \to Z_i$ is an isomorphism,
     \item for every $1 \le i \le n$, the morphism $Y_i \to  X_j$ where $\alpha(j)=i$ exhibits $Y_i \cong \prod_{j,\alpha(j)=i } X_j$. 
 \end{enumerate}
\end{lemma}
\begin{proof}
We prove the two implications as follows:
\begin{itemize}
    \item     \textbf{Proving the conditions provided that the morphism is coCartesian:} let $f$ be a morphism which is $p$-coCartesian. We prove the two conditions as follows:
    \begin{enumerate}
        \item \textbf{Proving $Y_j \to Z_j$ is an isomorphism:} As every morphism in $\CrrCopEal$ factorises as an active and inert. It is enough to show that if $f$ is of the following form:
        \begin{equation*}
            \begin{tikzcd}
                X &  Y \arrow[l] \arrow[d,"q"] \\
                {} & Z
            \end{tikzcd}
    \end{equation*}
    which is $p$-coCartesian , then $q$ is an isomorphism. Recall that a $f$ is $p$-coCartesian implies that the following diagram 
    \begin{equation*}
        \begin{tikzcd}
            \Delta^{\{0,1\}} \arrow[dr,"f"] \arrow[d] & {} \\
            \Lambda^n_0 \arrow[r,"\sigma"] \arrow[d] & \CrrCopEal \arrow[d,"p"] \\
            \Delta^n \arrow[r] \arrow[ur,dotted,"\sigma'"] & N(\op{Fin}_*)
         \end{tikzcd}
    \end{equation*}
    admits a solution. For $n=2$, and for two $1$-simplices given by $f$ and the morphism :
    \begin{equation*}
        \begin{tikzcd}
            X & Y \arrow[l] \arrow[d,"\op{id}"] \\
            {} & Y
        \end{tikzcd}
    \end{equation*}
    which form a morphism $\Lambda^2_0 \to \CrrCopEal$, the lifting problem gives us the following $2$ simplex $\tau$
    \begin{equation*}
        \begin{tikzcd}
            X & Y \arrow[d,"q"] \arrow[l] & Y \arrow[l,"\op{id}"] \arrow[d,"h_1"] \\
            {} & Z  & Y'\arrow[l,"g_1"] \arrow[d,"h_2"] \\
            {} & {} & Y
        \end{tikzcd}
    \end{equation*}

    We shall use this simplex while using the lifting problem for $n=3$. We construct a morphism $\sigma: \Lambda^3_0 \to \CrrCopEal$ given by $3$ 2-simplices : $\tau$: and the other two given as follows:
    \begin{equation*}
        \begin{tikzcd}
            X & Y \arrow[l] \arrow[d,"q"] & Y \arrow[l,"\op{id}"] \arrow[d,"q"]  \\
            {} & Z  & Z \arrow[l,"\op{id}"] \arrow[d,"\op{id}"] \\
            {} & {} & Z
        \end{tikzcd}
        \quad; \quad
        \begin{tikzcd}
            X & Y \arrow[l] \arrow[d,"\op{id}"] & Y \arrow[l,"\op{id}"] \arrow[d,"\op{id}"] \\
            {} & Y  & Y \arrow[l,"\op{id}"] \arrow[d,"q"] \\
            {} & {} & Z.
        \end{tikzcd}
    \end{equation*}
    The morphism $\sigma$ extends  to $\sigma'$ given by the $3$-simplex

    \begin{equation*}
        \begin{tikzcd}
            X & Y \arrow[l] \arrow[d,"q"] & Y \arrow[l,"\op{id}"] \arrow[d,"h_1"]  & Y \arrow[d,"q"] \arrow[l,"\op{id}"] \\
            {} & Z  & Y'\arrow[l,"g_1"] \arrow[d,"h_2"] & Z \arrow[l,"g_1"] \arrow[d,"q'"] \\
            {} & {} & Y & Y \arrow[l,"\op{id}"] \arrow[d,"q"] \\
            {} & {} & {} & Z
            {} & 
        \end{tikzcd}
    \end{equation*}
    The last vertical arrow shows that $q \circ q' = \op{id}_Z$ and $q' \circ q = \op{id}_Y$. This shows that $q$ is an isomorphism.
    \item In this case as the previous one, it is enough to show that if $f$ is of the form:
    \begin{equation*}
    \begin{tikzcd}
                \{X_j\}_{1 \le j \le m} & \arrow[d,"\op{id}"] Y \arrow[l,"\{f_j\}"] \\
                {} & Y
    \end{tikzcd}
    \end{equation*}
    is coCartesian, then $f_j$'s make $Y$ as a product of $\prod_{j=1}^m X_j$. We notice that $f$ lies in the image of the inclusion map :
    \[(\Ca^{\op{op}})^{\coprod} \to \CrrCopEal.\] For every $n \ge 2$, let us consider the lifting problem 
    
 \begin{equation*}
        \begin{tikzcd}
            \Delta^{\{0,1\}} \arrow[dr,"f"] \arrow[d] & {} \\
            \Lambda^n_0 \arrow[r,"\sigma"] \arrow[d] & (\Ca^{\op{op}})^{\coprod} \arrow[r] \arrow[d,"p'"] & \CrrCopEal \arrow[dl,"p"] \\
            \Delta^n \arrow[r,"\tau"] \arrow[urr,dotted,"\sigma'"] & N(\op{Fin}_*)
         \end{tikzcd}
    \end{equation*}
    where :
    \begin{enumerate}
        \item $\beta : \{X_j\}_{1 \le j \le m} \leftarrow Y \leftarrow Y_2 \cdots Y_n$.
        \item $\tau : \langle m \rangle \to \langle 1 \rangle \cdots \langle 1 \rangle$.
    \end{enumerate}
    Now as $\sigma'$ lifts to $\CrrCpEal$ and the squares are pullback squares. This shows that $\sigma'$ does factors via $\sigma'' : \Delta^n \to (\Ca^{\op{op}})^{\coprod}$. This implies that $f$ is already coCartesian in $(\Ca^{\op{op}})^{\coprod}$. By \cite[Remark 2.4.3.4]{HA}, this gives $Y$ is product of $X_j$'s.
    \end{enumerate}

    \item \textbf{Proving that $f$ is coCartesian:} As stated before, we would like to have a solution to the lifting problem :
    
    \begin{equation*}
        \begin{tikzcd}
            \Delta^{\{0,1\}} \arrow[dr,"f"] \arrow[d] & {} \\
            \Lambda^n_0 \arrow[r,"\sigma"] \arrow[d] & \CrrCopEal \arrow[d,"p"] \\
            \Delta^n \arrow[r,"\tau"] \arrow[ur,dotted,"\sigma'"] & N(\op{Fin}_*)
         \end{tikzcd}
    \end{equation*}
    where $\tau$ is given by $\langle k_0 \rangle \xrightarrow{\alpha} \langle k_1 \rangle \cdots \langle k_n \rangle$.
    As coCartesian morphisms are stable under compositions and the maps in $\op{Fin}_*$ admit a factorization system, we will show the lifting problem only for the following two cases :
    \begin{enumerate}
        \item \textbf{$\alpha$ is inert:} In that case $f$ is an equivalence in $\op{Corr}(\Ca)_{\E,\op{all}}$. Mimicking the arguments in \cite[Proposition 2.4.3.3]{HA}, we see that $f$ is $p$-coCartesian. 
    \item \textbf{$\alpha:\langle n \rangle \to \langle 1 \rangle$ active:} As $k_1 =1$, it is also enough to consider the case when $k_i=1$ for $i=2,\cdots ,n$. The morphism $\sigma$ induces for every $1\le j \le m$ the diagrams :
    \begin{equation*}
        \begin{tikzcd}
            \Delta^{\{0,1\}} \arrow[dr,"f_j"] \arrow[d] & {} \\
            \Delta^0  \boldsymbol{*} \partial\Delta^{n-1} \cong \Lambda^n_0 \arrow[r,"\sigma_j"] 
 &\op{Corr}(\Ca)_{\E,\op{all}}
        \end{tikzcd}
    \end{equation*}
    induce a diagram :
    \begin{equation*}
        h: \coprod_{j=1}^m\Delta^0 \boldsymbol{*}\partial\Delta^{n-1} \to \CrrCpEal
    \end{equation*}
    where $h|_0 : : (\coprod_{j=1}^m\Delta^0)^{\triangleleft} \to \op{Corr}(\Ca)_{\E,\op{all}}$. As $Y$ is a product of $X_j$'s and $\Ca$ admits finite products implies $\op{Corr}(\Ca)$ admits finite coproducts along the inclusion $\Ca^{op} \subset \CrrCpEal$ (\cref{preservingcoproductsCor}) . Thus $h|_0$ is a limit diagram and hence extends to 
    \begin{equation*}
        h' : \coprod_{j=1}^m \Delta^0 \boldsymbol{*} \Delta^{n-1} \to \CrrCpEal. 
    \end{equation*}
    The morphism $h'$ is indeed the desired  $\sigma' : \Delta^n \to \CrrCopEal$.
    
    \end{enumerate}
\end{itemize}

\end{proof}
\begin{proposition}
    Let $(\Ca,\E)$ be an $\infty$-category such that $\Ca$ admits finite products. Then 
    \begin{equation*}
        p : \CrrCopEal \to N(\op{Fin}_*)
    \end{equation*}
    is a coCartesian symmetric monoidal $\infty$-category whose underlying $\infty$-category is $\op{Corr}(\Ca)_{\E,\op{all}}$.
\end{proposition}
\begin{proof}
   From \cref{coCaredginCor}, we see that for any morphism in $N(\op{Fin}_*)$  and an element $(X_1,X_2,\cdots,X_n) \in \CrrCopEal$, we can construct a coCartesian edge. Thus the morphism $p$ is a coCartesian fibration. It follows from that the constructions that on every fiber $\langle n \rangle \in \op{Fin}_*$, we see that the fiber is $\CrrCpEal^n$. 
\end{proof}
\begin{remark}
\begin{enumerate}
    \item  

    The first map also admits a lifting in the operadic level, i.e we have a morphism of $\infty$-operads :
    \begin{equation*}
        (\Ca^{\op{op}})^{\coprod} \to \CrrCopEal
    \end{equation*}
    \item The maps $\alpha_{\Ca,\E,K}$ and $\alpha_{K,\Ca,\E}'$ from \cref{Corrfunctorialmap} and \cref{Corrfunctorialmapdual} also upgrade to the operadic level :
    \begin{equation}\label{Corroperadfunctorial}
        \alpha_{K,\Ca,\E}^{\otimes}: \op{Corr}^{\op{all-cart}}(\op{Fun}(K,\Ca))^{\otimes}_{\tilde{E},\op{all}}  \to \op{Fun}(K,\op{Corr}(\Ca)^{\otimes}_{\E,\op{all}})  \end{equation} and
        \begin{equation}\label{Corroperadfunctorialdual}
            \alpha_{K,\Ca,\E}^{'\otimes}: \op{Corr}^{\E-\op{cart}}(\op{Fun}(K,\Ca))^{\otimes}_{\tilde{E},\op{all}}  \to \op{Fun}(K^{\op{op}},\op{Corr}(\Ca)^{\otimes}_{\E,\op{all}}) 
        \end{equation}
        respectively.
\end{enumerate}
   
\end{remark}

\subsection{Six-functor formalisms.}\label{App.A.2:_6FF}

\begin{definition}\cite[Definition A.5.6]{mann2022padic}
    Let $(\Ca,\E)$ be a geometric setup where $\Ca$ admits finite products. Then a \textit{pre-6-functor formalism} is a morphism of $\infty$-operads :
    \begin{equation*}
        \D_{(\Ca,\E)} : \CrrCopEal \to \op{Cat}^{\otimes}_{\infty}.
    \end{equation*}
    Given a pre-6-functor formalism, we introduce the following notations :
    \begin{enumerate}
        \item Restricting to the sub-operad $\Ccc$, we get a functor :
        \begin{equation*}
            \D^{*\otimes} : \Ca^{\op{op},\coprod} \to \op{Cat}^{\otimes}_{\infty}
        \end{equation*}
        This is equivalent to the functor 
        \begin{equation*}
            \D^{*\otimes} : \Ca^{\op{op}} \to \op{CAlg}(\op{Cat}_{\infty}).
        \end{equation*} 
        \item As $\D(X):=\D_{(\Ca,\E)}(X)$ is symmetric monoidal for every $X \in \Ca$, we get a tensor product structure :
        \begin{equation*}
            - \otimes - : \D(X) \times \D(X) \to \D(X).
        \end{equation*}
        \item Using the inclusion $\Ca_{\E} \to \CrrCpEal$, we get the following functor :
        \begin{equation*}
           \D_!: \Ca_{\E} \to \op{Cat}_{\infty}.
        \end{equation*}
    \end{enumerate}
\end{definition}

\begin{definition}\cite[Definition A.5.7]{mann2022padic}
    Let $(\Ca,\E)$ be a geometric setup such that $\Ca$ admits finite products. A \textit{6-functor formalism} is a pre-6-functor formalism $\D_{(\Ca,\E)}: \CrrCopEal \to \op{Cat}^{\otimes}_{\infty}$ such that 
    \begin{enumerate}
        \item For all $X\in \Ca$, $\D(X)$ is closed. The internal Hom  functor which is  the right adjoint of the tensor operation shall be denoted by  $\op{Hom}(-,-)$ .
        '\item For $f: X \to Y$ in $\Ca$, the morphism $f^*=\D^*(f)$ admits a right adjoint $f_* : \D(X) \to \D(Y)$. We denote by 
        \begin{equation*}
            \D_* : \Ca \to \op{Cat}_{\infty}
        \end{equation*}
        by the associated functor $X \mapsto \D(X), f \mapsto f_*$.
        \item For $f: X \to Y$ in $\Ca_{\E}$, the morphism $f_! : \D_!(f)$ admits a right adjoint $f^! : \D(Y) \to \D(X)$. We denote by
        \begin{equation*}
            \D^! : \Ca_{\E}^{op} \to \op{Cat}_{\infty}
        \end{equation*}
        the associated functor $X \mapsto \D(X), f \mapsto f^!$
    \end{enumerate}
\end{definition}
In the rest of this subsection, we describe how the a six functor formalism encodes properties like projection formula and base change.
\begin{enumerate}
    \item \textbf{Projection formula:} Let $f: X \to Y$ be a morphism in $E$, we consider the diagram :
    \begin{equation*}
        \begin{tikzcd}
            (X,Y) \arrow[r,"\sigma"] \arrow[d,"\tau"] & (Y,Y) \arrow[d,"\tau'"] \\
            X \arrow[r,"f"] & Y
        \end{tikzcd}
    \end{equation*}
    where :
    \begin{itemize}
        \item $\sigma':$\begin{equation*}
             \begin{tikzcd}
                (X,Y) & \arrow[l," \op{id}"] (X,Y) \arrow[d,"f_1"] \\
                {} & (Y,Y).
            \end{tikzcd}
        \end{equation*}
        where $f_1 = (f,\op{id})$.
        \item $\tau:$ \begin{equation*}
             \begin{tikzcd}
                (X,Y) & \arrow[l,"f_2"]  X \arrow[d,"\op{id}"] \\
               {} & X.
                \end{tikzcd}
        \end{equation*}
        where $f_2=(\op{id},f)$
        \item  $\tau':$\begin{equation*}
             \begin{tikzcd}
                (Y,Y) & Y \arrow[d,"\op{id}"] \arrow[l,"f_3"] \\
                {} & Y.
            \end{tikzcd}
        \end{equation*}
        where $f_3 = (\op{id},\op{id})$.
    \end{itemize}

This gives a morphism $\Delta^1 \times \Delta^1 \to \CrrCopEal$.  Applying $\D_{(\Ca,\E)}$, we get the following commutative square in $\op{Cat}_{\infty}$:
\begin{equation*}
    \begin{tikzcd}
        \D(X) \times \D(Y) \arrow[r,"f_! \times \op{id}"] \arrow[d, "\op{id} \otimes f^* "] & \D(Y) \times \D(Y) \arrow[d,"-\otimes-"]\\
        \D(X) \arrow[r,"f_!"] & \D(Y).
    \end{tikzcd}
\end{equation*}
which is the projection formula :
\begin{equation*}
    f_!((-) \otimes f^*(-)) \cong f_!(-) \otimes (-).
\end{equation*}
\item \textbf{Base change :} 
Let 
\begin{equation*}
    \begin{tikzcd}
        X' \arrow[d,"p'"] \arrow[r,"q'"] & X \arrow[d,"p"] \\
        Y' \arrow[r,"q"] & Y
    \end{tikzcd}
\end{equation*}
be a cartesian square with $p',p \in E$. We have a commutative square in $\CrrCopEal$ 
\begin{equation*}
    \begin{tikzcd}
        X \arrow[r] \arrow[d] & X' \arrow[d] \\
        Y\arrow[r,] & Y'
    \end{tikzcd}
\end{equation*}
which is comprised of two $2$-simplices
\begin{itemize}
    \item $\sigma_1:$
    \begin{equation*}
        \begin{tikzcd}
            X & \arrow[d,"p"] \arrow[l,"\op{id}"] X & X' \arrow[l,"q'"] \arrow[d,"p'"] \\
            {} & Y & Y'\arrow[l,"q"] \arrow[d,"\op{id}"] \\
            {} & {} & Y'.
         \end{tikzcd}
    \end{equation*}
    \item $\sigma_2':$
    \begin{equation*}
        \begin{tikzcd}
            X & X' \arrow[l,"q'"] \arrow[d,"\op{id}"] & X' \arrow[d,"\op{id}"] \arrow[l,"\op{id}"] \\
            {} & X' & X'\arrow[l,"\op{id}"]\arrow[d,"p'"]\\
            {} & {} & Y'.
        \end{tikzcd}
    \end{equation*}
    Applying $\D_{(\Ca,\E)}$ to the square, we get the commutative square :
    \begin{equation*}
        \begin{tikzcd}
            \D(X) \arrow[r,"q^{'*}"] \arrow[d,"p_!"] & \D(X') \arrow[d,"p'_!"]\\
            \D(Y) \arrow[r,"q^*"] & \D(Y')
        \end{tikzcd}
    \end{equation*}
    in $\op{Cat}_{\infty}$. Spelling this out, we get the base change equivalence $Ex^*_!$:
    \begin{equation}\label{App.:_Base_Change_Ex*_!}
    	q^*p_!\stackrel{Ex^*_!}{\simeq} p_!'q^{*'}
    \end{equation}
    
\end{itemize}
\end{enumerate}
\subsection{Extending Six-functor formalisms.}
In this subsection, we extend six functor formalism from smaller geometric setups to larger geometric setups. Similar extension results have been proven in DESCENT algorithms of Liu-Zheng (\cite[Section 4]{liu2017enhanced}) and also proved by Mann in his thesis (\cite[Lemma A.5.11,Proposition A.5.14]{mann2022padic}). We provide proofs of these results using the theory of localizations (in particular \cref{localizationcriterion}).

\subsection*{Extension along nice geometric pairs.}
\begin{definition}
An inclusion of two marked $\infty$-categories $(\Ca,\mathcal{S},\E) \subset (\Ca',\mathcal{S'},\E')$ is a \textit{nice geometric pair} if the following conditions hold :
    \begin{enumerate}
        \item Each of the four pairs $(\Ca,\S),(\Ca,\E),(\Ca,\S')$ and $(\Ca',\E')$ are geometric setups.
        \item $\mathcal{S}' \cap \Ca_1 =\mathcal{S}$.
        \item For  $X' \in \Ca'$, there exists a morphism $x : X \to X'$ called an \textit{atlas} such that $X \in \Ca$ and for every $Y \to X'$ where $Y \in \Ca$, the base change $Y \times_{X'} X \to Y$ lies in $\mathcal{S}$.
        \item For every $f : X' \to Y'$ in $\E'$ and for every atlas $y : Y \to Y'$, the base change morphism $X' \times_{Y'}Y \to Y$ is in $\E.$
    \end{enumerate}
    \end{definition}
    \begin{example}
        Let $\Ca= \op{Sch}$ be the category of schemes and $\Ca'= \op{Algst}$ be the $(2,1)$-category of algebraic stacks. Considering $\mathcal{S}$ and $\mathcal{S}'$ as smooth surjections of schemes and algebraic stacks respectively along with $\E$ and $\E'$ be representable morphisms of locally of finite type of schemes and algebraic stacks respectively, we see that $(\Ca,\mathcal{S},\E) \subset (\Ca',\mathcal{S}',\E')$ is a nice geometric pair.
    \end{example}
\begin{proposition}\label{extendingsixfunctorCside}
   Let $(\Ca,\mathcal{S},\E) \subset (\Ca',\mathcal{S}',\E')$ be a nice geometric pair. Let 
   \begin{equation}
       \D_{( \Ca,\E )} : \op{Corr}(\Ca)^{\otimes}_{\E,\op{all}} \to \op{Cat}^{\otimes}_{\infty}
   \end{equation}
   be a six-functor formalism such that the functor :
    \begin{equation*}
        \D^{*\otimes} : \C^{\op{op}} \to\op{CAlg}(\op{Pr}^L)
    \end{equation*}
    satisfies descent for $\S$-\v{C}ech-covers.
Then $\D_{(\Ca,\E)}$ can be extended to a six-functor formalism :
\begin{equation}
    \D_{(\Ca',\E')} : \op{Corr}(\Ca)^{\otimes}_{\E',\op{all}} \to \op{Cat}^{\otimes}_{\infty}.
\end{equation}
\end{proposition}
\begin{proof}
    We denote by $\op{Cov}(\Ca)$ the full subcategory of $\op{Fun}(N(\Delta)^{op},\Ca)$ spanned by \v{C}ech nerves of the atlases of elements $\Ca'$.An element of $\op{Cov}(\Ca)$ is given by $(x: X \to X',X')$. We have a canonical morphism 
    \begin{equation*}
        p : \CrrCovopEal \to \op{Corr}(\Ca')^{\otimes}_{\E',\op{all}}
    \end{equation*}
    Let $R$ be a collection of morphisms of the form: 
    \begin{equation*}
        \begin{tikzcd}
            (x_1 : X_1 \to X_1', X_1;\cdots; x_n : X_n \to X_n', X_n') \arrow[r,"f"] \arrow[d,"\op{id}"] &(y_1: Y_1 \to X_1', X_1';\cdots ; y_n : Y_n \to X_n', X_n') \\
            (x_1 : X_1 \to X_1', X_1;\cdots; x_n : X_n \to X_n', X_n') & {}.
        \end{tikzcd}
    \end{equation*} where $f=(f_i)_{i=1}^n$  and $f_i$ is a morphism of \v{C}ech nerves of $X_i$ for all $i$. The morphism $p$ sends $R$ to an equivalences. We have the following claim :
    \begin{claim}\label{claimforCorrCC'}
        The morphism $p$ is a localization of $\CrrCovopEal$ along $R$.
    \end{claim}
\begin{proof}[Proof of \cref{claimforCorrCC'}]
    It is enough to check the conditions of \cref{localizationcriterion}. First of all, the existence of atlases imply that $p$ is surjective on $n$-simplices. \\
    Second of all, if $(x: X \to X',X')$ and $(y: Y \to X',X')$ are two objects over $X'$, then we have the product of these two objects namely$( z: X \times _{X'} Y \to X',X')$.
Hence the conditions of \cref{localizationcriterion} are verified proving the claim.

\end{proof} 
We construct a morphism $\phi_{\Ca\Ca'} :\CrrCovopEal \to \op{Cat}_{\infty}^{\otimes}$ as follows:
\begin{itemize}
    \item The map $\alpha^{'\otimes}_{(\Delta_+)^{\op{op}},\Ca',\E'}$ (\cref{Corroperadfunctorialdual} provides a morphism :
    \begin{equation*}
        \phi_1 : \CrrCovopEal \xrightarrow{\alpha'_{(\Delta_+)^{\op{op}},\Ca',\E'}} \op{Fun}(N(\Delta_+),\op{Corr}(\Ca')^{\otimes}_{\E',\op{all}}) \xrightarrow{\op{res}} \op{Fun}(N(\Delta),\op{Corr}(\Ca)^{\otimes}_{\E,\op{all}})
    \end{equation*}
    where the second map is just restricting the augmented simplicial objects to the simplicial object. The simplicial object maps to $\CrrCpEal$ follows from the conditions in the proposition.
    \item The map $\D_{\Ca,\E}$ induces a functor :
    \begin{equation*}
       \phi_2: \op{Fun}(N(\Delta),\CrrCopEal) \to \op{Fun}(N(\Delta),\op{Cat}_{\infty}^{\otimes})
    \end{equation*}
    \item Taking limits of simplicial diagrams and using the theory of Kan extensions (\cite[Proposition 4.3.2.15]{HTT}, we have a limit functor :
    \begin{equation*}
        \phi_3 : \op{Fun}(N(\Delta),\op{Cat}_{\infty}^{\otimes}) \to \op{Cat}_{\infty}^{\otimes}
    \end{equation*}
    \item We define :
    \begin{equation*}
        \phi_{\Ca\Ca'}:= \phi_3 \circ \phi_2 \circ \phi_1.
    \end{equation*}
    In particular the morphism $\phi$ sends an object $(x_1: X_1 \to X_1', X_1';\cdots;x_n : X_n \to X_n',X_n')$ to the element $\Pi_{i=1}^n \op{lim}_{\bb \in \Delta} \D(X_{i,\bb})$.
\end{itemize}

We want to use that fact that as $p$ is a localization along $R$, then $\phi_{\Ca\Ca'}$ descends to a functor $\D_{\Ca'.\E'}$. For this we show the following claim :
\begin{claim}
    The functor $\phi_{\Ca\Ca'}$ sends $R$ to equivalences.
\end{claim}
\begin{proof}[Proof of the claim]
    Let $\X:= (X' \to X, X)$ and $\Y = (Y' \to X,X)$ be  two objects in $\op{Cov}(\Ca)$ and $f$ be a morphism in $\op{Cov}(\Ca)$ which induces an element of $R$. Let $\X \times \Y := (X' \times_{X} Y' \to X,X)$ be the product of $\X$ and $\Y$. We have the following diagram in $\op{Cov}(\Ca)$ :
    
    	\begin{center}
    	\begin{tikzpicture}[baseline={(0,1)}, scale=2]

    		\node (a) at (0,1) {$ \mc X\times \mc Y $};
    		\node (b) at (1, 1) {$ \mc Y $};
    		\node (c)  at (0,0.5) {$  \mc X $};
    		\node (e) at (0.2,0.75) {$  $};
    		\node (f) at (-0.75,1.5) {$ \mc X $};
    		\node (g) at (0.5,0.5) {$  $};

    		\path[font=\scriptsize,>= angle 90]

    		(a) edge [->] node [above ] {$ p_2 $} (b)
    		(a) edge [->] node [left] {$ p_1 $} (c)

    		(f) edge [bend right=-30,->] node [below] {$ f $} (b)
    		(f) edge [bend left=-30, double equal sign distance] node [below] {$  $} (c)
    		(f) edge [ ->] node [above] {$  $} (a);
    	\end{tikzpicture}
    \end{center}
    
%    \begin{equation}
%        \begin{tikzcd}
%            \X \arrow[dr] \arrow[ddr,"\op{id}", bend right=80] \arrow[drr,"f",bend left = 80] & {} & {} \\
%            {} & \X \times \Y \arrow[r,"p_1"] \arrow[d,"p_2"] & \Y \\
%            {} & \X & {}
%        \end{tikzcd}
%    \end{equation}
    \noindent which induces a diagram in $\CrrCovEopEal$ via the map:
    \[\op{Cov}(\Ca)^{\op{op}\coprod} \to \CrrCovopEal\]

As $\D^{*\otimes}$ satisfies descent along $\mathcal{S}$-\v{C}ech covers, we see that $\phi_{\Ca\Ca'}$ sends $p_1$ and $p_2$ to 
 equivalences. Using the two out of three property of equivalences, we see that $\phi_{\Ca\Ca'}$ sends $f$ to equivalences.
 \end{proof}
 \textit{Back to the proposition:} As $p'$ is a localization along $R$ and $\phi_{\Ca\Ca'}$ sends $R$ to equivalences, we get that $\phi_{\Ca\Ca'}$ descends to a map :
 \begin{equation*}
     \D_{(\Ca',\E')} : \op{Corr}(\Ca')^{\otimes}_{\E',\op{all}} \to \op{Cat}^{\otimes}_{\infty}.
 \end{equation*}
 
 \end{proof}
 \subsection*{Extension along exceptional pairs.}

 \begin{definition}
 An inclusion of $2$-marked $\infty$-categories $(\Ca,\mathcal{S},\E) \subset (\Ca,\mathcal{S},\E')$ is an \textit{exceptional pair} if the following conditions are satisfied :
    \begin{enumerate}
        \item The pairs $(\Ca,\S),(\Ca,\E),(\Ca,\E')$ are geometric setups.
        \item $\S \subset \E$.
        \item For every $f: X \to Y$ in $\E'$, there exists a morphism of augmented simplicial objects 
    \begin{equation*}
        f_{\bb} : X_{\bb} \to Y_{\bb}
        \end{equation*}
        where 
        \begin{itemize}
            \item $f_{-1}=f$
            \item $f_n \in E$ for $n \ge 0$,
            \item $X_{\bb} \to X$ and $Y_{\bb} \to Y$ are $\mathcal{S}$-hypercovers.
        \end{itemize}
        \end{enumerate}
        \end{definition}
        \begin{example}
         Let $\Ca$ be the category of schemes and $\E$ be the collection of morphisms which are separated and of finite type. Let $\mathcal{S}$ be the collection of Zariski covers. It follows from the definition that any morphism $f: X \to Y$ which is locally of finite type  admits a morphism of augmeneted simplicial objects $f_{\bb} : X_{\bb} \to Y_{\bb}$ where $X_{\bb}$ and $Y_{\bb}$ are Zariski hypercovers and $f_n$ for $n \ge 0$ is separated and of finite type. Thus letting $\E'$ to be the locally of finite type morphisms shows that $(\Ca,\mathcal{S},\E) \subset (\Ca,\mathcal{S}',\E')$ is an exceptional pair.
        \end{example}
 \begin{proposition}\label{extendingsixfunctorEside}
   Let $(\Ca,\mathcal{S},\E) \subset (\Ca,\mathcal{S},\E')$ be an exceptional pair. Let 
   \begin{equation}
       \D_{( \Ca,\E )} : \op{Corr}(\Ca)^{\otimes}_{\E,\op{all}} \to \op{Cat}^{\otimes}_{\infty}
   \end{equation}
   be a six-functor formalism such that the functor
    \begin{equation*}
        \D_! : \C_E \to \op{Pr}^L
    \end{equation*}
    satisfies codescent for $\mathcal{S}$-hypercovers.
  
Then $\D_{(\Ca,\E)}$ can be extended to a six-functor formalism :
\begin{equation*}
    \D_{(\Ca,\E')} : \op{Corr}(\Ca)^{\otimes}_{\E',\op{all}} \to \op{Cat}^{\otimes}_{\infty}.
\end{equation*}
\end{proposition}
\begin{proof}
    The proof of this proposition follows the similar ideas of the proof of the \cref{extendingsixfunctorCside}.\\

    Let $\op{Cov}_{\E}(\Ca)$ be the full-subcategory of $\op{Fun}(N(\Delta)^{op},\Ca)$ spanned by simplicial objects which are $\mathcal{S}$-hypercovers and spanned by $1$-simplices $f_{\bb} : X _{\bb} \to Y_{\bb}$ where if $f_{-1} \in \E'$, then $f_n \in \E$ for all $n \ge 0$.\\

    We have a canonical morphism :
    \begin{equation*}
        p_{\E} : \CrrCovEopEal \to \op{Corr}(\Ca)^{\otimes}_{\E',\op{all}}
    \end{equation*}
Let $R_{\E}$ be the collection of morphisms of the form :
\begin{equation*}
        \begin{tikzcd}
            (x_1 : X_1 \to X_1', X_1;\cdots; x_n : X_n \to X_n', X_n') \arrow[d,"f"] \arrow[r,"\op{id}"] &(x_1 : X_1 \to X_1', X_1;\cdots; x_n : X_n \to X_n', X_n')  \\
            (y_1: Y_1 \to X_1', X_1';\cdots ; y_n : Y_n \to X_n', X_n') & {}.
        \end{tikzcd}
    \end{equation*}
    where $f=(f_i)_{i=1}^n$ is a morphism of \v{C}ech nerves between $\mathcal{S}$-hypercovers. Mimicking the ideas from the previous proposition and using \cref{localizationcriterion}, we see that $p_{\E}$ is a localization along $R_{\E}$. \\

    We construct a morphism 
    \begin{equation*}
        \phi_{\E\E'} : \CrrCovEopEal \to \op{Cat}^{\otimes}_{\infty}
    \end{equation*}
    as follows:
    \begin{itemize}
        \item The map $\alpha_{(\Delta_+)^{op},\Ca,\E'}$(\cref{Corrfunctorialmap}) induces a morphism :
        \begin{equation*}
            \phi_{1\E} : \CrrCovEopEal \to \op{Fun}(N(\Delta_+)^{\op{op}},\op{Corr}(\Ca)^{\otimes}_{\E',\op{all}}) \xrightarrow{\op{res}} \op{Fun}(N(\Delta)^{\op{op}},\CrrCopEal)
        \end{equation*}
        where the second map is restriction to $\CrrCopEal$ as morphisms in $\E'$ induce a morphism of augmented simplicial objects $f_{\bb} : X_{\bb} \to Y_{\bb}$ where $f_n \in \E$ for $n \ge 0$.
        \item The map $\D_{(\Ca,\E)}$ induces a functor :
        \begin{equation*}
            \phi_{2\E} : \op{Fun}(N(\Delta)^{\op{op}},\CrrCopEal) \to \op{Fun}(N(\Delta)^{\op{op}},\op{Cat}^{\otimes}_{\infty}).
        \end{equation*}
        \item Using the theory of Kan extensions (\cite[Proposition 4.3.2.15]{HTT}), we get a morphism :
        \begin{equation*}
            \phi_{3\E} : \op{Fun}(N(\Delta_+)^{\op{op}}, \op{Cat}^{\otimes}_{\infty}) \to \op{Cat}^{\otimes}_{\infty}.
        \end{equation*}
        \item We define :
        \begin{equation*}
            \phi_{\E\E'} := \phi_{3\E} \circ \phi_{2\E} \circ \phi_{1\E}
        \end{equation*}
        In particular, we see that $\phi_{\E\E'}(X_{\bb} \to X) \cong  \op{colim}_{\bb \in \Delta} \D_!(X_{\bb})$. In particular for any map $f: X \to Y$ in $\E'$ and $f_{\bb}: X_{\bb} \to Y_{\bb}$ a morphism of augmented simplicial objects where $f_n \in E$ , we see that \[\phi_{\E\E'}(f_{\bb}) \cong \op{colim}_{\bb \in \Delta}\D_!(f). \]
    \end{itemize}
    Again mimicking the ideas in the proof of \cref{extendingsixfunctorCside} and using that $\D_!$ has codescent along $\mathcal{S}$-hypercovers, we see that $\phi_{\E\E'}$ sends $R_{\E}$ to equivalences. As $p_{\E}$ is a localization and $\phi_{\E\E'}$ sends $R_{\E}$ to equivalences, $\phi_{\E\E'}$ descends to a morphism :
    \begin{equation*}
        \D_{(\Ca,\E')} : \op{Corr}(\Ca)^{\otimes}_{\E',\op{all}} \to \op{Cat}_{\infty}^{\otimes},
    \end{equation*}

\end{proof}
\section{A criterion regarding localizations.}
In this section, we prove a proposition under what conditions a morphism of $\infty$-categories is a localization.
Let us recall the notion of Dwyer-Kan localizations.

\begin{definition}\cite[Definition 2.4.2]{Land_introductionQC}
Let $\Ca$ be an $\infty$-category and let $S \subset \Ca_1$  be a set of morphisms. A functor $\Ca \to \Ca[S^{-1}]$ is a \textit{Dwyer-Kan localization} of $\Ca$ along $S$ if for every auxiliary $\infty$-category $\D$ the functor :
\begin{equation*}
    \op{Fun}(\Ca[S^{-1}],\D) \to \op{Fun}(\Ca,\D)
\end{equation*}
    is fully-faithful and its essential image is all such functors $\Ca \to \D$ which sends $S$ to equivalences.
\end{definition}
\begin{remark}
    By \cite[Lemma 2.4.6]{Land_introductionQC}, localization exists along all morphisms of $\Ca$.
\end{remark}
\begin{proposition}\label{localizationcriterion}
    Let $\Ca$ be an $\infty$-category and $R$ be a set of morphisms in $\Ca$. Let $p : \Ca \to \D$ which sends $R$ to equivalences. Suppose we have the following conditions: 

\begin{itemize}
    \item $p$ is surjective on $n$-simplices for $n \ge 0$.
    \item For every $d \in \D$, the $\infty$-category $\Ca_d$ admits products. In particular $x,y \in \Ca_d$, the projection maps $x \times y \to x , x \times y \to y$ are in $R$
\end{itemize}

    Then $p$ is a Dwyer-Kan localization along $R$.
\end{proposition}
\begin{proof}
    As $p$ sends $R$ to equivalences, it factorizes to a map \[p':\Ca[R^{-1}] \to \D. \]
The goal is to show $p'$ is a categorical equivalence. As trivial fibrations are categorical equivalences (\cite[Proposition 2.2.12]{Land_introductionQC}),we show that $p'$ is a trivial fibration i.e for $m \ge 0$, the diagram below admits a solution:
\begin{equation*}
    \begin{tikzcd}
        \partial\Delta^m \arrow[r,"\sigma"] \arrow[d,hookrightarrow] & \Ca[R^{-1}] \arrow[d,"p'"] \\
        \Delta^m \arrow[ur,dotted,] \arrow[r,"\tau"] & \D
    \end{tikzcd}
\end{equation*}
The case $m =0$ follows as map $p'$ is surjective on objects. Thus we assume $m >0$.
We know that $\tau$ lifts to $\tau' : \Delta^n \to \Ca$ which even restricts to $ \sigma':\partial\Delta^n \to \Ca$. Let $X_0,X_1,....X_n$ be the vertices of $\sigma'$ and $Y_0,Y_1,Y_2,...Y_n$ be vertices of $\sigma$. By condition of the proposition, we have vertices $Z_i := X _i \times Y_i$ with maps $p_i: Z_i \to X_i, q_i : Z_i 
\to Y_i$ are in $R$.
We make the following two observations :
\begin{enumerate}
\item The vertices $X_i,Z_i$ when realized in $\Ca[R^{-1}]$ along with $\sigma'$ amalgamate to define a morphism 
\begin{equation}
    h_1 ; \Delta^1 \times (\coprod_{j=0}^m \Delta^0) \times_{[1] \times (\coprod_{j=0}^m \Delta^0)} \Delta^m \to \Ca[R^{-1}]
\end{equation}
where for all $0\le k \le n$, $h_1|_{\Delta^1 \times [k]}$ is an equivalence.
\item The vertices $Y_i,Z_i$ along with $\sigma$ amalgamate to define a morphism :
\begin{equation}
    h_2 : \Delta^1 \times(\coprod_{j=0}^m \Delta^0) \times_{[1] \times (\coprod_{j=0}^m\Delta^0)} \partial\Delta^m \to \Ca[R^{-1}]
\end{equation}
where for all $0\le k \le n$, $h_2|_{\Delta^1 \times [k]}$ is an equivalence

 \end{enumerate}
The proof follows from the following claim : 
\begin{claim}\label{isofibrationfunclift}
		Let $\E$ be an $\infty$-category and $n \ge 1$.  Let \[h: \Delta^1 \times \partial\Delta^n \coprod_{\{0\} \times \partial\Delta^n} \{0\} \times\Delta^n \to \E \] be a morphism such that $h|_{\Delta^1 \times [k]}: \Delta^1 \to \D $ is an equivalence for all $ 0 \le k \le n$.  Then there exists a morphism $h': \Delta^n \times \Delta^1 \to \D$ such that the diagram 
		\begin{equation*}
			\begin{tikzcd}
				\Delta^1 \times \partial\Delta^n \coprod_{\{0\} \times \partial\Delta^n} \{0\} \times\Delta^n \arrow[r,"f"] \arrow[d,hookrightarrow] & \E  \\
				\Delta^n \times \Delta^1 \arrow[ur,"h'"] & {}
			\end{tikzcd}
		\end{equation*}
		commutes.
	\end{claim}
	\begin{proof}[Proof of claim]
		The morphism $h$ gives us the following commutative diagram 
		\begin{equation*}
			\begin{tikzcd}
				\{0\} \arrow[r] \arrow[d, hookrightarrow] & \op{Fun}(\Delta^n,\E) \arrow[d,"p"] \\
				\Delta^1 \arrow[r,"g'"] & \op{Fun}(\partial\Delta^n,\E)
			\end{tikzcd}
		\end{equation*}
		where $g'$ is an equivalence. As $i: \partial \Delta^n \to \Delta^n$ is bijective on $0$ simplices applying \cite[Proposition 2.2.5]{Land_introductionQC} gives us that $p$ is an isofibration. Thus there exists a morphism $h': \Delta^1 \to \op{Fun}(\Delta^n,\E)$ which extends $h$. This completes the proof.
	\end{proof}
 \textit{Back to proving the proposition:} The idea is to apply the  \cref{isofibrationfunclift} two times as follows:
 \begin{enumerate}
     \item Applying proposiion to $h_1$ starting from all $n=1$-subsimplices to $n=m$ extends $h_1$ to the following morphism : 
\begin{equation}
    h_1': \Delta^1 \times \Delta^m \to \Ca[R^{-1}]
\end{equation}
In particular the $Z_i$'s are now vetices of an $m$-simplex $\sigma'': h_1'|_{\Delta^1 \times [0]} : \Delta^m \to \Ca[R^{-1}]$.
     \item Applying the proposition with the amalgamation of $h_2$ and $\sigma''$ extends $h_2$ to the following morphism: 
     \begin{equation}
         h_2' : \Delta^1 \times \Delta^m \to \Ca[R^{-1}]
     \end{equation}
     In particular $h_2'|_{[1] \times \Delta^m}: \Delta^m \to \Ca[R^{-1}]$ extends $\sigma$  solving the lifting problem. 
 \end{enumerate}
	This completes the proof that $p'$ is a trivial fibration and the proof of the proposition is complete. 
\end{proof}
\section{Exchange Transformations}
For future reference, let us record here all the available exchange transformations we have for the motivic homotopy category of $NL$-stacks. The natural transformations of \textit{base change} $ Ex_!^* $, \textit{smooth base change}  $ Ex_{\#}^* $, and \textit{proper base change} $ Ex_*^* $ were constructed respectively in \cite[Theorem 4.5.1(3), Proposition 4.1.2, Proposition 4.1.5]{Chowdhury}. We already upgraded $Ex^*_!$ and $Ex^*_{\#}$ to the non-representable case by \cref{Sect.3.2:_SH*!_Corr_functor} and \cref{Sec.5:_smooth_cl_BC}.

%Using the results in \cite{ChoDA24}, one can now remove the representability assumptions of \cite{Chowdhury24} for the natural transformations $Ex^*_!$, by \cite[Proposition 3.22]{ChoDA24}, and for $Ex_{\#}^*$, using \cite[4.7.4.18]{HA} on the natural transformation $Ex^{*,*}$ that witness composability of pullbacks.  

But we are still missing some natural transformations from our list, so let us state and prove the following:

\begin{proposition}\label{App:_Exchange_Transf}
	Consider the following cartesian diagram in $ \oocatname{ASt}^{\leq 1} $: 
		\begin{center}
		\begin{tikzpicture}[baseline={(0,-1)}, scale=2]
			\node (a) at (0,1) {$ \mc W $};
			\node (b) at (1, 1) {$ \mc Y $};
			\node (c)  at (0,0) {$  \mc Z$};
			\node (d) at (1,0) {$ \mc X $};
			\node (e) at (0.2,0.75) {$ \ulcorner $};
			\node (f) at (0.5,0.5) {$ \Delta $};

			\path[font=\scriptsize,>= angle 90]
			
			(a) edge [->] node [above ] {$ g $} (b)
			(a) edge [->] node [left] {$ q $} (c)
			(b) edge[->] node [right] {$ p $} (d)
			(c) edge [->] node [below] {$ f $} (d);
		\end{tikzpicture}
	\end{center}
\begin{enumerate}
    \item [\textit{(i)}] If $ p $ is lft, then we have an exchange equivalence:
    \[ Ex^*_!: p_!f^*\simeq g^*q_! \]
    \item [\textit{(ii)}] If $ p $ is lft, then we have an exchange equivalence:
    \[ Ex^!_*: p^!f_*\simeq g_*q^! \]

    \item [\textit{(iii)}] If $ p $ is smooth, then we have an exchange equivalence:
    \[ Ex^*_{\#}: p\epfs f^*\simeq g^*q\epfs \]
    
    \item [\textit{(iv)}] We have an exchange equivalence:
    \[ Ex^*_*: p^*f_*\longrightarrow g_*q^* \]
    \noindent that is an equivalence if $p$ is representable and proper.
    
	\item [\textit{(v)}]  Suppose $ f $ is smooth and $ p $ is representable and proper. Then we have a natural exchange transformation:
    \[ Ex_{\#*}: f_{\#}q_* \longrightarrow p_*g_{\#} \]
    \noindent and moreover this is an equivalence.
    \item [\textit{(vi)}] If $ p $ is lft, then we have a natural transformation:
    \[ Ex_{!*}: p_!g_*\longrightarrow f_*q_! \]
    \noindent If $ f $ is representable and proper the above transformation is an equivalence.
    \item [\textit{(vii)}] If $ p $ is lft and $ f $ is smooth, we have a natural  exchange transformation:
    \[ Ex_{!\#}: f\epfs q\epf \longrightarrow p\epf g\epfs \]
    \noindent  and moreover this is an equivalence.
    \item [\textit{(viii)}] If $ p $ is lft, then we have a natural transformation:
    \[ Ex^{*!}: g^*p^!\longrightarrow q^!f^* \]
    \noindent If $ f $ is smooth the above transformation is an equivalence.

\end{enumerate}
\end{proposition}

\begin{proof}
    The base change transformation $ Ex_!^* $ follows directly from the construction of the functor $\SH^*_!$ in \Cref{Sect.3.2:_SH*!_Corr_functor}. The exchange transformation in $ (ii) $ is obtained via adjunctions from $ Ex_!^* $.\\

	\noindent The construction of the natural transformations is exactly like in \cite[\S 6]{Hoyois_Equiv_Six_Op} and in \cite[\S 2.4]{Cisinski_2019}. 
    \noindent For $(iii)$ we already constructed and proved everything in \cref{Sec.5:_smooth_cl_BC}.
 
    \noindent For $(iv)$, the exchange transformation $Ex^*_*$ is given by the composite:
   \[ f^*p_* \overset{\eta^*_*(q)}{\longrightarrow} f^*p_*g_*g^* \simeq f^*f_*q_*g^* \overset{\varepsilon^*_*(f)}{\longrightarrow} q_*g^* \]
    For $ (v) $, $ Ex_{\#*} $ is given by the following composition:
	\[ Ex_{\#*}: f\epfs q_* \stackrel{\eta_*^*(p)}{\longrightarrow} p_*p^*f\epfs q_* \stackrel{Ex_{\#}^*}{\simeq} p_*g\epfs q^*q_* \stackrel{\epsilon^*_*(q)}{\longrightarrow} p_*g\epfs \]
	For $ (vi) $ and $ (viii) $, the exchange transformations are given by:
	\[ Ex_{!*}: p\epf g_* \stackrel{\eta_*^*(f)}{\longrightarrow} f_*f^*p\epf g_* \stackrel{Ex_{!}^*}{\simeq} f_*q\epf g^*g_* \stackrel{\epsilon^*_*(g)}{\longrightarrow} f_*q\epf  \]
	\[ Ex^{*!}: g^*p^! \stackrel{\eta_*^*(f)}{\longrightarrow} g^*p^!f_*f^* \stackrel{Ex^{!}_*}{\simeq}  g^*g_*q^!f^* \stackrel{\epsilon^*_*(g)}{\longrightarrow} q^!f^* \]
	For $ (vii) $, as we have seen multiple time in Section 4, the exchange transformation $ Ex_{!\#} $ is given by the composition:
	\[ Ex_{!\#}: f\epfs q\epf \stackrel{\epsilon_{\#}^*(g)}{\longrightarrow} f\epfs q\epf g^*g\epfs\stackrel{Ex^*_!}{\simeq} f\epfs f^*p\epf g\epfs \stackrel{\eta_{\#}^*(f)}{\longrightarrow} p\epf g\epfs \]
	
	Now we only need to prove that the exchange transformations are equivalences (under the appropriate hypothesis).

	Consider the following diagram:
	
\begin{equation}\label{ch1:_cube_Smooth-Proper_BC}
		\begin{tikzpicture}[baseline=(current  bounding  box), scale=0.35]
			
			\matrix (m) [matrix of math nodes, row sep=2. em,
			column sep=2.5 em]{
				W & &  Y &\\
				& Z & & X & \\
				\mc W & &   \mc Y &\\ 
				& \mc Z & &  \phantom{c}\mc X \\};
			\path[-stealth]
			(m-1-1) edge [->] node [above] {$ \tilde{g} $} (m-1-3) edge[->] node [left=2mm, below=-1mm] {$ \tilde{q} $} (m-2-2)
			edge [->] node [left] {$ w $}  (m-3-1)
			(m-1-3) edge [->] node [below] {\phantom{ciao}$ y $} (m-3-3) edge  [->]  node [right=1mm, above=0mm] {$ \tilde{p} $} (m-2-4)
			(m-2-2) edge [-,line width=6pt,draw=white] (m-2-4) edge [->] node [above=3mm, left=1mm] {$\tilde{f} $}  (m-2-4) edge [right hook->] (m-4-2)
			(m-3-1) edge [->] node [above] {\phantom{ciao} $ g $} (m-3-3)
			edge [->] node [left=2mm, below=-1] {$ q $}  (m-4-2)
			(m-4-2) edge [->] node [below] {$ f $}  (m-4-4)
			(m-3-3) edge [->] node [right=1mm, above=0mm] {$ p$}  (m-4-4)
			(m-2-2) edge [-,line width=6pt,draw=white] (m-4-2)
			(m-2-2) edge [->] node [above left] {$ z $} (m-4-2)
			(m-2-4)edge [->] node [right] {$ x $} (m-4-4);
		\end{tikzpicture}
\end{equation}
	\noindent where every square is cartesian and the vertical maps $ x,y,w,z $ are NL-atlases of $ \mc X, \mc Y, \mc W, \mc Z $ respectively.  The induced maps between the atlases will be denoted by $ \tilde{f}, \tilde{g}, \tilde{p}, \tilde{q} $.
	
\begin{enumerate}
    \item [(i)$+$(ii)] The fact that $ Ex_*^! $ is an equivalence follows by adjunction from the fact that $ Ex_!^*$ is an equivalence too as a consequence of the construction of the functor $\SH^*_!$ in \Cref{Sect.3.2:_SH*!_Corr_functor}.
    \item [(iii)] The claim was already proved in \cref{Sec.5:_smooth_cl_BC}.
    \item [(iv)] If we assume that $p$ is representable and proper, it is not hard to see that $p_*=p_!$. Then we can reduce to the schematic case using arguments similar to the ones in the proof of \Cref{lemma_Ex-sharp-shriek_mixed_repr_non_repr_case} and conclude for example by \cite[6.10]{Hoyois_Equiv_Six_Op}.
	\item [(v)]	Suppose that $f$ is smooth and $p$ is representable and proper. Let us start proving that $ Ex_{\#*} $ is an equivalence. First assume that $ f $ and $ p $ are both representable. By the conservativity of $ x^* $, proving that:
	\begin{equation}\label{ch1:_smooth-proper-BC}
		Ex_{\#*}: f\epfs q_* \stackrel{\eta_*^*(p)}{\longrightarrow} p_*p^*f\epfs q_* \stackrel{Ex_{\#}^*}{\simeq} p_*g\epfs q^*q_* \stackrel{\epsilon^*_*(q)}{\longrightarrow} p_*g\epfs 
	\end{equation}
\noindent is an equivalence, is the same as proving that:
\begin{equation}\label{ch1:_smooth-proper-BC-2}
	x^*Ex_{\#*}: x^*f\epfs q_* \stackrel{x^*\eta_*^*(p)}{\longrightarrow} x^*p_*p^*f\epfs q_* \stackrel{x^*Ex_{\#}^*}{\simeq} x^*p_*g\epfs q^*q_* \stackrel{x^*\epsilon^*_*(q)}{\longrightarrow} x^*p_*g\epfs 
\end{equation}

\noindent is an equivalence. But we can rewrite $ x^*Ex_{\#*} $ as:

\begin{center}
	\begin{tikzpicture}[scale=1.85]
		%\begin{pgfonlayer}{nodelayer}
		\node (1) at (0,0) {$ x^*f{}_{\#}q_* $};
		\node (2) at (0,-0.5) {\rotatebox{-90}{$ \simeq $}}; 
		\node (2.2) at (0.25,-0.52) {$ \scriptstyle Ex_{\#}^* $};
		\node (3) at (0,-1) {$ \tilde{f}{}_{\#}z^*q_*$};
		\node  (4) at (0,-1.5) {\rotatebox{-90}{$ \simeq $}};
		\node  (4.2) at (0.25,-1.52) { $ \scriptstyle Ex^*_*$};
		\node  (5) at (0,-2) { $ \tilde{f}\epfs \tilde{q}_* w^* $};
		\node (6) at (2,0) {$x^*p_*p^*f{}_{\#}q_*$};
		
		\node (7) at (2,-0.5) {\rotatebox{-90}{$ \simeq $}}; 
		\node (7.2) at (2.25,-0.52) {$ \scriptstyle Ex_{*}^* $};
		
		\node (8) at (2,-1) {$\tilde{p}_*y^*p^*f{}_{\#}q_*$};
		\node  (9) at (2,-1.5) {\rotatebox{-90}{$ \simeq $}};
		\node (10) at (2,-2) {$\tilde{p}_*\tilde{p}^* x^* f{}_{\#} q_*$};
		\node (11) at (2,-2.5) {\rotatebox{-90}{$ \simeq $}};
		\node (11.2) at (2.25,-2.52) {$ \scriptstyle Ex_{\#}^* $};
		\node  (12) at (2, -3) {$\tilde{p}_*\tilde{p}^*\tilde{f}{}_{\#} z^*q_*$};
		\node  (13) at (2,-3.5) {\rotatebox{-90}{$ \simeq $}};
		\node  (13.2) at (2.25,-3.52) {$ \scriptstyle  Ex_*^* $};
		
		\node (14) at (2,-4) {$\tilde{p}_*\tilde{p}^* \tilde{f}{}_{\#} \tilde{q}_* w^*$};
		\node (15) at (3.5,0) {$x^* p_*g{}_{\#} q^*q_*$};
		\node (16) at (3.5,-0.5) {\rotatebox{-90}{$ \simeq $}};
		\node (16.2) at (3.75,-0.52) { $ \scriptstyle Ex_{*}^* $};
		\node (17) at (3.5,-1) {$\tilde{p}_*y^*g{}_{\#} q^*q_*$};
		\node (18) at (3.5,-1.5) {\rotatebox{-90}{$ \simeq $}};
		\node (18.2) at (3.75,-1.52) {$ \scriptstyle Ex_{\#}^* $};
		\node (19) at (3.5,-2) {$\tilde{p_*} \tilde{g}{}_{\#} w^* q^*q_*$};
		\node (20) at (3.5,-2.5) {\rotatebox{-90}{$ \simeq $}};
		\node (21) at (3.5,-3) {$\tilde{p}_* \tilde{g}{}_{\#} \tilde{q}^* z^* q_*$};
		\node (22) at (3.5,-3.5) {\rotatebox{-90}{$ \simeq $}};
		\node (22.2) at (3.75,-3.52) {$ \scriptstyle Ex_{*}^* $};
		
		\node (23) at (3.5,-4) {$\tilde{p}_* \tilde{g}{}_{\#} \tilde{q}^* \tilde{q_*}w^*$};
		\node (24) at (5.5,0) {$x^*p_*g{}_{\#}$};
		\node (25) at (5.5,-0.5) {\rotatebox{-90}{$ \simeq $}};
		\node (25) at (5.75,-0.52) {$ \scriptstyle  Ex_{*}^* $};
		\node (26) at (5.5,-1) {$\tilde{p}_*y^* g{}_{\#}$};
		\node (27) at (5.5,-1.5) {\rotatebox{-90}{$ \simeq $}};
		\node (27.2) at (5.75,-1.52) {$ \scriptstyle Ex_{\#}^* $};
		\node (28) at (5.5,-2) {$\tilde{p}_* \tilde{g}{}_{\#} w^*$};
		\node (29) at (5.5,-4) {$\tilde{p}_* \tilde{g}{}_{\#} w^*$};
		
		\node (30) at (2.8,0.1) {$ \stackrel{Ex_{\#}^*}{\simeq}$};
		\node (31) at (2.8,-0.9) {$ \stackrel{Ex_{\#}^*}{\simeq}$};
		%\node (32) at (2.8,-1.9) {$ \stackrel{Ex_{\#}^*}{\simeq}$};
		\node (33) at (2.8,-2.9) {$ \stackrel{Ex_{\#}^*}{\simeq}$};
		\node (34) at (2.8,-3.9) {$ \stackrel{Ex_{\#}^*}{\simeq}$};

		%\end{pgfonlayer}
		%\begin{pgfonlayer}{edgelayer}
		\path[font=\scriptsize,>= angle 90]
		
		(1) edge [->] node [above ] {$ \epsilon_*^*(p) $} (6)
		(1) edge [->] node [left ] {$ \epsilon_*^*(\tilde{p})  $} (10)
		(3) edge [->] node [left] {$ \epsilon_*^*(\tilde{p}) $} (12)
		(5) edge[->] node [left] {$ \epsilon_*^*(\tilde{p}) $} (14)
		%(6) edge [] node [above] {$ \simeq $} (15)
		%(8) edge [->] node [left] {$  $} (17)
		%(14) edge[->] node [right] {$  $} (23)
		(15) edge [->] node [below] {$ \eta_*^*(q) $} (24)
		(17) edge [->] node [below] {$ \eta_*^*(q) $} (26)
		(19) edge[->] node [below] {$ \eta_*^*(q) $} (28)
		(23) edge [->] node [below] {$ \eta_*^*(\tilde{q}) $} (29)
		(28) edge [-, double distance=2pt] node [left] {$  $} (29);
		%	\end{pgfonlayer}
\end{tikzpicture}
\end{center}
\noindent We can fill each cell of the diagram by the naturality of the adjunctions and the exchange transformations we already have. Hence looking at the bottom row of the big diagram above, $ x^*Ex_{\#*} $ in \eqref{ch1:_smooth-proper-BC-2} becomes:
\begin{equation}\label{ch1:_smooth-proper-BC-3}
	x^*Ex_{\#*}: \tilde{f}\epfs \tilde{q}_* w^* \stackrel{\epsilon_*^*(\tilde{p})}{\longrightarrow} \tilde{p}_*\tilde{p}^* \tilde{f}{}_{\#} \tilde{q}_* w^* \stackrel{Ex_{\#}^*}{\simeq} \tilde{p}_* \tilde{g}{}_{\#} \tilde{q}^* \tilde{q_*}w^* \stackrel{\eta_*^*(\tilde{q})}{\longrightarrow} \tilde{p}_*\tilde{g}\epfs w^* 
\end{equation}

But we already know, from \cite[Proposition 6.12]{Hoyois_Equiv_Six_Op}, that the smooth-proper exchange transformation (induced by the top square made by the atlases in \eqref{ch1:_cube_Smooth-Proper_BC}):
\[ Ex_{\#*}: \tilde{f}\epfs \tilde{q}_*  \stackrel{\epsilon_*^*(\tilde{p})}{\longrightarrow} \tilde{p}_*\tilde{p}^* \tilde{f}{}_{\#} \tilde{q}_*  \stackrel{Ex_{\#}^*}{\simeq} \tilde{p}_* \tilde{g}{}_{\#} \tilde{q}^* \tilde{q_*} \stackrel{\eta_*^*(\tilde{q})}{\longrightarrow} \tilde{p}_*\tilde{g}\epfs   \]
\noindent is an equivalence. Thus \eqref{ch1:_smooth-proper-BC-3} is an equivalence too and, by conservativity of $ x^* $, we get that the exchange transformation:
\[ 	Ex_{\#*}: f\epfs q_* \stackrel{\eta_*^*(p)}{\longrightarrow} p_*p^*f\epfs q_* \stackrel{Ex_{\#}^*}{\simeq} p_*g\epfs q^*q_* \stackrel{\epsilon^*_*(q)}{\longrightarrow} p_*g\epfs  \]
\noindent is an equivalence as well.\\

To extend the result to the case when $f$ is not representable, one can proceed mimicking the same argument given for a similar statement in \Cref{lemma_Ex-sharp-shriek_mixed_repr_non_repr_case}. Details are left to the reader.

    \item [(vi)] %Use the same notation of \eqref{ch1:_cube_Smooth-Proper_BC}. First 
        Suppose that $ p $ is lft and $ f $ is proper and representable. Then it is not hard to show that we have $f_!=f_*$, and the exchange transformation $Ex_{!*}$ becomes $Ex_{!!}$, i.e. simply the exchange transformation witnessing the composability of exceptional functors, that is clearly an equivalence and we are done.
        %Suppose that $ p $ is lft and $ f $ is proper and representable. Then it is not hard to show that we have $f_!=f_*$, and the exchange transformation $Ex$
        %To show that $ Ex_{!*} $ is an equivalence, by conservativity of $ x^* $, it is enough to show that:
        %\[ x^*Ex_{!*}: x^* p\epf g_* \longrightarrow x^* f_*q\epf \]
        %\noindent is an equivalence. Similarly to what we did for $ (ii) $, it is possible to rewrite $ x^*Ex_{!*} $ as:
        %\[ \tilde{p}\epf\tilde{g}_* w^* \stackrel{Ex_{!*}w^*}{\longrightarrow} %\tilde{f}_*\tilde{q}\epf w^* \]
        %\noindent and this is an equivalence since $ Ex_{!*}: \tilde{p}\epf\tilde{g}_*\stackrel{\sim}{\rightarrow} \tilde{f}_*\tilde{q}\epf  $ by \cite[\S 6.2]{Hoyois_Six_Operations}. The general case where $p$ is not representable follows again by similar arguments as the ones used in \cite[Lemma 4.15]{ChoDA24}.

%Use the same notation as in the diagram \eqref{ch1:_cube_Smooth-Proper_BC}. Suppose $ p $ is separated of finite type and $ f $ is smooth. To show that $ Ex^{*!} $ is an equivalence, by conservativity of $ x^* $, it is enough to show that:
%\[ x^*Ex^{*!}: x^* g^* p^! \longrightarrow x^* q^!f^* \]
%\noindent is an equivalence. Once again, it is possible to rewrite $ x^*Ex_{!*} $ as:
%\[ \tilde{g}\epf\tilde{p}_* w^* \stackrel{Ex_{!*}w^*}{\longrightarrow} \tilde{q}_*\tilde{f}\epf w^* \]
%\noindent and this is an equivalence since $ Ex_{!*}: \tilde{g}\epf\tilde{p}_*\stackrel{\sim}{\rightarrow} \tilde{q}_*\tilde{f}\epf  $ by \cite[\S 6.2]{Hoyois_Six_Operations}.

\item [(vii)] %Use the same notation of \eqref{ch1:_cube_Smooth-Proper_BC}. 
Assume that $ f $ is smooth (and hence $g$ is smooth too). Then via the purity theorem in \Cref{Thm:_stacky_Ambidexterity}, we can reduce ourselves to the same statement for the natural transformation $Ex_{!,!}$, witnessing the functoriality of the exceptional functor. But clearly $Ex_{!,!}$ is an equivalence and hence we are done.
\item [(viii)] The claim follows by adjunction from the statement in $(vii)$ that we just proved.
%Suppose $ p,q $ are separated of finite type and $ f,g $ are smooth. Similarly to all the other cases, we only need to show that:
%\[ x^*Ex_{!\#}: x^*f\epfs q\epf \longrightarrow x^*p\epf g\epfs \]
%\noindent is an equivalence. We can rewrite $ x^*Ex_{!\#}  $ as:
%\[ \tilde{f}\epfs \tilde{q}\epf w^*\stackrel{Ex_{!\#}w^*}{\longrightarrow} \tilde{p}\epf \tilde{g}\epfs w^* \]
%But $ Ex_{!\#}:\tilde{f}\epfs \tilde{q}\epf \stackrel{}{\longrightarrow} \tilde{p}\epf \tilde{g}\epfs   $ is already an equivalence by \cite[Theorem 2.4.26]{Cisinski-Deglise_Book}.

\end{enumerate}

\end{proof}

\end{appendices}

\bibliography{bibliography}

\begin{thebibliography}{{N}ee23}

\bibitem[AGV22]{AGV_Rigid6FF}
Joseph Ayoub, Martin Gallauer, and Alberto Vezzani.
\newblock The six-functor formalism for rigid analytic motives.
\newblock {\em Forum Math. Sigma}, 10:Paper No. e61, 182, 2022.

\bibitem[AI23]{Annala_Iwasa}
Toni Annala and Ryomei Iwasa.
\newblock Motivic spectra and universality of {$K$}-theory, 2023.

\bibitem[Alp]{Alper_Book}
Jarod Alper.
\newblock Stacks and moduli.

\bibitem[AP24]{Aranha-Pstragowski}
Dhyan Aranha and Piotr Pstr{\k{a}}gowski.
\newblock The intrinsic normal cone for artin stacks.
\newblock {\em Ann. Inst. Fourier (Grenoble)}, 74(1):71--120, May 2024.

\bibitem[Ayo07a]{Ayoub_6FF_vol1}
Joseph Ayoub.
\newblock Les six op\'{e}rations de {G}rothendieck et le formalisme des cycles
  \'{e}vanescents dans le monde motivique. {I}.
\newblock {\em Ast\'{e}risque}, (314):x+466 pp., 2007.

\bibitem[Ayo07b]{Ayoub_6FF_vol2}
Joseph Ayoub.
\newblock Les six op\'{e}rations de {G}rothendieck et le formalisme des cycles
  \'{e}vanescents dans le monde motivique. {II}.
\newblock {\em Ast\'{e}risque}, (315):vi+364 pp., 2007.

\bibitem[Ayo10]{Ayoub-Betti}
Joseph Ayoub.
\newblock Note sur les op{\'e}rations de grothendieck et la r{\'e}alisation de
  betti.
\newblock {\em J. Inst. Math. Jussieu}, 9(2):225--263, April 2010.

\bibitem[BH21]{Bachmann_Hoyois_Norms_MHT}
Tom Bachmann and Marc Hoyois.
\newblock Norms in motivic homotopy theory.
\newblock {\em Ast\'{e}risque}, 425, September 2021.

\bibitem[BLR90]{NeronModels}
Siegfried Bosch, Werner L\"{u}tkebohmert, and Michel Raynaud.
\newblock {\em N\'{e}ron models}, volume~21 of {\em Ergebnisse der Mathematik
  und ihrer Grenzgebiete (3) [Results in Mathematics and Related Areas (3)]}.
\newblock Springer-Verlag, Berlin, 1990.

\bibitem[CD19]{Cisinski_2019}
Denis-Charles Cisinski and Frédéric Déglise.
\newblock {\em Triangulated Categories of Mixed Motives}.
\newblock Springer International Publishing, 2019.

\bibitem[Cho24]{Chowdhury}
Chirantan Chowdhury.
\newblock Motivic {H}omotopy {T}heory of {A}lgebraic {S}tacks.
\newblock {\em Annals of K-Theory}, 9(1):1–22, May 2024.

\bibitem[D{\'e}g03]{Deglise2003_Thesis}
Fr{\'e}d{\'e}ric D{\'e}glise.
\newblock Modules de cycles et motifs mixtes.
\newblock {\em C. R. Math. Acad. Sci. Paris}, 336(1):41--46, January 2003.

\bibitem[DG22]{Drew-Gallauer}
Brad Drew and Martin Gallauer.
\newblock The universal six-functor formalism.
\newblock {\em Ann. K-Theory}, 7(4):599--649, December 2022.

\bibitem[DLP23]{Di_Lorenzo_Pirisi_Coh_Inv_Root}
Andrea Di~Lorenzo and Roberto Pirisi.
\newblock Cohomological invariants of root stacks and admissible double
  coverings.
\newblock {\em Canad. J. Math.}, 75(1):202--224, February 2023.

\bibitem[DS]{Neeraj-Felix-SHcl}
Neeraj Deshmukh and Felix Sefzig.
\newblock {T}he {M}orel-{V}oevodsky {C}onstruction on {A}lgebraic {S}tacks.

\bibitem[EK20]{Elmanto-Khan}
Elden Elmanto and Adeel~A Khan.
\newblock Perfection in motivic homotopy theory.
\newblock {\em Proc. Lond. Math. Soc. (3)}, 120(1):28--38, January 2020.

\bibitem[GK17]{Gepner-Kock2017}
David Gepner and Joachim Kock.
\newblock Univalence in locally cartesian closed $\infty$-categories.
\newblock {\em Forum Math.}, 29(3):617--652, May 2017.

\bibitem[Hoy15]{Hoyois_Quad_Lefschetz}
Marc Hoyois.
\newblock A quadratic refinement of the {Grothendieck--Lefschetz--Verdier}
  trace formula.
\newblock {\em Algebr. Geom. Topol.}, 14(6):3603--3658, January 2015.

\bibitem[Hoy17]{Hoyois_Equiv_Six_Op}
Marc Hoyois.
\newblock The six operations in equivariant motivic homotopy theory.
\newblock {\em Adv. Math.}, 305:197--279, 2017.

\bibitem[HRS23]{Hekking-Rydh-Savvas}
Jeroen Hekking, David Rydh, and Michail Savvas.
\newblock Stabilizer reduction for derived stacks and applications to
  sheaf-theoretic invariants, 2023.

\bibitem[Kha16]{Adeel_PhD}
Adeel Khan.
\newblock {\em Motivic homotopy theory in derived algebraic geometry}.
\newblock PhD thesis, Sep 2016.

\bibitem[Kha22]{Khan_K-Theory}
Adeel~A. Khan.
\newblock K-theory and g-theory of derived algebraic stacks.
\newblock {\em Japanese Journal of Mathematics}, 17(1):1--61, jan 2022.

\bibitem[Kha23]{Khan-Intro-DAG}
Adeel Khan.
\newblock An introduction to derived algebraic geometry.
\newblock \url{https://www.preschema.com/lecture-notes/2023-kias/dagkias.pdf},
  2023.

\bibitem[Knu71]{Knutson_AlgSp}
Donald Knutson.
\newblock {\em Algebraic Spaces}.
\newblock Lecture notes in mathematics. Springer, Berlin, Germany, 1971
  edition, January 1971.

\bibitem[KR24]{Khan-Ravi_Generalised_Coh_Stacks}
Adeel~A Khan and Charanya Ravi.
\newblock Generalized cohomology theories for algebraic stacks.
\newblock {\em Adv. Math. (N. Y.)}, 458(109975):109975, December 2024.

\bibitem[Lan21]{Land_introductionQC}
Markus Land.
\newblock {\em Introduction to infinity-categories}.
\newblock Compact Textbooks in Mathematics. Birkh\"{a}user/Springer, Cham,
  2021.

\bibitem[LMB00]{Laumon2000}
G{\'e}rard Laumon and Laurent Moret-Bailly.
\newblock {\em Champs alg{\'e}briques}.
\newblock Springer Berlin Heidelberg, Berlin, Heidelberg, 2000.

\bibitem[Lur09]{HTT}
Jacob Lurie.
\newblock {\em Higher topos theory}, volume 170 of {\em Annals of Mathematics
  Studies}.
\newblock Princeton University Press, Princeton, NJ, 2009.

\bibitem[Lur17]{HA}
Jacob Lurie.
\newblock Higher algebra.
\newblock \url{http://people.math.harvard.edu/~lurie/papers/HA.pdf}, 2017.

\bibitem[Lur18]{SAG}
Jacob Lurie.
\newblock Spectral algebraic geometry.
\newblock \url{https://www.math.ias.edu/~lurie/papers/SAG-rootfile.pdf}, 2018.

\bibitem[LZ17]{liu2017enhanced}
Yifeng Liu and Weizhe Zheng.
\newblock {Enhanced six operations and base change theorem for higher Artin
  stacks}.
\newblock \url{https://arxiv.org/pdf/1211.5948.pdf}, September 2017.

\bibitem[Man22]{mann2022padic}
Lucas Mann.
\newblock A $p$-adic 6-functor formalism in rigid-analytic geometry.
\newblock \url{https://arxiv.org/abs/2206.02022}, 2022.

\bibitem[Mil80]{MilneEtaleCoh}
J.~S. Milne.
\newblock {\em Etale Cohomology (PMS-33)}.
\newblock Princeton University Press, 1980.

\bibitem[MV99]{Morel-Voevodsky}
Fabien Morel and Vladimir Voevodsky.
\newblock {${\bf A}^1$}-homotopy theory of schemes.
\newblock {\em Publ. Math. Inst. Hautes \'{E}tudes Sci.}, (90):45--143, 1999.

\bibitem[{N}ee23]{deshmukh2023motivichomotopytypealgebraic}
{D}eshmukh {N}eeraj.
\newblock {O}n the {M}otivic {H}omotopy {T}ype of {A}lgebraic {S}tacks.
\newblock \url{https://arxiv.org/abs/2304.10631}, 2023.

\bibitem[Pir18]{Pirisi}
Roberto Pirisi.
\newblock Cohomological invariants of algebraic stacks.
\newblock {\em Trans. Amer. Math. Soc.}, 370(3):1885--1906, 2018.

\bibitem[Pst22]{Pstragowski_Synth}
Piotr Pstragowski.
\newblock Synthetic spectra and the cellular motivic category.
\newblock {\em Invent. Math.}, December 2022.

\bibitem[Rob14]{Robalothesis}
Marco Robalo.
\newblock Th\'eorie homotopique motivique des espaces noncommutatifs.
\newblock \url{https://webusers.imj-prg.fr/~marco.robalo/these.pdf}, 2014.

\bibitem[Rob15]{robalo2013noncommutative}
Marco Robalo.
\newblock {$K$}-theory and the bridge from motives to noncommutative motives.
\newblock {\em Adv. Math.}, 269:399--550, 2015.

\bibitem[Ryd11]{Rydh_Etale_Devissage}
David Rydh.
\newblock Étale dévissage, descent and pushouts of stacks.
\newblock {\em Journal of Algebra}, 331(1):194--223, 2011.

\bibitem[{Sta}21]{stacks-project}
The {Stacks project authors}.
\newblock The stacks project.
\newblock \url{https://stacks.math.columbia.edu}, 2021.

\bibitem[TV14]{HAG-II}
Bertrand To{\"e}n and Gabriele Vezzosi.
\newblock {\em Homotopical algebraic geometry {II}: Geometric stacks and
  applications}.
\newblock September 2014.

\end{thebibliography}
\end{document}